\documentclass[11pt,leqno]{article}
\usepackage[T1]{fontenc}
\usepackage{amsmath}%
\usepackage{amsthm}
\usepackage{amsxtra}%
\usepackage{amsfonts}%
\usepackage{amssymb}%
\usepackage[margin=1.25in]{geometry}
\usepackage{color,hyperref}
\usepackage{graphicx}
\usepackage[scientific-notation=true]{siunitx}
\usepackage{cite}
\usepackage{authblk}
\hypersetup{colorlinks,breaklinks,
             linkcolor=blue,urlcolor=blue,
             anchorcolor=blue,citecolor=blue}
\usepackage{color}
\usepackage{array}
\usepackage{booktabs}
\usepackage{enumerate}
\usepackage{dsfont}
\usepackage[normalem]{ulem}
\usepackage{enumitem}
\usepackage{multirow}
\usepackage{algorithm2e}
\usepackage{tikz}
\usepackage{caption}
\usepackage{subcaption}

\renewcommand{\bar}[1]{{\overline{#1}}}
\newcommand{\cof}{\text{cof}}

\newcommand{\X}{{\mathcal X}}
\newcommand{\R}{\mathbb{R}}

\newcommand{\M}{\mathcal{M}}

\newcommand{\vb}[1]{\mathbf{#1}}

\newcommand{\eps}{\varepsilon}
\renewcommand{\epsilon}{\varepsilon}

\renewcommand{\tilde}[1]{\widetilde{#1}}
\renewcommand{\phi}{\varphi}

\DeclareMathOperator*{\argmin}{argmin}
\DeclareMathOperator*{\argmax}{argmax}

\DeclareMathOperator{\sgn}{sgn}
\DeclareMathOperator{\dist}{dist}

\theoremstyle{plain}
\newtheorem{theorem}{Theorem}
\newtheorem*{theorem*}{Theorem}

\newtheorem{lemma}[theorem]{Lemma}
\newtheorem{proposition}[theorem]{Proposition}

\theoremstyle{definition}
\newtheorem{definition}[theorem]{Definition}

\theoremstyle{remark}

\newtheorem{remark}[theorem]{Remark}
\numberwithin{equation}{section}
\numberwithin{theorem}{section}
\numberwithin{figure}{section}
\numberwithin{table}{section}
\graphicspath{{images/}}

\begin{document}

\title{Monotone discretizations of levelset convex geometric PDEs}

\author{Jeff Calder\thanks{School of Mathematics, University of Minnesota. \href{mailto:jwcalder@umn.edu}{jwcalder@umn.edu}}, Wonjun Lee\thanks{Institute for Mathematics and its Applications, University of Minnesota. \href{mailto:lee01273@umn.edu}{lee01273@umn.edu}}}
\date{\today}
\maketitle
\begin{abstract}
% We construct monotone schemes for geometric PDEs that involve partial derivatives in directions orthogonal to the gradient. This includes a wide class of curvature motion PDEs, as well as a recent Hamilton-Jacobi equation for the Tukey depth. 

We introduce a novel algorithm that converges to level set convex viscosity solutions of high-dimensional Hamilton-Jacobi equations. The algorithm is applicable to a broad class of curvature motion PDEs, as well as a recently developed Hamilton-Jacobi equation for the Tukey depth, which is a statistical depth measure of data points. A main contribution of our work is a new \emph{monotone} scheme for approximating the \emph{direction} of the gradient, which allows for monotone discretizations of pure partial derivatives in the direction of, and orthogonal to, the gradient. We provide a convergence analysis of the algorithm on both regular Cartesian grids and unstructured point clouds in any dimension, and present numerical experiments that demonstrate the effectiveness of the algorithm in approximating solutions of the affine flow in two dimensions and the Tukey depth measure of high-dimensional datasets such as MNIST and FashionMNIST.
\end{abstract}

%\begin{itemize}
%\item \red Red \nc is just my comments.
%\item \blue Blue \nc is changes I've made that need to be reviewed.
%\end{itemize}

\section{Introduction}
\label{sec:intro}

The motion of curves or surfaces with normal velocity that depends on curvature has a wide range of applications in science, engineering, and mathematics. A short, and nowhere near complete list includes materials science~\cite{mullins1956two,allen1979microscopic}, fluid and bubble motion~\cite{chang1996level,sussman2000coupled}, image processing~\cite{alvarez1993axioms}, computer vision~\cite{chan2001active,mumford1989optimal}, stochastic control~\cite{soner2003stochastic}, and more recently, data science~\cite{calder2020convex}.

There is a wealth of literature on numerical schemes for approximating geometric motions, and one of the most successful and widely used algorithms is the level set method. This method was pioneered by Osher and Sethian~\cite{osher1988fronts} and implicitly represents the evolving curve or surface as the zero level set of a function $u(x,t)$. The algorithm then solves a level set PDE for the evolution of $u$. The implicit representation allows for topological changes in the surface, and has led to a rigorous notion of geometric flows past singularities by utilizing the machinery of viscosity solutions~\cite{evans1991motion}.

In current numerical practice, there is a significant discrepancy between the numerical schemes used and their theoretical counterparts. Specifically, there is no proof of convergence of the finite difference numerical solutions to the viscosity solution of the level set equation as the grid resolution approaches zero. The difficulty is that convergence proofs are only available for \textbf{monotone} schemes ~\cite{barles1991convergence} (refer to Definition \ref{def:monotone}), and the standard discretizations of curvature are not monotone.   

Several attempts have been made to address the lack of monotonicity in the literature. Merriman, Bence, and Osher~\cite{merriman1992diffusion} introduced a class of \emph{monotone} approximation schemes known as \emph{diffusion generated motion} or \emph{threshold dynamics}. The algorithms consist of two simple steps: (1) Convolution with a positive kernel (diffusion), and (2) thresholding. The original algorithm has been extended to a wide range of anisotropic curvature motions, as well as motions of networks (see \cite{elsey2018threshold} for recent results). Since the schemes are monotone, rigorous proofs of convergence to the viscosity solution are available~\cite{evans1993convergence,barles1995simple}. One drawback of threshold dynamics is that the algorithm may become ``stuck'' if the time step\footnote{The time step refers to the width of the convolution kernel.} is chosen too small, limiting the  accuracy~\cite{esedog2010diffusion}. This can be alleviated by using the signed distance function in place of characteristic functions (see, e.g., \cite{elsey2018threshold,elsey2011large}).  Oberman~\cite{oberman2004convergent} developed a wide-stencil monotone finite difference scheme for curvature motion based on a connection between the local median and curvature. Oberman's wide stencil approach has been extended to more general degenerate elliptic PDEs, including certain types of Hamilton-Jacobi and Monge-Amp\'ere equations \cite{oberman2008wide,benamou2010two,froese2013convergent,oberman2006convergent}, and more recently the affine flow \cite{oberman2018numerical}. In general, monotone schemes are less flexible than non-monotone ones, and in many cases they must be specifically designed for each application.

A noteworthy application of this class of curvature motion PDEs is the computation of data depth. Data depth can be seen as an extension of the notion order statistics to high-dimensional data sets. The \emph{depth} of a data point in a cluster is a notion of how close it is to the center, i.e., the mean or median of the data, with deeper points being more central and representative of the typical data point, and shallower points being identified as outliers. A definition of data depth leads naturally to a notion of high dimensional medians (i.e., the deepest points), and the study of robustness of medians to data perturbations is a central topic in the field of robust statistics. The Tukey, or half-space, depth \cite{tukey1975mathematics} is one of the seminal notions of data depth, and it has been extended to graphs \cite{small1997multidimensional} and metric spaces \cite{carrizosa1996characterization}. Other notions of data depth include convex hull peeling \cite{barnett1976ordering}, the Monge-Kantorovich depth \cite{chernozhukov2017monge}, non-dominated sorting \cite{calder2014}, and Pareto envelope peeling \cite{chepoi2010pareto,bou2021hamilton}.  Many notions of data depth have been connected to Hamilton-Jacobi and curvature motion equations in the large data continuum limit. It was shown in \cite{calder2014,calder2015PDE,cook2022rates} that non-dominated sorting has a Hamilton-Jacobi equation continuum limit. A related Hamilton-Jacobi equation continuum limit was established for Pareto-envelope peeling in \cite{bou2021hamilton}. In \cite{calder2020convex} it was shown that the continuum limit of convex hull peeling is a weighted version of affine invariant curvature motion (i.e., the affine flow). 

Recently, connections have also been made between Hamilton-Jacobi equations and Tukey depth \cite{molina2022tukey}. Tukey depth serves as a statistical measure of data depth and is defined given a data density function $\rho$ as follows:
\begin{align*}
T(x) := \inf_{|v|=1} \int_{(y-x)\cdot v \geq 0} \rho(y) dy.
\end{align*}
In other words, the depth $T(x)$ of a datapoint $x\in \R^n$ is the least amount of probability mass contained in any halfspace that contains $x$. The study \cite{molina2022tukey} showed that the Tukey depth function $T(x)$, under some reasonable assumptions on $\rho$ and its support $\Omega\subset \R^d$, is the viscosity solution of the nonstandard eikonal equation
\begin{equation}\label{eq:tukey_pde}
\begin{aligned}
|\nabla T(x)| &= \int_{{ (y-x)\cdot \nabla T(x)=0}} \rho(y) \, dS(y),&& \text{for } x \in \Omega,
\end{aligned}
\end{equation}
subject to the homogeneous Dirichlet boundary condition $u=0$ on $\partial \Omega$. The viscosity solution of \eqref{eq:tukey_pde} has convex level sets, i.e., it is a \emph{quasiconcave} function. The nonstandard dependence on $\nabla T$ on the right-hand side of Eq.~\eqref{eq:tukey_pde} poses a challenge in constructing a monotone, and hence provably convergent, numerical method. Currently, we are unaware of any existing numerical methods that can be used to solve \eqref{eq:tukey_pde} with provable convergence guarantees. Let us also mention that recent works, some inspired by \cite{molina2022tukey}, have considered using a more standard eikonal equation of the form $|\nabla T| = \phi(\rho)$ for data depth (see \cite{molina2022eikonal} and \cite{calder2022boundary}). These standard eikonal equations can be solved with the Fast Marching Method~\cite{sethian1999fast}, which is known for its speed and efficiency in solving the eikonal equation. In addition, a recent study \cite{JMLR:v23:22-0293} considered a family of graph $p$-eikonal equations, and demonstrated its applications in applications to data depth and semi-supervised learning.

The lack of numerical methods with rigorous guarantees for solving \eqref{eq:tukey_pde} was one of the main motivations for this work. Notice that the right hand side of \eqref{eq:tukey_pde} depends only the \emph{direction} of $\nabla T$, and not on its magnitude. The same types of dependencies arise in curvature motion Hamilton-Jacobi equations, where one can view the various principal curvatures arising in the front propagation speed as pure second derivatives in directions orthogonal to the gradient. In this work, we develop a novel wide-stencil finite-difference technique for discretizing the \emph{direction} of the gradient that works for the Tukey depth equation \eqref{eq:tukey_pde}, as well as Hamilton-Jacobi equations with curvature dependent speeds. Our current work is focused on the setting of monotone front evolution in which the level sets of the solution are convex, but we expect the methods are more general and this constraint can be relaxed in future work. Since our scheme is monotone, we are able to use the Barles-Souganidis framework \cite{barles1991convergence} to prove convergence to the viscosity solution. An interesting feature of our work is that our proposed scheme is not dependent on any grid structure, and it can be easily applied on unstructured, possibly high dimensional, point clouds. While the accuracy of the schemes will suffer from the curse of dimensionality, the computational cost depends only on the number of datapoints and is largely insensitive to dimension. As an application, we present results of solving the Tukey depth PDE \eqref{eq:tukey_pde} on high dimensional image data sets, including MNIST and FashionMNIST.

\subsection{Outline}

This paper is organized as follows. In the following sections, we describe a new technique for constructing monotone finite difference schemes for discretizing the direction of the gradient. We begin in Section~\ref{sec:background} by reviewing the definitions of quasiconcave functions, viscosity solutions, and monotone schemes. In Section~\ref{sec:numer-methods}, we propose a new monotone and consistent numerical scheme for computing viscosity solutions of curvature-driven PDEs and prove the convergence of the scheme on general point clouds in $\mathbb{R}^d$, with an arbitrary dimension $d$. Section~\ref{sec:appl} presents several applications of using the proposed numerical methods to compute solutions of the Tukey depth eikonal equation and mean curvature motion PDEs. Finally, in Section~\ref{sec:exp}, we present numerical examples of using the proposed scheme to solve various eikonal equations in general point clouds settings in dimensions ranging from $d=2$ to $d=784$.

\section{Background}\label{sec:background}

In this paper, we are interested in a general class of second order Hamilton-Jacobi equations of the form
\begin{equation}\label{eq:first-hje}
\left\{
    \begin{aligned}
        H(\nabla^2 u, \nabla u, u, x) &= 0, && x \in \Omega\\
        u(x) &= g(x), && x \in \partial \Omega,
    \end{aligned}
\right.
\end{equation}
where $\Omega \subset \mathbb{R}^d$ is an open and bounded domain, $\partial \Omega$ is a boundary of $\Omega$, $H:\mathbb{R}^{d\times d}_{sym} \times\mathbb{R}^d \times \mathbb{R} \times \Omega \rightarrow \mathbb{R}$, $g:\Omega\rightarrow \mathbb{R}$, and $u:\Omega \rightarrow \mathbb{R}$, with $\nabla u$ denoting the gradient of $u$ and $\nabla^2 u$ denoting the Hessian. In particular, we are interested in the setting where the solution $u$ is \emph{quasiconcave}, which means the super level set $\{u > t\}$ is convex for all $t \in \mathbb{R}$.

\begin{figure}[t!]
\centering
\tikzset{every picture/.style={line width=0.75pt}} %set default line width to 0.75pt        
\begin{tikzpicture}[x=0.75pt,y=0.75pt,yscale=-1,xscale=1]
%uncomment if require: \path (0,393); %set diagram left start at 0, and has height of 393
%Straight Lines [id:da9884306628121746] 
\draw    (20,120.59) -- (189.41,120.59) ;
%Straight Lines [id:da7553428344385347] 
\draw    (104.71,50) -- (175.29,120.59) ;
%Straight Lines [id:da21442019416381808] 
\draw    (104.71,50) -- (34.12,120.59) ;
%Straight Lines [id:da5626466431612559] 
\draw    (207.06,120) -- (376.47,120) ;
%Straight Lines [id:da9041489841736126] 
\draw    (312.94,70.59) -- (362.35,120) ;
%Straight Lines [id:da22116220162191758] 
\draw    (270.59,70.59) -- (221.18,120) ;
%Straight Lines [id:da8176178367316211] 
\draw    (270.59,70.59) -- (291.76,91.76) ;
%Straight Lines [id:da2179975496866393] 
\draw    (291.76,91.76) -- (312.94,70.59) ;
%Straight Lines [id:da7031834609667492] 
\draw    (390.59,120) -- (560,120) ;
%Straight Lines [id:da8636485094949135] 
\draw    (517.65,91.76) -- (545.88,120) ;
%Straight Lines [id:da4868679633976797] 
\draw    (432.94,91.76) -- (404.71,120) ;
%Straight Lines [id:da4168056795640137] 
\draw    (432.94,91.76) -- (447.06,105.88) ;
%Straight Lines [id:da5496549453152181] 
\draw    (503.53,105.88) -- (517.65,91.76) ;
%Straight Lines [id:da6994668207814059] 
\draw    (447.06,105.88) -- (461.18,91.76) ;
%Straight Lines [id:da3834906892258466] 
\draw    (461.18,91.76) -- (475.29,105.88) ;
%Straight Lines [id:da17888155718562104] 
\draw    (475.29,105.88) -- (489.41,91.76) ;
%Straight Lines [id:da9901031323024376] 
\draw    (489.41,91.76) -- (503.53,105.88) ;
\end{tikzpicture}
\caption{Non-unique solutions of the 1D eikonal equation $|u'| = 1$ with Dirichlet boundary conditions $u(0)=u(1)=0$.}
\label{fig:1d-example}
\end{figure}
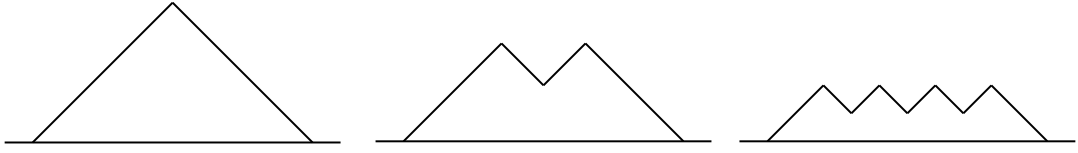

This class of equations usually does not admit classical solutions, i.e., solutions that are continuously differentiable so that the equation is satisfied classically at each $x\in \Omega$. See Figure~\ref{fig:1d-example} for a simple 1D example. Due to the fully nonlinear nature of the equation, the notion of weak solutions using test functions and integration by parts is not applicable. For equations of the form \eqref{eq:first-hje} that satisfy some basic ellipticity and monotonicity conditions, the notion of \emph{viscosity solution}~\cite{crandall1983viscosity,crandall1984some} identifies the physically correct solution for broad ranges of applications, and has proven to be an extremely useful tool in the study of nonlinear PDEs. In this section, we review definitions of quasiconcavity, viscosity solutions, and the convergence for numerical schemes for viscosity solutions.

\subsection{Quasi-concavity}

Let us introduce the definition of quasiconcave functions and their properties.
\begin{definition}
    A function $u:\Omega\to \R$ is \emph{quasiconcave} if
    \[u(\lambda x + (1-\lambda)y) \geq \min\big( u(x), u(y) \big)\] 
    for all $x,y \in \Omega$ and all $0< \lambda <1$. A function $u$ is \emph{strictly quasiconcave} if the inequality is strict. A function $u$ is \emph{locally quasiconcave} at $x \in \Omega$ (resp. \emph{locally strictly quasiconcave}) if $u$ satisfies the inequality (resp. strict inequality) in some neighborhood of $x \in \mathcal{O} \subset \Omega$ .
\end{definition}

%\red I added the assumption that $u$ is $C^1$ (and $C^2$) below.  It may be that (iii) does not require it though. I believe they can also be relaxed as long as you interpret the inequalities in the viscosity sense (see \cite{barron2013quasiconvex}). Do we need this for viscosity solutions (which are not regular) or for test functions? \nc
\begin{lemma}\label{lem:quasiconcave}
    Given $u\in C^1(\Omega)$, the following are equivalent.
    \begin{enumerate}[label=(\roman*)]
        \item $u$ is quasiconcave.
        \item For all $x,y \in \Omega$, \[(y-x) \cdot \nabla u(x) \leq 0 \implies u(y) \leq u(x).\]
        \item For all $x \in \Omega$, there exists a nonzero $p\in\mathbb{R}^d$ such that for all $y \in \Omega$, \[(y-x) \cdot p \leq 0 \implies u(y) \leq u(x).\]
    \end{enumerate}
    % for all $q \in \mathbb{R}^d$, \[q \cdot \nabla \phi(x) \leq 0 \implies u(x+q) < u(x)$.
    % \[ \forall x,y \in \Omega, \quad u(y) \geq u(x) \implies (y-x) \cdot \nabla u(x) > 0. \]
\end{lemma}

\begin{lemma}\label{lem:second-order-quasiconcave}
    If $u\in C^2(\Omega)$ is quasiconcave then for all $x,y \in \Omega$, \[(y-x) \cdot \nabla u(x) = 0 \implies (y-x) \cdot \nabla^2 u(x)(y-x)\leq 0.\]
\end{lemma}

\begin{lemma}\label{lem:neighbor-quasiconcave}
    Given a function $u\in C^2(\Omega)$ the following are equivalent.
    \begin{enumerate}[label=(\roman*)]
        \item $u$ is strictly quasiconcave.
        \item For all $x,y \in \Omega$, \[(y-x) \cdot \nabla u(x) \leq 0 \implies u(y) < u(x).\]
        \item For all $x \in \Omega$, there exists a nonzero $p\in\mathbb{R}^d$ such that for all $y \in \Omega$, \[(y-x) \cdot p \leq 0 \implies u(y) < u(x).\]
        \item For all $x,y \in \Omega$ and $x\neq y$, \[(y-x) \cdot \nabla u(x) = 0 \implies (y-x) \cdot \nabla^2 u(x)(y-x)<0.\] \label{item:lem-neighbor}
    \end{enumerate}
    % for all $q \in \mathbb{R}^d$, \[q \cdot \nabla \phi(x) \leq 0 \implies u(x+q) < u(x)$.
    % \[ \forall x,y \in \Omega, \quad u(y) \geq u(x) \implies (y-x) \cdot \nabla u(x) > 0. \]
\end{lemma}

Note that the second order condition is a necessary condition for the quasiconcavity in Lemma~\ref{lem:second-order-quasiconcave}  but it is a necessary and sufficient condition for the strict quasiconcavity in Lemma~\ref{lem:neighbor-quasiconcave}~\ref{item:lem-neighbor}. We refer the reader to~\cite{barron2013quasiconvex} and~\cite{boyd2004convex} for more details on quasiconcave functions.% \red Are all the results above proved in \cite{boyd2004convex}? There is also \cite{barron2013quasiconvex} where these types of results are established in the more general viscosity sense. \nc

\subsection{Viscosity solutions}
 Recall the definitions of upper and lower semicontinuous functions.

\begin{definition}
    A function $f: \mathcal{O} \rightarrow \mathbb{R}$ is upper (resp. lower) semicontinuous if
    \[
        \limsup_{\substack{y \rightarrow x \\ y \in \mathcal{O} } } f(y) \leq f(x) \quad ({\textrm{resp. }} \liminf_{y\rightarrow x} f(y) \geq f(x))
    \]
    for all $x \in \Omega$.
\end{definition}

\begin{definition}
    Given a function $f: \mathcal{O} \rightarrow \mathbb{R}$, the upper (resp. lower) semicontinuous envelop of $f$ is
    \[
        f^*(x) = \limsup_{\substack{y\rightarrow x \\ y \in \mathcal{O}}} f(y) \quad (\text{resp. } f_*(x) = \liminf_{\substack{y\rightarrow x \\ y \in \mathcal{O}}} f(y)).
    \]
\end{definition}

We present the definition of viscosity solutions of~\eqref{eq:first-hje} given upper or lower semicontinuous functions.

\begin{definition}\label{def:viscosity-solution}
    An upper semicontinuous (resp. lower semicontinuous) function $u:\Omega \rightarrow \mathbb{R}$ is a \emph{viscosity subsolution} (resp. \emph{supersolution}) of~\eqref{eq:first-hje} if for every $x\in \Omega$ and every smooth test function $\phi \in C^\infty(\mathbb{R}^d)$ such that $u-\phi$ has a local maximum at $x$,
    \[
        \begin{cases}
            H_*(\nabla^2 \phi, \nabla \phi, u, x) \leq 0. & \text{if } x\in\Omega\\
            \min\left(H_*(\nabla^2 \phi, \nabla \phi, u, x) , u(x)-g(x)\right) \leq 0 & \text{if } x\in\partial \Omega
        \end{cases}
    \]
    \noindent (respectively,
    \[
        \begin{cases}
            H^*(\nabla^2 \phi, \nabla \phi, u, x) \geq 0) & \text{if } x\in\Omega\\
            \max\left(H^*(\nabla^2 \phi, \nabla \phi, u, x) , u(x)-g(x)\right) \geq 0 & \text{if } x\in\partial \Omega)
        \end{cases}
    \]
    where $g:\partial\Omega \rightarrow \mathbb{R}$ is continuous.
    \noindent If $u$ is both a viscosity subsolution and a viscosity supersolution, then we call $u$ a \emph{viscosity solution} of~\eqref{eq:first-hje}. We say that the boundary condition in~\eqref{eq:first-hje} hold in the \emph{weak viscosity sense}.

\end{definition}
We note that the upper and lower semicontinuous envelopes $H^*$ and $H_*$ are computed with respect to all of the variables that $H$ depends on. We refer the reader to~\cite{crandall1992user,calder2018lecture} for more details on viscosity solutions. In particular, we treat the boundary conditions in the viscosity sense, as in \cite[Chapter 7]{crandall1992user}.

\subsection{Monotone schemes}

We provide a review of the definitions of monotone schemes used to approximate viscosity solutions based on the Barles-Souganidis framework~\cite{barles1991convergence}. Our finite difference schemes for~\eqref{eq:first-hje} are presented in the form
\[\left\{\begin{aligned}
    S_h(u_h,u_h(x),x) &= 0, &&\text{for } \ x \in \mathcal{X}_{n} \backslash \Gamma_n,\\
    u_h(x) &= g(x), &&\text{for } \ x \in \Gamma_n,
\end{aligned}\right.\]
where $\mathcal{X}_{n} \subset \bar\Omega$ is a set of points with spatial resolution $h$, $\Gamma_n \subset \mathcal{X}_n$ is a set of boundary nodes, $u_h: \mathcal{X}_{n} \to \R$ is the numerical solution, and $S_h$ is the scheme. The first argument of $S_h$ represents the dependence of the scheme on the values of $u_h$ at neighboring points, while the second represents the dependence of the scheme on the value of $u_h$ at the current point $x$. To ensure convergence, the Barles-Souganidis framework provides necessary properties that the scheme must satisfy. In this context, we review the definitions that are required for the convergence of the scheme.

\begin{definition}\label{def:monotone}
A scheme $S_h$ is \emph{monotone} if for all $t \in \R$, $x \in \mathcal{X}_{n}$, and $u,v:\mathcal{X}_{n} \to \R$
\[u\leq v \implies S_h(u,t,x) \geq S_h(v,t,x).\]
\end{definition}

\begin{definition}\label{def:consistent}
    A scheme $S_h$ is \emph{consistent} if for all $x \in \bar\Omega$ and $\phi \in C^\infty(\mathbb{R}^n)$
    \[
        \limsup_{\substack{\gamma \rightarrow 0\\h\rightarrow 0^+\\ y\rightarrow x}} S_h(\phi+\gamma,\phi(y) + \gamma,y) \leq H^*(\nabla^2 \phi,\nabla\phi,\phi,x).
    \]    
    and
    \[
        \liminf_{\substack{\gamma \rightarrow 0\\h\rightarrow 0^+\\ y\rightarrow x}} S_h(\phi+\gamma,\phi(y) + \gamma,y) \geq H_*(\nabla^2 \phi,\nabla\phi,\phi,x).
    \]    
\end{definition}

\begin{definition}\label{def:stable}
    A scheme  $S_h$ is \emph{stable} if the solution of the scheme $u_h$ satisfies
    \[
        \sup_{h>0} \sup_{x\in \mathcal{X}_{n}} |u_h(x)| \leq C
    \]    
    for some positive constant $C > 0$.
\end{definition}

\begin{definition}\label{def:strong-unique}
    The PDE~\eqref{eq:first-hje} satisfies the \emph{strong uniqueness} if  $u\leq v$ on $\bar\Omega$ for every viscosity subsolution $u$ and every viscosity supersolution $v$.
\end{definition}

When the PDE satisfies the comparison principle, in the sense of strong uniqueness in Definition \ref{def:strong-unique}, and the scheme satisfies monotonicity, consistency, and stability, one can show that the solution of the scheme converges uniformly to a unique viscosity solution based on Barles-Souganidis framework (refer to Theorem~\ref{thm:convergence}). We remark that the notion of strong uniqueness is different from a standard comparison principle for viscosity sub and supersolutions due to how Definition \ref{def:viscosity-solution} handles the boundary conditions (which is often called boundary conditions in the viscosity sense, see \cite[Chapter 7]{crandall1992user}).

\section{Numerical methods}\label{sec:numer-methods}

In this section, we introduce our novel monotone numerical scheme for computing quasiconcave viscosity solutions of Hamilton-Jacobi equations. Our scheme can be applied on general point clouds of arbitrary dimensions, provided they satisfy some reasonable properties. This allows the methods to be applied in graph settings, with various graph structures such as $\epsilon$-graphs or $k$-nearest neighbor graphs. Due to the monotonicity of the scheme, the method enjoys strong stability and convergence guarantees.

\subsection{Notation}
Before proceeding, let us fix some notation. Let $\Omega \subset \mathbb{R}^d$ be an open bounded domain. Define a set of points 
\begin{equation*}
    \mathcal{X}_n = \{x_1, x_2, \cdots, x_n\} \subset \bar\Omega,
\end{equation*}
a set of boundary points
\begin{equation*}
    \Gamma_n \subset \mathcal{X}_n,
\end{equation*}
and a spatial resolution
\[
    h := \max_{x\in\mathcal{X}_{n}} \min_{y\in\mathcal{X}_{n}} |x-y|.
\]
%\red
%I would suggest we just define the set of points as
%\begin{equation*}
%    \mathcal{X}_n = \{x_1, x_2, \cdots, x_n\} \subset \Omega,
%\end{equation*}
%and then define the spatial resolution of the point cloud as
%\[
%    h:=\max_{y\in\mathcal{X}_{n}}\min_{x\in\mathcal{X}_{n}}  |x-y|.
%\]
%Note I think the min and max should be switched (please check).\nc

For each $x\in \X_n$, we define a set of neighboring points $N_h(x)\subset \X_n$, and we assume there exists $0 < \delta < R$ such that 
\[
    % N_{h}(x) = \mathcal{X}_{n} \cap B(x,R) \backslash \{x\},
    N_{h}(x) \subset B(x,R) \backslash B(x,\delta) \ \ \text{for all } x\in \X_n.
\]
It will be important later on to take $R,\delta = O(h)$. 
%and
%\begin{equation}\label{eq:h-delta-relation}
%     \delta = c h, \quad \epsilon = C h
% \end{equation}
% for some constants $0 < c < C$.
% \noindent Note that $N_{h}$ may not be a Euclidean ball. In case of $k$-NN graphs, $N_{h}$ could be a non-symmetric irregular shape. 
% a vertices set \red I wouldn't say these are vertices, maybe "displacement vectors", or "relative stencil".\nc
Define a set of displacement vectors
\[
    V_{h}(x) = \left\{y-x:\, y \in N_{h}(x)  \right\}
\]
that denotes the vectors pointing from $x$ to each neighbor,
and the local directional resolution at $x \in \mathcal{X}_{n}$ 
\[
    d\theta(x) = \max_{|p| = 1} \min_{q \in V_{h}(x)} w(p,q)
\]
where $w(p,q) = \arccos\left(\frac{p\cdot q}{|p| |q|} \right)$. Define the global directional resolution 
\begin{equation*}
    d\theta  := \max_{x \in \mathcal{X}_{n}} d\theta(x).
\end{equation*}
% \red Is this assumption always necessary?  \nc
The following lemma describes the geometric properties of point clouds in $\mathbb{R}^d$, and will be used in the main theorems. The visual representations can be found in Figure~\ref{fig:lem-p-q}.
\begin{lemma}\label{lem:p-q-perp}
    Let $ 0 \leq \theta_1\leq \pi$ and $0 \leq \theta_2 \leq \pi$ be nonnegative constants, and $x$, $p$, $q$ be unit vectors such that $w(x, p) = \theta_1$. 
     % If $w(x,q) \leq \theta_2$ and $\theta_1 < \theta_2$ or if $w(x,q) \geq \theta_2$ and $\theta_1 > \theta_2$, then
    % \[
    %     w(x, q) \leq \theta_2 - \theta_1.
    % \]
    % The equality is attained if and only if
    % \begin{enumerate}[{\normalfont (i)}]
    \begin{enumerate}[label=(\roman*)]
        \item If $\theta_1 < \theta_2$ and $w(p,q) \geq \theta_2$, then
        \[
            x\cdot q \leq \cos(\theta_2 - \theta_1).
        \]
        The equality is attained if and only if $w(p,q) = \theta_2$ and
        \[
            p = \frac{x - \big(\sin \theta_1/\sin \theta_2\big) q}{|x - \big(\sin \theta_1/\sin \theta_2\big) q|}.
        \]

        \item If $\theta_1 > \theta_2$ and $w(p,q) \leq \theta_2$, then
        \[
            \cos(\theta_1 + \theta_2) \leq x\cdot q \leq \cos(\theta_1 - \theta_2).
        \]
        The left equality is attained if and only if $w(p,q) = \theta_2$ and
        \[
            p = \frac{q + \big(\sin \theta_2/\sin \theta_1\big) x}{|q + \big(\sin \theta_2/\sin \theta_1\big) x|}.  
        \]
        The right equality is attained if and only if $w(p,q) = \theta_2$ and
        \[
            p = \frac{q - \big(\sin \theta_2/\sin \theta_1\big) x}{|q - \big(\sin \theta_2/\sin \theta_1\big) x|}.
        \]
    \end{enumerate}
\end{lemma}
\begin{proof}
    Assume $\theta_1 < \theta_2$ and $w(p,q) \geq \theta_2$ and let $\lambda$ be an arbitrary positive constant. Then
    \begin{equation}\label{eq:lem4-first-ineq}
            p \cdot (x - \lambda q) \leq |x - \lambda q| = \sqrt{1 + \lambda^2 - 2 \lambda x \cdot q}.        
    \end{equation}
    By the assumption, $p \cdot (x - \lambda q) \geq \cos \theta_1 - \lambda \cos \theta_2$. Thus, by squaring both sides, we get
    \begin{align*}
        \cos^2 \theta_1 + \lambda^2 \cos^2 \theta_2 - 2\lambda \cos \theta_1 \cos \theta_2 \leq 1 + \lambda^2 - 2 \lambda x \cdot q.
    \end{align*}
    Using the equality $\cos^2 \theta + \sin^2 \theta = 1$,
    \begin{align*}
        2 \lambda x\cdot q &\leq \sin^2 \theta_1 + \lambda^2 \sin^2 \theta_2 + 2\lambda \cos \theta_1 \cos \theta_2 \\
        &= 2\lambda \cos(\theta_2 - \theta_1) + (\sin\theta_1 - \lambda \sin\theta_2)^2.
    \end{align*}
    Since $\lambda$ is an arbitrary number, we may choose $\lambda = \sin\theta_1 / \sin\theta_2$. Thus,
    \[
        x\cdot q \leq \cos(\theta_2 - \theta_1).
    \]
    From~\eqref{eq:lem4-first-ineq}, the equality is attained if and only if  $w(p,q) = \theta_2$ and $p = \frac{x-\lambda q}{|x-\lambda q|}$.

    For the second part of the lemma, assume $\theta_1 > \theta_2$ and $w(p,q) \leq \theta_2$ and let $\lambda$ be an arbitrary constant. Similar to the proof of the first part, consider
    \begin{align*}
        \cos \theta_2 - \lambda \cos \theta_1 \leq p \cdot (q - \lambda x) \leq |q - \lambda x|
    \end{align*}
    where the first inequality comes from the assumption and $\sgn$ is a sign function. By squaring both sides and rearranging terms,
    \begin{align*}
        2   \lambda x\cdot q &\leq  \sin^2 \theta_2 + \lambda^2 \sin^2 \theta_1 +  2  \lambda \cos \theta_1 \cos \theta_2.
    \end{align*}
    If $\lambda > 0$, then
    \begin{align*}
        2   \lambda x\cdot q &\leq  2\lambda \cos (\theta_1 - \theta_2) + (\sin \theta_2 - \lambda \sin \theta_1)^2.
    \end{align*}
    By choosing $\lambda = \sin\theta_2/\sin\theta_1$,
    \begin{equation}\label{eq:lem4-2nd-1}
        x\cdot q \leq \cos (\theta_1 - \theta_2).
    \end{equation}
    If $\lambda < 0$, then
    \begin{align*}
        - 2   \lambda x\cdot q &\geq  - \sin^2 \theta_2 - \lambda^2 \sin^2 \theta_1 -  2  \lambda \cos \theta_1 \cos \theta_2\\
        &= -2\lambda \cos(\theta_1 + \theta_2) - (\sin\theta_2 + \lambda \sin\theta_1)^2.
    \end{align*}
    By choosing $\lambda = - \sin\theta_2/\sin\theta_1$,
    \begin{equation}\label{eq:lem4-2nd-2}
        x\cdot q \geq \cos (\theta_1 + \theta_2).
    \end{equation}
    The equalities in~\eqref{eq:lem4-2nd-1} and~\eqref{eq:lem4-2nd-2} are attained if and only if  $w(p,q) = \theta_2$ and $p = \frac{q-\lambda x}{|q-\lambda x|}$. This concludes the proof.
\end{proof}

\begin{figure}[ht!]
    \centering
    \subfloat[{Lemma~\ref{lem:p-q-perp}: Case \textit{(i)}}]{

\tikzset{every picture/.style={line width=0.75pt}} %set default line width to 0.75pt        

\begin{tikzpicture}[x=0.75pt,y=0.75pt,yscale=-0.8,xscale=0.8]

%Straight Lines [id:da07904863216045266] 
\draw    (138.83,229.76) -- (138.83,74.11) ;
\draw [shift={(138.83,72.11)}, rotate = 90] [color={rgb, 255:red, 0; green, 0; blue, 0 }  ][line width=0.75]    (10.93,-3.29) .. controls (6.95,-1.4) and (3.31,-0.3) .. (0,0) .. controls (3.31,0.3) and (6.95,1.4) .. (10.93,3.29)   ;
%Straight Lines [id:da5187092360126608] 
\draw    (138.83,229.76) -- (232.22,105.24) ;
\draw [shift={(233.42,103.64)}, rotate = 126.87] [color={rgb, 255:red, 0; green, 0; blue, 0 }  ][line width=0.75]    (10.93,-3.29) .. controls (6.95,-1.4) and (3.31,-0.3) .. (0,0) .. controls (3.31,0.3) and (6.95,1.4) .. (10.93,3.29)   ;
%Straight Lines [id:da6413244897447126] 
\draw    (138.83,229.76) -- (76.58,89.7) ;
\draw [shift={(75.77,87.87)}, rotate = 66.04] [color={rgb, 255:red, 0; green, 0; blue, 0 }  ][line width=0.75]    (10.93,-3.29) .. controls (6.95,-1.4) and (3.31,-0.3) .. (0,0) .. controls (3.31,0.3) and (6.95,1.4) .. (10.93,3.29)   ;
%Straight Lines [id:da5836855246637641] 
\draw    (138.83,229.76) -- (298.03,259.63) ;
\draw [shift={(300,260)}, rotate = 190.63] [color={rgb, 255:red, 0; green, 0; blue, 0 }  ][line width=0.75]    (10.93,-3.29) .. controls (6.95,-1.4) and (3.31,-0.3) .. (0,0) .. controls (3.31,0.3) and (6.95,1.4) .. (10.93,3.29)   ;
%Shape: Ellipse [id:dp3290041095948738] 
\draw  [dash pattern={on 4.5pt off 4.5pt}][line width=0.75]  (78.99,83.71) .. controls (89.65,69.95) and (147.39,96.81) .. (207.95,143.72) .. controls (268.51,190.63) and (308.96,239.82) .. (298.29,253.58) .. controls (287.63,267.35) and (229.89,240.48) .. (169.33,193.58) .. controls (108.77,146.67) and (68.32,97.48) .. (78.99,83.71) -- cycle ;
%Curve Lines [id:da7594882360919474] 
\draw    (138.83,150.93) .. controls (166.68,146.2) and (171.93,150.41) .. (186.12,166.7) ;
%Curve Lines [id:da5489136671262639] 
\draw    (194.33,156.33) .. controls (225.67,175.67) and (232.33,215.67) .. (219.41,244.88) ;

% Text Node
\draw (134.11,50.21) node [anchor=north west][inner sep=0.75pt]    {$x$};
% Text Node
\draw (238.16,81.56) node [anchor=north west][inner sep=0.75pt]    {$p$};
% Text Node
\draw (66.08,64.7) node [anchor=north west][inner sep=0.75pt]    {$q$};
% Text Node
\draw (161.64,127.86) node [anchor=north west][inner sep=0.75pt]    {$\theta _{1}$};
% Text Node
\draw (229.33,187.07) node [anchor=north west][inner sep=0.75pt]    {$\theta _{2}$};

\end{tikzpicture}
        \label{fig:lem-case-1}
        }
    % \hspace{1cm}
    \hfil
    \subfloat[Lemma~\ref{lem:p-q-perp}: Case \textit{(ii)}]{

\tikzset{every picture/.style={line width=0.75pt}} %set default line width to 0.75pt        

\begin{tikzpicture}[x=0.75pt,y=0.75pt,yscale=-0.8,xscale=0.8]

%Straight Lines [id:da07904863216045266] 
\draw    (138.83,229.76) -- (138.83,74.11) ;
\draw [shift={(138.83,72.11)}, rotate = 90] [color={rgb, 255:red, 0; green, 0; blue, 0 }  ][line width=0.75]    (10.93,-3.29) .. controls (6.95,-1.4) and (3.31,-0.3) .. (0,0) .. controls (3.31,0.3) and (6.95,1.4) .. (10.93,3.29)   ;
%Straight Lines [id:da5187092360126608] 
\draw    (138.83,229.76) -- (273.48,114.96) ;
\draw [shift={(275,113.67)}, rotate = 139.55] [color={rgb, 255:red, 0; green, 0; blue, 0 }  ][line width=0.75]    (10.93,-3.29) .. controls (6.95,-1.4) and (3.31,-0.3) .. (0,0) .. controls (3.31,0.3) and (6.95,1.4) .. (10.93,3.29)   ;
%Straight Lines [id:da6413244897447126] 
\draw    (138.83,229.76) -- (191.03,75.56) ;
\draw [shift={(191.67,73.67)}, rotate = 108.7] [color={rgb, 255:red, 0; green, 0; blue, 0 }  ][line width=0.75]    (10.93,-3.29) .. controls (6.95,-1.4) and (3.31,-0.3) .. (0,0) .. controls (3.31,0.3) and (6.95,1.4) .. (10.93,3.29)   ;
%Straight Lines [id:da5836855246637641] 
\draw    (138.83,229.76) -- (299.31,206.3) ;
\draw [shift={(301.29,206.01)}, rotate = 171.68] [color={rgb, 255:red, 0; green, 0; blue, 0 }  ][line width=0.75]    (10.93,-3.29) .. controls (6.95,-1.4) and (3.31,-0.3) .. (0,0) .. controls (3.31,0.3) and (6.95,1.4) .. (10.93,3.29)   ;
%Curve Lines [id:da7594882360919474] 
\draw    (138.83,150.93) .. controls (166.68,146.2) and (185.67,160.33) .. (193.67,183) ;
%Curve Lines [id:da5489136671262639] 
\draw    (206.91,171.71) .. controls (223.67,181) and (229,197.67) .. (225,216.33) ;
%Shape: Ellipse [id:dp5736912189751872] 
\draw  [dash pattern={on 4.5pt off 4.5pt}] (191.67,73.67) .. controls (200.17,66.62) and (231.01,89.78) .. (260.54,125.41) .. controls (290.08,161.03) and (307.13,195.63) .. (298.62,202.68) .. controls (290.12,209.73) and (259.28,186.56) .. (229.75,150.94) .. controls (200.21,115.31) and (183.16,80.72) .. (191.67,73.67) -- cycle ;

% Text Node
\draw (134.11,50.21) node [anchor=north west][inner sep=0.75pt]    {$x$};
% Text Node
\draw (279.5,98.89) node [anchor=north west][inner sep=0.75pt]    {$p$};
% Text Node
\draw (174.98,136.52) node [anchor=north west][inner sep=0.75pt]    {$\theta _{1}$};
% Text Node
\draw (229.33,178.4) node [anchor=north west][inner sep=0.75pt]    {$\theta _{2}$};
% Text Node
\draw (308,191.73) node [anchor=north west][inner sep=0.75pt]    {$q$};

\end{tikzpicture}
    \label{fig:lem-case-2}
    }
    \caption{Visual representations of Lemma~\ref{lem:p-q-perp}.}
    \label{fig:lem-p-q}
    \end{figure}
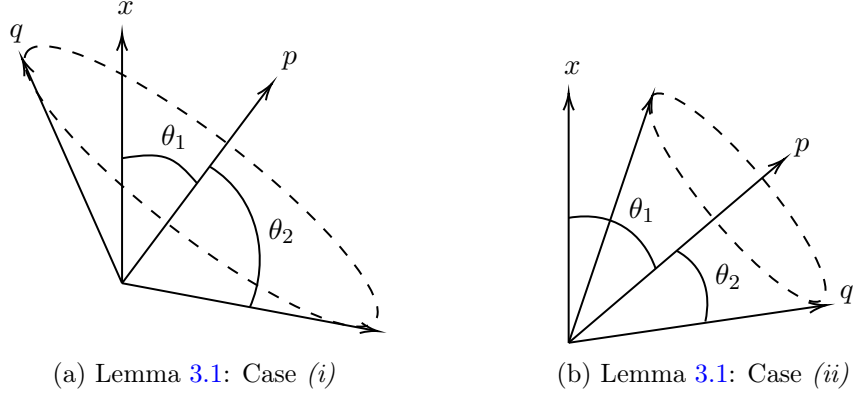

\subsection{Wide stencil schemes}
% Let $h>0$ denote the grid resolution, and let $\mathcal{X}_{n}=h\Z^2$ denote the grid of spacing $h$. For simplicity of presentation we will ignore boundary conditions. For $x \in \mathcal{X}_{n}$, let $N_{h}(x)\subset \mathcal{X}_{n}$ denote the neighboring grid points of $x$, and let 
% \[V_{h}(x) = \{y-x \, : \, y \in N_{h}(x)\}\]
% denote the vectors pointing from $x$ to each neighbor. 

The schemes we consider in this paper are wide stencil schemes, inspired by schemes for degenerate elliptic equations such as the Monge-Amp\`ere equation~\cite{oberman2008wide}. Consider the first-order Hamilton-Jacobi equation of the form
\begin{equation*}
    H(\nabla u, u, x) =0 \; \text{ in } \Omega.
\end{equation*}
We recall (see \cite{bardi1997optimal}) that the notion of viscosity subsolution can be equivalently expressed as
\[\sup_{p \in D^-(u, x) } H_*(p, u, x) \leq 0\; \text{ in } \Omega\]
where the subdifferential set $D^-$ is defined as
\begin{align*}
D^-(u, x) :=  \Big\{ p\in \R^d:\, u(y) - u(x) \leq p\cdot (y-x) + o(|x-y|) \text{ as } y\to x \Big\}.
\end{align*}
Similarly, the notion of viscosity supersolution can be expressed as  
\[\inf_{p \in D^+(u, x) } H^*(p, u, x) \geq 0\; \text{ in } \Omega\]
where the superdifferential set $D^+$ is defined as
\begin{align*}
D^+(u, x) :=  \Big\{ p\in \R^d:\, u(y) - u(x) \geq p\cdot (y-x) + o(|x-y|) \text{ as } y\to x \Big\}.
\end{align*}
When $u$ is quasiconcave, so that the set $$\{y\in \R^d \, : \, u(y)\geq u(x)\}$$ is convex, we can drop the $o(|x-y|)$ term from the definition of the subdifferential, and equivalently write
\[D^-(u, x) =  \Big\{ p\in \R^d:\, u(y) \leq u(x) + p\cdot (y-x) \text{ for } y \text{ near } x \Big\}.\]
Since we are only concerned with the \emph{direction} of the gradient, and not the magnitude, we can further focus our attention only on the sign of $p\cdot (y-x)$. This leads to the following approximation of the subdifferential set on a general point cloud
\begin{align}\label{eq:P}
P_h^-(u,u(x),x) &:= \Big\{p \in \R^d \, : \, -p \in V_{h}(x), \text{ and } \\
&\hspace{1in}\forall y \in N_{h}(x), \ \ p\cdot (y-x) < 0 \implies u(y) \leq u(x)\Big\}. \notag
\end{align} 
We should explain the choice that $-p \in V_h(x)$ was made so that for any $p\in P^-_h(u,u(x),x)$, we have $x-p\in \X_n$, so that we can form a backward difference quotient (which is upwind/montone). Notice that we do not intend for $P_h^-$ to exactly approximate $D^-$ as $h\to 0$, since the magnitude $|p|$ will in general not converge to $|\nabla u(x)|$. This is the reason for the alternative notation $P^-_h$ instead of, say, $D^-_h$.  Instead, as we show below, the \emph{direction} of $p\in P_h^-$ converges to the direction of the gradient $\nabla u(x)$ as $h\to 0$. 
%\red Check that this $-p$ change is modified below.\nc
%\red I changed the definition of $P_h^-$ so that $p$ should point in the direction of the gradient, and not in the opposite direction. I think this is much more natural. Essentially, the new $P^-_h$ is the negative of the old one. Please check if this requires any changes below (the figure should be changed too).  \nc

The set-valued operator $P^-_h(u,u(x),x)$ is the collection of all displacement vectors that support the convex super level set $\{u\geq u(x)\}$. The displacement vector in the set operator lies in the opposite direction of $\nabla u$, that is the downwind direction. See Figure \ref{fig:P} for an illustration. We can also define an analogous approximation $P^+_h$ of the superdifferential, but this is generally the empty set for quasiconcave functions (but would be appropriate for quasiconvex functions).

This set-valued operator has many useful properties that allow us to easily construct convergent monotone schemes for quasiconcave viscosity solutions.  In what follows, we present some properties of the operator and new monotone schemes based on this operator.

\begin{figure}
\centering
\subfloat[]{

\tikzset{every picture/.style={line width=0.75pt}} %set default line width to 0.75pt        

\begin{tikzpicture}[x=0.75pt,y=0.75pt,yscale=-0.7,xscale=0.7]
%uncomment if require: \path (0,490); %set diagram left start at 0, and has height of 490

%Straight Lines [id:da9727452324809687] 
\draw    (20,96.51) -- (155,264.45) ;
%Shape: Polygon Curved [id:ds020217732760609897] 
\draw   (102.5,50.48) .. controls (115.5,5.87) and (248.5,11.09) .. (312.5,60.48) .. controls (376.5,109.87) and (320.5,198.98) .. (272.5,230.48) .. controls (224.5,261.98) and (87.5,182.98) .. (87.5,180.48) .. controls (87.5,177.98) and (89.5,95.09) .. (102.5,50.48) -- cycle ;
%Straight Lines [id:da7222815540509234] 
\draw    (87.5,180.48) -- (145.85,140.72) ;
\draw [shift={(147.5,139.59)}, rotate = 145.73] [color={rgb, 255:red, 0; green, 0; blue, 0 }  ][line width=0.75]    (10.93,-3.29) .. controls (6.95,-1.4) and (3.31,-0.3) .. (0,0) .. controls (3.31,0.3) and (6.95,1.4) .. (10.93,3.29)   ;
\draw [shift={(87.5,180.48)}, rotate = 325.73] [color={rgb, 255:red, 0; green, 0; blue, 0 }  ][fill={rgb, 255:red, 0; green, 0; blue, 0 }  ][line width=0.75]      (0, 0) circle [x radius= 3.35, y radius= 3.35]   ;

% Text Node
\draw (119.5,163.44) node [anchor=north west][inner sep=0.75pt]    {$p\in P_{h}^{-}$};
% Text Node
\draw (68.5,181.99) node [anchor=north west][inner sep=0.75pt]    {$x$};
% Text Node
\draw (148.5,101.99) node [anchor=north west][inner sep=0.75pt]    {$\{y: u(y) \geq u( x)\}$};

\end{tikzpicture}
\label{fig:P}}
% \hspace{1cm}
\subfloat[]{

\tikzset{every picture/.style={line width=0.75pt}} %set default line width to 0.75pt        

\begin{tikzpicture}[x=0.75pt,y=0.75pt,yscale=-0.7,xscale=0.7]
%uncomment if require: \path (0,490); %set diagram left start at 0, and has height of 490

%Straight Lines [id:da9727452324809687] 
\draw    (20,96.51) -- (155,264.45) ;
%Shape: Polygon Curved [id:ds020217732760609897] 
\draw   (110,70) .. controls (123,25.39) and (325,-3.5) .. (370,60) .. controls (415,123.5) and (407.5,194.02) .. (270,220) .. controls (132.5,245.98) and (87.5,182.98) .. (87.5,180.48) .. controls (87.5,177.98) and (97,114.61) .. (110,70) -- cycle ;
%Straight Lines [id:da7222815540509234] 
\draw    (87.5,180.48) -- (145.85,140.72) ;
\draw [shift={(147.5,139.59)}, rotate = 145.73] [color={rgb, 255:red, 0; green, 0; blue, 0 }  ][line width=0.75]    (10.93,-3.29) .. controls (6.95,-1.4) and (3.31,-0.3) .. (0,0) .. controls (3.31,0.3) and (6.95,1.4) .. (10.93,3.29)   ;
\draw [shift={(87.5,180.48)}, rotate = 325.73] [color={rgb, 255:red, 0; green, 0; blue, 0 }  ][fill={rgb, 255:red, 0; green, 0; blue, 0 }  ][line width=0.75]      (0, 0) circle [x radius= 3.35, y radius= 3.35]   ;
%Shape: Polygon Curved [id:ds33032693938846414] 
\draw  [fill={rgb, 255:red, 155; green, 155; blue, 155 }  ,fill opacity=0.33 ] (150,70) .. controls (174,51.5) and (282,29.5) .. (280,100) .. controls (278,170.5) and (276,183.5) .. (210,200) .. controls (144,216.5) and (88.5,181.98) .. (87.5,180.48) .. controls (86.5,178.98) and (126,88.5) .. (150,70) -- cycle ;

% Text Node
\draw (119.5,163.44) node [anchor=north west][inner sep=0.75pt]    {$p \in  P_{h}^{-}$};
% Text Node
\draw (68.5,181.99) node [anchor=north west][inner sep=0.75pt]    {$x$};
% Text Node
\draw (151,92.4) node [anchor=north west][inner sep=0.75pt]    {$\{u \geq t\}$};
% Text Node
\draw (301,92.4) node [anchor=north west][inner sep=0.75pt]    {$\{ v \geq t\}$};
% Text Node
\draw (251,232.4) node [anchor=north west][inner sep=0.75pt]    {$u( x) = t= v( x)$};

\end{tikzpicture}
\label{fig:touch}}
\caption{(a) An example of a vector $p$ belonging to the subdifferential $P^-_h(u,u(x),x)$ and (b) an illustration of the set-valued monotonicity of $P^-_h$ with $u \leq v$. }
\label{fig:Ptouch}
\end{figure}
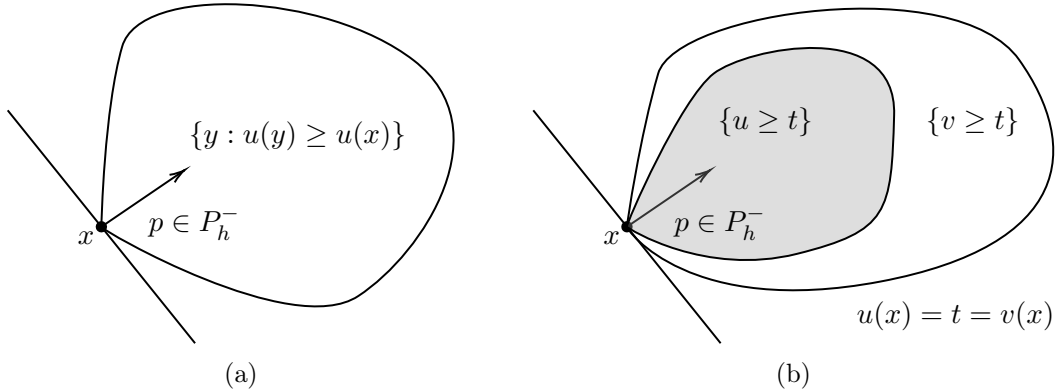

\subsection{Properties of a subdifferential set}

\noindent Monotonicity requires the scheme be a decreasing function of $u_h(y)$ for all neighboring grid points $y\in N_{h}(x)$. 
If the scheme is also an increasing function of $u_h(x)$, then the scheme is often called elliptic~\cite{oberman2006convergent}. The terms \emph{monotone} and \emph{upwind} are used interchangeably for first order equations, and refer to the same property.

A key property of $P_h^-$ is the following monotonicity with respect to set inclusion, which is immediate from the definition.
\begin{proposition}\label{prop:Pm}
For all $t\in \R$, $x \in \mathcal{X}_{n}$ and $u,v:\mathcal{X}_{n} \to \R$
\begin{equation}\label{eq:Pm}
{u\leq v \implies P_h^-(u,t,x) \supset P_h^-(v,t,x).}
\end{equation}
\end{proposition}

\noindent In words, if $u(x)=t=v(x)$ and $u\leq v$, then any halfspace supporting $\{v\geq t\}$ also supports $\{u \geq t\}$. See Figure \ref{fig:touch} for an illustration. Let us write $P^-_h[u](x) = P^-_h(u,u(x),x)$ for simplicity.

% \begin{proposition}\label{prop:ph-p}
%     If $P^-(u,u(x),x)$ is empty and $d\theta(x) < 1$, then $P^-_h(u,u(x),x)$ is empty.
% \end{proposition}

Next, we present a theorem that establishes conditions under which the subdifferential set is nonempty.
It turns out that this requires \emph{strict} quasiconcavity of the test function $\phi$. Without the strictness, one can choose a sufficiently flat function $\phi$, depending on the local point cloud structure, for which the subdifferential set becomes empty. 

Throughout this section, given $x_0\in\mathcal{X}_{n}$, we will assume $\phi \in C^\infty(\mathbb{R}^d)$ is a smooth function for which there exists $h_0>0$ such that
\begin{equation*}
    q \cdot \nabla \phi(x) = 0 \, \implies q \cdot \nabla^2 \phi(x) q < 0 \quad \text{and} \quad |\nabla \phi(x)| > 0
\end{equation*}
for all $ x \in B(x_0,h_0)$.
The first part is equivalent to $\phi$ being strictly quasiconcave by Lemma~\ref{lem:neighbor-quasiconcave}. By defining a function
\begin{equation}\label{eq:def-quasi-function}
    % {L}(\phi, x) := \sup_{\substack{q\cdot \nabla\phi(x) = 0 \\ |q| = 1}} q \cdot \nabla^2 \phi(x) q,
    {L}(X, p) := \sup_{\substack{q\cdot p = 0 \\ |q| = 1}} q \cdot X q,
\end{equation}
%\red I think it would be more clear, and standard, to write
%    \[{L}(X,p) := \sup_{\substack{q\cdot p = 0 \\ |q| = 1}} q \cdot X q,\]
%    and then $L(\nabla^2\phi(x),\nabla \phi(x))$. This would require changes in other places as well.
%\nc
we can rewrite the assumption as
\begin{equation}\label{eq:assumption-concave-phi}
    \sup_{ x\in B(x_0,h_0)} {L}(\nabla^2 \phi(x), \nabla \phi(x)) < 0 .
\end{equation}

\begin{theorem}[Existence]\label{thm:existence}
    Let $x_0 \in \mathcal{X}_{n}$ and assume $\phi\in C^\infty(\mathbb{R}^d)$ satisfies $|\nabla \phi(x_0)|>0$ and ~\eqref{eq:assumption-concave-phi}.  Denote by
    \begin{align*}
        A_1   := - {L}(\nabla^2 \phi(x_0), \nabla \phi(x_0))
    \end{align*}
    where $L$ is defined in~\eqref{eq:def-quasi-function}.
    % \red Should $x$ be $x_0$? And is the gradient condition needed, since it is in \eqref{eq:assumption-concave-phi}? \nc 
    Then the subdifferential set $P_{h}[\phi](x_0)$ is nonempty if $d\theta(x_0)$ and $\delta$ satisfy
    \begin{equation}\label{eq:thm-existence-condition}
        % d\theta(x_0) < \frac{\delta A}{2|\nabla\phi(x_0)| + \delta(A + C_1 + C_2)}.
        d\theta(x_0) \leq \frac{A_1 \delta}{2|\nabla\phi(x_0)| + C \delta}
    \end{equation}
    where $C$ is a positive constant depending on $\phi$.
\end{theorem}
\begin{proof}
    By the definition of $d\theta$, there exists $-p\in V_{h}(x_0)$ such that
    \[
        w(\nabla \phi(x_0), p) \leq d\theta(x_0).
    \]
    We want to show $\phi(x_0+q) \leq \phi(x_0)$ for any $q \in V_{h}(x_0)$ such that $w(p, q) > \pi/2$. Choose $q \in V_{h}(x_0)$ such that $w(p, q) > \pi/2$. By Lemma~\ref{lem:p-q-perp}, we have $w(\nabla \phi(x_0), q) > \pi/2 - d\theta(x_0)$. If $w(\nabla \phi(x_0), q) > \pi/2$, then $\phi(x_0+q) \leq \phi(x_0)$ by Lemma~\ref{lem:neighbor-quasiconcave}. Thus, assume
    \[
        \pi/2 - d\theta(x_0) < w(\nabla \phi(x_0), q) \leq \pi/2.
    \]
    Decompose $q$ such that
    \[
        q = |q| \left( \cos\Theta \, \frac{r}{|r|} + \sin\Theta \frac{\nabla\phi(x_0)}{|\nabla\phi(x_0)|} \right)
    \]
    where $r$ is an orthogonal vector to $\nabla\phi(x_0)$ and $\Theta = w(r, q) = \pi/2 - w(\nabla \phi(x_0), q)$. 
    Using a Taylor expansion of $\phi$,
    \begin{align*}
        &\phi(x_0 + q)\\ 
        &\leq \phi(x_0) + |\nabla \phi(x_0)| |q| \sin d\theta(x_0) + \frac{1}{2} q \cdot \nabla^2\phi(x_0) q\\
            &= \phi(x_0) + |\nabla \phi(x_0)| |q| \sin d\theta(x_0) + \frac{|q|^2}{2} \bigg( 
                        \cos^2\Theta \frac{r}{|r|} \cdot \nabla^2\phi(x_0)\frac{r}{|r|}\\
                        &\hspace{2cm} + \sin^2\Theta \frac{\nabla\phi(x_0)}{|\nabla\phi(x_0)|} \cdot \nabla^2\phi(x_0)\frac{\nabla\phi(x_0)}{|\nabla\phi(x_0)|}
                        + 2 \sin\Theta \cos\Theta  \frac{\nabla\phi(x_0)}{|\nabla\phi(x_0)|} \cdot \nabla^2\phi(x_0)\frac{r}{|r|} \bigg)\\
            &\leq \phi(x_0) + |\nabla \phi(x_0)| |q| \sin d\theta(x_0) + \frac{|q|^2}{2} (
                        -A_1 \cos^2\Theta + C_1 \sin^2\Theta 
                        + 2 C_2 \sin\Theta )
    \end{align*}
    where we denote
    \begin{equation}\label{eq:in-thm-existence-constants}
        \begin{aligned}
            % A_1   &= \min_{\substack{x\in B(x_0,h)\\r\cdot \nabla\phi(x)=0}} - \frac{r}{|r|} \cdot \nabla^2\phi(x)\frac{r}{|r|},\\
            % C_1 &= \max_{x\in B(x_0,h)} \left|\frac{\nabla\phi(x)}{|\nabla\phi(x)|} \cdot \nabla^2\phi(x)\frac{\nabla\phi(x)}{|\nabla\phi(x)|} \right|,\\
            % C_2 &= \max_{\substack{x\in B(x_0,h)\\r\cdot \nabla\phi(x)=0}}  \left|\frac{\nabla\phi(x)}{|\nabla\phi(x)|} \cdot \nabla^2\phi(x)\frac{r}{|r|} \right|
            % A_1   &= - \frac{r}{|r|} \cdot \nabla^2\phi(x_0)\frac{r}{|r|},\\
            C_1 &= \sup_{x \in B(x_0,h_0)} \left|\frac{\nabla\phi(x)}{|\nabla\phi(x)|} \cdot \nabla^2\phi(x)\frac{\nabla\phi(x)}{|\nabla\phi(x)|} \right| ,\\
            C_2 &= \sup_{\substack{x \in B(x_0,h_0) \\ r\cdot \nabla\phi(x)=0}} \left|\frac{\nabla\phi(x_0)}{|\nabla\phi(x_0)|} \cdot \nabla^2\phi(x_0)\frac{r}{|r|} \right|.
        \end{aligned}
    \end{equation}
    Using $\cos^2\Theta + \sin^2\Theta = 1$ and $\Theta < d\theta(x_0)$,
    \begin{align*}
            &\leq \phi(x_0) + |\nabla \phi(x_0)| |q| \sin d\theta(x_0) + \frac{|q|^2}{2} (
                        -A_1 + (A_1 + C_1 + 2 C_2)\sin d\theta(x_0))\\
            &\leq \phi(x_0) + \frac{|q|^2}{2} \left( \Big( \frac{2|\nabla\phi(x_0)|}{\delta} + C \Big) d\theta(x_0) - A_1 \right)\\
            & \leq \phi(x_0)
    \end{align*}
    where $C= A_1 + C_1 + 2 C_2$ and the last inequality comes from~\eqref{eq:thm-existence-condition}.
    Thus, $p \in P^-_h[\phi](x_0)$.
\end{proof}

% Theorem~\ref{thm:existence} \blue \sout{shows the} gives conditions under which the \nc subdifferential set is nonempty on unstructured point \blue clouds \sout{cloud settings} \nc in $\mathbb{R}^d$. If the point cloud satisfies some form of \blue \sout{symmetricity} symmetry \nc in $\mathbb{R}^2$, then the set is nonempty without any assumption on $\delta$. \red I'm not sure if it's correct so say without any assumption on $\delta$, since the main condition in Theorem \ref{thm:existence} can be viewed as a condition on $\delta$. \nc

Theorem~\ref{thm:existence} gives conditions that guarantee the subdifferential set to be nonempty on general point clouds in $\mathbb{R}^d$. Note that $\phi$ needs to be strictly quasiconcave because the constant $A_1$ being strictly positive is crucial for the condition~\eqref{eq:thm-existence-condition} to hold. If the point cloud satisfies some form of symmetry in $\mathbb{R}^2$, then the set can be nonempty with a quasiconcave $\phi$.
% \red I'm not sure if it's correct so say without any assumption on $\delta$, since the main condition in Theorem \ref{thm:existence} can be viewed as a condition on $\delta$. \nc

\begin{theorem}[Existence on symmetric stencils on $\mathbb{R}^2$]
    Let $x_0 \in \mathcal{X}_{n} \subset \mathbb{R}^2$ and assume $\phi \in C^{\infty}(\mathbb{R}^2)$ is quasiconcave and $|\nabla \phi(x_0)| > 0$.
    % \begin{equation}\label{eq:thm-d2}
    %     q \cdot \nabla \phi(x_0) = 0 \, \Rightarrow \, q \cdot \nabla^2\phi(x_0) q \leq 0.
    % \end{equation}
    Suppose $V_{h}(x_0)$ satisfies
    % \begin{enumerate}[{\normalfont (i)}]
    \begin{enumerate}[label=(\roman*)]
        \item If $p \in V_{h}(x_0)$ then $-p \in V_{h}(x_0)$, and
        \item If $p \in V_{h}(x_0)$ then there exists $q \in V_{h}(x_0)$ such that $p \cdot q = 0$.
    \end{enumerate}
    Then the subdifferential set $P_{h}^{-}[\phi](x_0)$ is nonempty.
\end{theorem}
\begin{proof}
    Choose $-p^* \in V_{h}(x_0)$ such that 
    \begin{equation}\label{eq:p-star-argmin}
        -p^* = \argmin_{-p \in V_{h}(x_0)} w(\nabla \phi(x_0), p).
        % p^* = \argmax_{p\in V_{h}(x_0)} \, \nabla \phi(x_0) \cdot p
    \end{equation}
    Note that $w(\nabla \phi(x_0), p^*) \leq d\theta$. We will show $p^* \in P^-_h[\phi](x_0)$.
    Choose any $q \in V_{h}(x_0)$ such that $w(p^*, q) > \pi/2$. By Lemma~\ref{lem:p-q-perp}, we have
    \[
        \frac{\nabla \phi(x_0)}{|\nabla \phi(x_0)|} \cdot \frac{q}{|q|} < \cos\left( \frac{\pi}{2} - w(\nabla \phi(x_0), p^*) \right)
    \]
    which follows that $w(\nabla \phi(x_0), q) > \pi/2 - w(\nabla \phi(x_0), p^*)$. 
    Suppose $w(\nabla \phi(x_0), q) \leq \pi/2$. Then there exists $\eps$ 
    %\red We use $\epsilon,\delta$ for conditions on the point cloud earlier on, so we should find some other notation for this $\epsilon$ (or for the mesh constants). Please check other places where $\epsilon$ appears as well.\nc \wonjun{changed $\epsilon$ in conditions on the point cloud to another letter "$R$"}
     such that $0 \leq \eps < w(\nabla \phi(x_0), p^*)$ and
    \[
        w(\nabla \phi(x_0), q) = \frac{\pi}{2} - \eps.
    \]
    Then, by the symmetry of $N_{h}(x_0)$, there exists $r \in V_{h}(x_0)$ such that $w(q,r) = \pi/2$
    and by Lemma~\ref{lem:p-q-perp}, $r$ also satisfies
    \[
        w(\nabla \phi(x_0), r) = \eps.
    \]
    Thus, $w(\nabla \phi(x_0), r) < w(\nabla \phi(x_0), p^*)$, which is a contradiction to the definition of $p^*$. Thus, $q$ satisfies $w(\nabla \phi(x_0), q) > \pi/2$. Since $\phi$ is quasiconcave, we have $\phi(x_0 + q) \leq \phi(x_0)$. Thus, $p^* \in P_{h}^{-}[\phi](x_0)$. 
\end{proof}

Coming back to the general unstructured point cloud setting in $\mathbb{R}^d$, we will show that the subdifferential set $P^-_h[\phi](x_0)$ converges to the direction of $\nabla \phi(x_0)$ in a sense that
\[
    \lim_{h\rightarrow 0} \min_{p\in P^-_h[\phi](x_0)} \frac{p}{|p|} \cdot \frac{\nabla \phi(x_0)}{|\nabla \phi(x_0)|} = 1.
\]
% \red I don't see where $C$ is used in Theorem \ref{thm:P-upper-bound}.\nc \wonjun{It is used in Proposition~\ref{prop:con} to get the upper bound.}

\begin{theorem}\label{thm:P-upper-bound}
Let $x_0 \in \mathcal{X}_{n}$ and $\phi \in C^\infty(\mathbb{R}^d)$ be such that $|\nabla \phi(x_0)| > 0$.
Denote by
    \begin{align*}
        A_2  := \max_{\substack{r\cdot \nabla\phi(x_0)=0}} - \frac{r}{|r|} \cdot \nabla^2\phi(x_0)\frac{r}{|r|}.
        % A_2  := - {L}(\nabla^2 \phi(x_0), \nabla \phi(x_0)).
    \end{align*}
If $p \in P_{h}^{-}[\phi](x_0)$, then 
\begin{equation}\label{eq:subdiff-angle}
       % w(p, \nabla \phi(x_0)) \leq \frac{A_1\, h}{2|\nabla \phi(x_0)| - h(C_1 + 2C_2)}   + 2 d\theta. 
       w(p, \nabla \phi(x_0)) \leq \frac{A_2 h}{2|\nabla \phi(x_0)| - h (C_1 + 2 C_2)}   + 2 d\theta
\end{equation}
where $C_1$ and $C_2$ are from~\eqref{eq:in-thm-existence-constants}.
\end{theorem}
\begin{proof}
    For simplicity, denote by  $\Theta:= \frac{A_2}{2|\nabla \phi(x_0)|/h - C_1 - 2C_2}$. Suppose, on the contrary, there exists $p \in P_{h}^{-}[\phi](x_0)$ such that 
    \[
        w(p, \nabla \phi(x_0)) = \Theta   + 2d\theta + \eps 
    \]
    for some constant $\eps > 0$. By Lemma~\ref{lem:p-q-perp}, there exists a vector $q \in \mathbb{R}^d$ such that $w(p,q) = \pi/2 + d\theta + \eps$ and 
    \[
        \frac{\nabla \phi(x_0)}{|\nabla \phi(x_0)|} \cdot \frac{q}{|q|} = \sin  (\Theta + d\theta).
    \]
    Thus,
    \[
        w(q, \nabla \phi(x_0)) = \frac{\pi}{2} - \Theta - d\theta.
    \]
    By the definition of $d\theta$, there exists $q' \in V_{h}(x_0)$ such that $w(q,q') < d\theta$. By Lemma~\ref{lem:p-q-perp}, 
    \begin{align*}
        \frac{p}{|p|}\cdot \frac{q'}{|q'|} \leq \cos\left( \frac{\pi}{2} + \eps \right) =  - \sin \eps < 0 \Longrightarrow w(p,q') > \frac{\pi}{2}.
    \end{align*}
    Again, by Lemma~\ref{lem:p-q-perp}, we have
    \begin{align*}
        \frac{\nabla\phi(x_0)}{|\nabla\phi(x_0)|}\cdot \frac{q'}{|q'|} > \cos\left( \frac{\pi}{2} - \Theta \right) =  \sin \Theta.
    \end{align*}
    Using a Taylor expansion on $\phi$,
    \begin{align*}
        \phi(x_0+q') > \phi(x_0) + |\nabla\phi(x_0)| |q'| \sin\Theta + \frac{1}{2} q' \cdot \nabla^2\phi(x_0) q'.
    \end{align*}
    Similar to the proof in Theorem~\ref{thm:existence}, we may use the orthogonal decomposition of $q'$. This leads to
     \begin{align*}
        \phi(x_0+q') &> \phi(x_0) + |\nabla\phi(x_0)| |q'| \sin\Theta + \frac{|q'|^2}{2} \left( -A_2 \cos^2\Theta - C_1 \sin^2 \Theta - 2 C_2 \sin \Theta \right)\\
        &\geq \phi(x_0) + |\nabla\phi(x_0)| |q'| \sin\Theta + \frac{|q'|^2}{2} \left( -A_2 - (C_1 + 2C_2) \sin \Theta \right)\\
        &\geq \phi(x_0) + \frac{|q'|^2}{2} \left( \left(\frac{2|\nabla\phi(x_0)|}{h} - C_1 - 2 C_2 \right) \sin\Theta  - A_2  \right)\\
        &= \phi(x_0).
    \end{align*}
    Since $w(p,q')>\pi/2$ and $\phi(x_0+q') > \phi(x_0)$, this is a contradiction to $p \in P_{h}^{-}[\phi](x_0)$. 
\end{proof}

\subsection{Monotone and consistent scheme}

% \begin{itemize}
%     \item \red Comments are in red \nc
%     \item \blue Suggested changes are in blue \nc
% \end{itemize}
% \red I suggest we use $\nabla$ for the gradient, instead of $D$. (most places are $\nabla$ but a few $D$ remain) \nc

% \red 

% \begin{itemize}
%     \item I would suggest to split this section into different subsections depending on the type of PDE we are addressing. Maybe just first and second order are enough. Then I would integrate it with some of the discussion in the first part of the paper. That is, introduce the PDEs, then their approximations, and prove the consistency. This may help motivate the results in this section, which otherwise can seem hard to interpret.
%     \item We also need to figure out how to incorporate the concavity assumption into the Barles-Souganidis framework without reproving anything. The consistency statement must hold for all test functions, not just those that satisfy the concavity assumption. 
% \end{itemize}
% \nc

In this section, we use the subdifferential set to construct monotone and consistent schemes for Hamilton-Jacobi equations with quasiconcave solutions. 
% (in later sections, we will address higher order equations, but for now, let us focus on the simplest setup). 
Since we are only interested in viscosity solutions that are quasiconcave, we consider the following operator
\begin{align}\label{eq:new-H-operator}
    \tilde{H}(\nabla^2 u, \nabla u, u, x) := \begin{cases}
        H(\nabla^2 u, \nabla u, u, x) & \text{if } L(\nabla^2 u(x),\nabla u(x)) \leq 0,\\
        -\infty & \text{otherwise.}
    \end{cases}
\end{align}
A similar operator is used in~\cite{calder2020convex}. Since $H$ and $L$ are elliptic, $\tilde{H}$ is also elliptic, i.e. for any $p\in\mathbb{R}^d$, $z\in \R$, $x\in \Omega$, and $X,Y\in \mathbb{R}^{d\times d}_{sym}$ we have
\[
    X \leq Y \implies \tilde{H}(X,p,z,x) \geq \tilde{H}(Y,p,z,x).
\]
%\red Make sure to check the ellipticity closely, since the operator is not exactly the same as \cite{calder2020convex}.\nc \wonjun{Confirmed that the operator is elliptic.}
Recall that the condition $L(\nabla^2 u(x),\nabla u(x)) \leq 0$ is the second-order necessary condition for the quasiconcavity from Lemma~\ref{lem:quasiconcave}. 
Thus, if $u$ is a quasiconcave solution of $H$ then the subdifferential set is nonempty for all $x\in\Omega$ and
\[ \tilde{H}(\nabla^2 u, \nabla u, u, x) = H(\nabla^2 u, \nabla u, u, x) = 0. \]

% We will assume that there exists a constant $C>1$ such that \begin{equation}\label{eq:delta-theta-relation} 1 \leq \frac{a \delta}{d\theta(x)|\nabla \phi(x)|} \leq  \end{equation} for all $x \in \mathcal{X}_{n}$. 
Throughout the section, we will assume that 
\begin{equation}\label{eq:delta-theta-relation}
    d\theta < h^{1+\alpha}    
\end{equation}
for some $\alpha >0$. Given a strictly quasiconcave function, the inequality in Theorem~\ref{thm:existence} is satisfied for all $x\in\mathcal{X}_{n}$ by choosing a sufficiently small $h$.
Thus, the subdifferential set is nonempty in $\mathcal{X}_{n}$.

We propose a new numerical scheme $S_h$ using the subdifferential operator,
\begin{equation}\label{eq:general-form-s}
    S_h(u,u(x),x) := 
    \begin{cases}
        \max_{p \in P_{h}[u](x)} F_h(p,u,u(x),x) & \text{if } P^-_h[u](x) \neq \emptyset,\\
        -\infty & \text{otherwise}
    \end{cases}
\end{equation}
where $F_h = F_h(p,u,t,x)$ is a function that satisfies
\begin{enumerate}[label=(F\arabic*)]
    \item\label{item:F-mon} $F_h$ is monotone,
    \item\label{item:F-cts} $F_h$ is continuous in $u$ and $t$,
    \item\label{item:F-H} given $x\in\Omega$, $p\in\mathbb{R}^d$, $X \in \mathbb{R}^{d\times d}_{sym}$ and $u \in C^\infty(\mathbb{R}^d)$, $F_h$ approximates the function $H(X,p,u,x)$ in~\eqref{eq:new-H-operator} such that for all 
        \begin{equation*}
            \left| F_h(p,u,u(x),x) - H(X,p,u,x) \right| \leq C(h^{m_1} + d\theta^{m_2}),\quad m_1,m_2 \geq 1.
        \end{equation*} 
\end{enumerate}

% Let us present the properties of the proposed scheme. The following lemma shows that numerical solutions of the scheme $S_h$ is locally quasiconcave at every $x\in\Omega$.

% \wonjun{This lemma is to show that the solution of the scheme has to be quasiconcave, but the lemma is not used anywhere. We may be able to delete it.}
% \begin{lemma}\label{lem:numerical-sol-qc}
%     Let $u_h$ be a numerical solution of the scheme $S_h$ defined in~\eqref{eq:general-form-s} such that
%     \[ S_h(u_h, u_h(x), x) = 0, \quad \text{for all } x \in \mathcal{X}_{n}. \] 
%     Then for any $x_0 \in \mathcal{X}_{n}$ and $x,y \in N_h(x_0)$ such that $x_0 = \lambda x + (1-\lambda) y$ for some $0 < \lambda < 1$, $u_h$ satisfies
%     \[
%         u_h(x_0) \geq \min\big( u_h(x), u_h(y) \big).
%     \]
%     % Moreover, for any $q \in V_{h}(x_0)$,
%     % \[ q \cdot \nabla u_h(x_0) = 0 \implies q \cdot \nabla^2 u_h(x_0) q < 0. \]
% \end{lemma}
% \begin{proof}
%     For the sake of contradiction, suppose there exists $x_0 \in \mathcal{X}_{n}$ and $x,y\in N_h(x_0)$ such that 
%     \[x_0 = \lambda x + (1-\lambda) y \quad  \text{for some} \quad 0 < \lambda < 1\]
%     and
%     \[u_h(x_0) < u_h(x) \quad \text{and} \quad u_h(x_0) <  u_h(y).\]
%     Choose any $- p \in V_h(x_0)$. Then, either $p\cdot (x-x_0) \leq 0$ or $p\cdot (y-x_0) \leq 0$.  Thus, the subdifferential set $P^-_h(u,u(x_0),x_0)$ is empty. This is a contradiction to $u_h$ being the solution of the scheme.
% \end{proof}

The following lemmas shows the monotonicity and consistency of the proposed scheme.

\begin{proposition}[Monotonicity] \label{prop:mon-new-scheme}
    The scheme~\eqref{eq:general-form-s} is monotone.
\end{proposition}
\begin{proof}
    Let $u,v : \bar\Omega \rightarrow \mathbb{R}^d$ be functions such that $u \leq v$ near $x$. Suppose $P^-_{h}(v,v(x),x)$ is nonempty. By Proposition~\ref{prop:Pm} and \ref{item:F-mon}, we have
    \[
    \begin{aligned}
        S_h(u,u(x),x) &= \max_{p \in P_{h}(u,u(x),x)} F_h(p,u,u(x),x) \\
        &\geq \max_{p \in P_{h}(v,u(x),x)} F_h(p,v,u(x),x)
        = S_h(v,u(x),x).
    \end{aligned}
    \]
    If $P^-_{h}(v,v(x),x)$ is empty, then
    \[
    \begin{aligned}
        S_h(u,u(x),x) \geq -\infty
        = S_h(v,u(x),x).
    \end{aligned}
    \]
    Thus, the scheme is monotone.
\end{proof}

\begin{proposition}[Consistency] \label{prop:con}
    Suppose the function $H(X,p,\phi,x)$ in~\eqref{eq:new-H-operator} satisfies
    \begin{equation}\label{eq:condition-lipschitz-f}
        \left| H(X,p,\phi,x) - H(Y,q,\phi,y) \right| \leq 
        C \left( \left| \frac{p}{|p|} - \frac{q}{|q|} \right|
        + |x-y| \right)
    \end{equation}
    for all $X,Y \in \mathbb{R}^{d\times d}_{sym}$, $p,q \in \mathbb{R}^d$, $x,y\in\bar\Omega$, and $C$ is a constant depending on $\phi$ and $\Omega$. 
    Given $x \in \bar\Omega$, assume~\eqref{eq:delta-theta-relation} and $\phi \in C^\infty(\mathbb{R}^d)$ satisfies $|\nabla \phi(x)| > 0$. 
    \begin{enumerate}[label=(\roman*)]
        % \item If $\phi\in C^\infty(\mathbb{R}^{d})$ is strictly quasiconcave, the scheme satsifies
        \item \label{item:con-prop-1} If $L(\nabla^2\phi(x), \nabla\phi(x)) \leq 0$, the scheme satisfies
    \begin{equation*}
        \limsup_{ \substack{h\rightarrow 0^+\\\gamma \rightarrow 0 \\ y\rightarrow x} } S_h(\phi + \gamma, \phi(y) +\gamma, y) \leq \tilde{H}^*(\nabla^2 \phi, \nabla \phi, \phi, x).
    \end{equation*}
        \item If $L(\nabla^2\phi(x), \nabla\phi(x)) > 0$, the scheme satisfies
        \begin{equation*}
            \liminf_{ \substack{h\rightarrow 0^+\\\gamma \rightarrow 0 \\ y\rightarrow x} } S_h(\phi + \gamma, \phi(y) +\gamma, y) \geq \tilde{H}_*(\nabla^2 \phi, \nabla \phi, \phi, x).
        \end{equation*}
    \end{enumerate}

\end{proposition}
\begin{proof}

    We prove the first part of the proposition. Let $x\in\bar\Omega$ and assume $L(\nabla^2\phi(x),\nabla\phi(x)) \leq 0$. Choose ${\eps} >0$ and define $\phi_{\eps}$ a purterbation of $\phi$ such that
    \[ \phi_{\epsilon}(y) = \phi(y) - \frac{{\eps}}{2} |y-x|^2.  \] 
    Then $L(\nabla^2\phi_{\epsilon}(x),\nabla\phi_{\epsilon} (x)) < 0$. Suppose that $\phi_{\epsilon}$ satisfies
    \begin{equation}\label{eq:phi-epsilon-con}
        \lim_{\substack{h\rightarrow 0 \\ \gamma \rightarrow 0 \\ y \rightarrow x}} S_h(\phi_{\epsilon} + \gamma, \phi_{\epsilon}(y) + \gamma, y) 
        =
        H(\nabla^2 \phi_{\epsilon},\nabla \phi_{\epsilon}, \phi_{\epsilon} , x).
    \end{equation}
    By the definition of the operator,
    \begin{align*}
        \tilde H(\nabla^2 \phi_{\epsilon}, \nabla \phi_{\epsilon}, \phi_{\epsilon} , x)
        = H(\nabla^2 \phi_{\epsilon}, \nabla \phi_{\epsilon}, \phi_{\epsilon} , x)
        = H(\nabla^2 \phi,\nabla \phi, \phi , x)
        = \tilde H(\nabla^2 \phi, \nabla \phi, \phi , x).
    \end{align*}
    Thus, by the monotonicity of the scheme,
    \begin{align*}
        \limsup_{\substack{h\rightarrow 0 \\ \gamma \rightarrow 0 \\ y \rightarrow x}} S_h(\phi, \phi(y), y) \leq \limsup_{\substack{h\rightarrow 0 \\ \gamma \rightarrow 0 \\ y \rightarrow x}} S_h(\phi_{\epsilon} + \gamma, \phi_{\epsilon}(y) + \gamma, y) 
        \leq
        \tilde H^*(\nabla^2 \phi, \nabla \phi, \phi, x).
    \end{align*}
    Thus,~\ref{item:con-prop-1} is proven if~\eqref{eq:phi-epsilon-con} is shown.

    Let us show~\eqref{eq:phi-epsilon-con}. From the assumption~\eqref{eq:delta-theta-relation}, there exists $h_0$ such that the inequality~\eqref{eq:thm-existence-condition} is satisfied for all $h<h_0$.
    Furthermore, we may assume $h_0$ is small enough that $L(\nabla^2\phi_{\epsilon}(y), \nabla\phi_{\epsilon}(y)) < 0$ for all $y \in B(x, h_0)$.
    Choose $h < h_0$, $y\in\mathcal{X}_{n}$ such that $|x-y| < h$, and $\gamma > 0$. Denote by $p^* := \argmax_{p \in P_{h}[\phi](x)} F_h(p,\phi_{\epsilon} + \gamma,\phi_{\epsilon}(y) + \gamma,y)$. Then
    \begin{align}
        &\left| F_h(p^*,\phi_{\epsilon}+\gamma,\phi_{\epsilon}(y)+\gamma,y) - H(\nabla^2 \phi_{\epsilon}, \nabla \phi_{\epsilon}, \phi_{\epsilon}, x) \right| \nonumber \\ % \label{eq:consistent-thing-to-bound}\\
        &\leq \left| F_h(p^*,\phi_{\epsilon}+\gamma,\phi_{\epsilon}(y)+\gamma,y) - F_h(p^*,\phi_{\epsilon},\phi_{\epsilon}(y),y)  \right| +\left| F_h(p^*,\phi_{\epsilon},\phi_{\epsilon}(y),y) - H(\nabla^2 \phi_{\epsilon}, \nabla \phi_{\epsilon}, \phi_{\epsilon}, x) \right|.\nonumber
    \end{align}
    By~\ref{item:F-cts}, the first term converges to $0$ as $\gamma \rightarrow 0$. The second term can be bounded by
    \begin{align*}
        &\leq \left| F_h(p^*,\phi_{\epsilon},\phi_{\epsilon}(y),y) - H(\nabla^2 \phi_{\epsilon}, p^*, \phi_{\epsilon}, y) \right|
             + \left| H(\nabla^2 \phi_{\epsilon}, p^*, \phi_{\epsilon}, y) - H(\nabla^2 \phi_{\epsilon}, \nabla \phi_{\epsilon}, \phi_{\epsilon}, x) \right|\nonumber\\
        &\leq C(h^{m_1} + d\theta^{m_2}) + C \left(\left| \frac{p^*}{|p^*|} - \frac{\nabla \phi_{\epsilon}(x)}{|\nabla \phi_{\epsilon}(x)|} \right|  + |x-y|
                \right)
    \end{align*}
    where the last inequality uses~\ref{item:F-H} and~\eqref{eq:condition-lipschitz-f}.  The second term in the last line can be bounded by
    \begin{align*}
        \leq C \left(\left| \frac{p^*}{|p^*|} - \frac{\nabla \phi_{\epsilon}(y)}{|\nabla \phi_{\epsilon}(y)|} \right|  
                        + \left| \frac{\nabla \phi_{\epsilon}(y)}{|\nabla \phi_{\epsilon}(y)|} - \frac{\nabla \phi_{\epsilon}(x)}{|\nabla \phi_{\epsilon}(x)|} \right|  
                        + |x-y|
                        \right).
    \end{align*}
    By  Theorem~\ref{thm:P-upper-bound},
    \[
        \left| \frac{p^*}{|p^*|} - \frac{\nabla \phi_{\epsilon}(y)}{|\nabla \phi_{\epsilon}(y)|} \right|   \leq C(h + d\theta)
    \]
    and since $\phi_{\epsilon} \in C^\infty(\mathbb{R}^d)$,
    \begin{align*}
        \left| \frac{\nabla \phi_{\epsilon}(y)}{|\nabla \phi_{\epsilon}(y)|} - \frac{\nabla \phi_{\epsilon}(x)}{|\nabla \phi_{\epsilon}(x)|} \right|  
        &\leq \left| \frac{\nabla \phi_{\epsilon}(y)}{|\nabla \phi_{\epsilon}(y)|} - \frac{\nabla \phi_{\epsilon}(x)}{|\nabla \phi_{\epsilon}(y)|} \right|  
            + \left| \frac{\nabla \phi_{\epsilon}(x)}{|\nabla \phi_{\epsilon}(y)|} - \frac{\nabla \phi_{\epsilon}(x)}{|\nabla \phi_{\epsilon}(x)|} \right|\\
        &\leq \frac{\max_{z \in B(x,h)} |\nabla^2 \phi_{\epsilon}(z)|}{\min_{z \in B(x,h)} |\nabla \phi_{\epsilon}(z)|} |x-y| \leq C h.
        % &\leq \frac{2A|x-y|}{|\nabla \phi(y)|} \leq C d\theta
    \end{align*}
    This proves~\eqref{eq:phi-epsilon-con}, and thus proves the first part of the proposition.\\

    Next, we prove the second part of the proposition. Since $\phi$ is smooth and $L(\nabla^2\phi(x),\nabla\phi(x)) > 0$, there exists $h_0$ such that $L(\nabla^2\phi(y),\nabla \phi(y)) > 0$ for all $y\in B(x,h_0)$. Thus, for any sequence $y_k \rightarrow x$, there exists $K$ such that $\tilde{H}(\nabla^2\phi, \nabla\phi,\phi,y_k) = -\infty$ for all $k>K$. Thus,
        \[ 
            \liminf_{ \substack{h\rightarrow 0^+\\\gamma \rightarrow 0 \\ y\rightarrow x} } S_h(\phi + \gamma, \phi(y) +\gamma, y) \geq -\infty =  \liminf_{k\rightarrow \infty} \tilde H(\nabla^2 \phi, \nabla \phi, \phi, y_k) \geq \tilde{H}_*(\nabla^2 \phi, \nabla \phi, \phi, x)
        \] 
    which proves the proposition. 

\end{proof}

Finally, we show the scheme $S_h$ is convergent.

\begin{theorem}\label{thm:convergence}
    Suppose the assumption~\eqref{eq:delta-theta-relation} and the strong uniqueness property~\eqref{def:strong-unique} hold. Suppose $u$ is the unique quasiconcave viscosity solution of the PDE
    \begin{equation}\label{eq:new-pde}
    \left\{
        \begin{aligned}
            H(\nabla^2 u, \nabla u, u, x) &= 0 && \text{in } \Omega\\
            u &= g && \text{on } \partial \Omega
        \end{aligned}
    \right.
    \end{equation}
    where $g:\partial\Omega \rightarrow \mathbb{R}$ is a continuous function.
    % where $\tilde{H}$ is defined in~\eqref{eq:new-H-operator}.
    Then the numerical solutions $u_h$ of the scheme $S_h$ converges uniformly to $u$ on $\bar\Omega$.
\end{theorem}
\begin{proof}
    Denote by $\mathcal{X}_{n(h)}$ and $\Gamma_{n(h)}$ the set of points in $\bar\Omega$ and the set of boundary points, respectively, with the number of points $n(h)$ depending on the spatial resolution $h$. Let $\bar{u}, \underbar{$u$}: \bar\Omega \rightarrow \mathbb{R}$ be defined by
    \[
        \bar{u}(x) := \limsup_{\substack{\mathcal{X}_{n(h)}\ni y \rightarrow x\\ h\rightarrow 0}} u_h(y) \quad \text{and} \quad \underbar{$u$}(x) := \liminf_{\substack{ \mathcal{X}_{n(h)}\ni y \rightarrow x\\ h\rightarrow 0}} u_h(y).
    \]
    
    We claim that $\bar{u}$ and $\underbar{$u$}$ are viscosity subsolution and supersolution of~\eqref{eq:new-pde}, respectively. First, let $x_0 \in \Omega$ and $\phi \in C^\infty(\mathbb{R}^d)$ be such that $\underbar{$u$} - \phi$ has a local minimum at $x_0$.
    Without the loss of generality, we can replace $\phi$ by $\phi(x) - \phi(x_0) + \underbar{$u$}(x_0) - K |x- x_0|^2$. By choosing $K$ large enough $\phi$ satisfies the quasiconcavity assumption~\eqref{eq:assumption-concave-phi}, and there exists $\eps>0$ such that
    \[ \underbar{$u$}(x) - \phi(x) > 0 = \underbar{$u$}(x_0) - \phi(x_0)\quad \text{for all } x \in B(x_0,\eps) \cap \bar \Omega. \]
    There exist sequences $h_k \rightarrow 0$ and $y_k \rightarrow x_0$ where $u_{h_k} - \phi$ attains the local minimum at $y_k \in B(x_0,\eps) \cap \mathcal{X}_{n(h_k)}$ for each $k$. Denote by $\gamma_k := u_{h_k}(y_k) - \phi(y_k)$. Then we have $\gamma_k \rightarrow 0$ and $u_{h_k}(x) - \phi(x) \geq \gamma_k$ for all $x\in B(x_0,\eps) \cap \mathcal{X}_{n(h_k)}$. By the definition of $u_h$ and the monotonicity of $S_h$,
    \begin{equation}\label{eq:consistency-holds}
        0 = S_h(u_{h_k}, u_{h_k}(y_k), y_k) \leq S_h(\phi + \gamma_k, \phi(y_k) + \gamma_k, y_k).
    \end{equation}
    By the consistency of $S_h$,
    \begin{align*}
        0 \leq \limsup_{k} S_h(\phi + \gamma_k, \phi(y_k) + \gamma_k, y_k)
        \leq \tilde{H}^*(\nabla^2 \phi(x_0), \nabla \phi(x_0), \phi(x_0), x_0).
    \end{align*}
    If $x_0 \in \partial \Omega$, then we can arrange it so that either $y_k \in \Gamma_{n(h_k)}$ or $y_k \in \mathcal{X}_{n(h_k)}\backslash \Gamma_{n(h_k)}$ for all $k$. In the first case, we have
    \begin{align*}
        \underbar{$u$}(x_0) = \lim_{h_k \rightarrow 0^+} u_{h_k}(y_k) \geq g(x_0),
    \end{align*}
    due to the continuity of $g$. In the second case, by the same argument as above,~\eqref{eq:consistency-holds} holds.
    Thus, $\underbar{$u$}$ is a viscosity supersolution of~\eqref{eq:new-pde}.

    The proof of $\bar{u}$ being a viscosity subsolution of~\eqref{eq:new-pde} is similar to the above proof. The only change is that given a smooth test function $\phi \in C^\infty(\mathbb{R}^d)$ such that $\bar{u} - \phi$ has a local maximum at $x_0$, we add a quadratic term to $\phi$ so that $L(\nabla^2\phi(x_0), \nabla \phi(x_0))>0$ and $x_0$ is a strict local maximum point.
     
    By definitions, $\underbar{$u$} \leq \bar{u}$ on $\bar\Omega$, and by the strong uniqueness property, $\underbar{$u$} \geq \bar{u}$ on $\bar\Omega$. Thus, we have $\underbar{$u$} \equiv \bar{u}$, and we conclude $u_h$ converges uniformly to the unique viscosity solution of~\eqref{eq:new-pde}.

    % Choose $x_0 \in \Omega$ and a smooth test function $\phi \in C^\infty(\mathbb{R}^d)$ such that $\bar{u} - \phi$ has a local maximum at $x_0$. Without loss of generality, we may assume $L(\phi,x_0)>0$, $\bar{u}(x_0) - \phi(x_0)=0$ and $x_0$ is a strict local maximum point by adding a quadratic term to $\phi$. 
    % Thus there exists $r>0$ such that 
    % \[ \bar{u}(x) - \phi(x) < 0= \bar{u}(x_0) - \phi(x_0) \]
    % for all $x\in B(x_0,r)$.
    % Consider sequences $h_k \rightarrow 0$ and $y_k \rightarrow x_0$ such that $u_{h_k} - \phi$ has a local maximum at $y_k$. Denote $\gamma_k := u_{h_k}(y_k) - \phi(y_k)$. Then we have $\gamma_k \rightarrow 0$ and $u_{h_k}(x) - \phi(x) \leq \gamma_k$ for all $x\in B(x_0,r)$. By the definition of $u_h$ and the monotonicity of $S_h$,
    % \[
    %     0 = S_h(u_{h_k}, u_{h_k}(y_k), y_k) \geq S_h(\phi + \gamma_k, \phi(y_k) + \gamma_k, y_k).
    % \]
    % By the consistency of $S_h$ in Proposition~\ref{prop:con},
    % \begin{align*}
    %     0 &\geq \liminf_{k} S_h(\phi + \gamma_k, \phi(y_k) + \gamma_k, y_k) \geq H_*(\nabla \phi(x_0), \phi(x_0), x_0).
    % \end{align*}
    % Thus, $\bar{u}$ is a viscosity subsolution.\\

\end{proof}

\subsection{Iterative scheme}
\label{sec:schemes}

To solve the global scheme (S$_h$), we propose an implicit iterative method. Given point clouds $\mathcal{X}_{n}$, the implicit iteration can be formulated by solving
\[ \begin{aligned}
S_h(u^{n}_h,u^{n+1}_h(x),x) &= 0&&\text{ for every } x \in \mathcal{X}_{n}
\end{aligned}\]
starting from some initial guess $u_h^0\in \M_h$. Using the monotonicity of the scheme $u^{n+1}$ can be computed through bisection methods.
 Since the scheme is monotone, homogeneous (for mean curvature motion), and satisfies a maximum principle, it is possible to show that the resulting solution is within $O(h)$ of the exact solution of the scheme. The alogrithm is displayed in Algorithm~\ref{alg:main}. In the algorithm, the error of $u$ is defined by 
\begin{align*}
    \text{error} = \frac{1}{|\mathcal{X}_{n}|} \sum_{x \in \mathcal{X}_{n}} \left|S_h(u,u(x),x)\right|
\end{align*}
where $|\mathcal{X}_{n}|$ denotes the total number of points in $\mathcal{X}_{n}$.

\RestyleAlgo{ruled}
\begin{algorithm}[ht!]
\caption{Implicit iterative method}\label{alg:main}

\textbf{Input:} A point cloud $\mathcal{X}_{n}$ and a function $F_h(p, u, t, x)$ in~\eqref{eq:general-form-s}.

\KwResult{Solution of the scheme $u_h$ up to $O(h)$ error.}
 \vspace{0.2cm}
 \While{error $>$ tolerance}{
    \textbf{For} each $x \in \mathcal{X}_{n}$ \textbf{ do}

    \quad    Use bisection methods to compute $u^{n+1}(x)$ from $t \mapsto S_h(u^{n}, t, x)$.

    \textbf{end}
 }
\end{algorithm}
\noindent In the numerical experiments described in Section~\ref{sec:exp}, we initialize $u^{(0)}$ using the computed solution on a coarser graph. Specifically, we compute the solution on a $\frac{N}{2} \times \frac{N}{2}$ Cartesian grid and use it as an initial guess function to compute the solution on an $N\times N$ grid. In practice, this initialization significantly accelerates the convergence of the algorithm compared to setting $u^{(0)}\equiv 0$. We believe that implementing multigrid-type methods can further improve the algorithm's performance, which we plan to explore in future projects.

% The second method is a type of iterative fast marching method, and is experimentally much faster than the implicit iterative method. The algorithm is identical to fast marching, except that when the value of $u$ at $x$ is finalized, all neighbors in $N_{h}(x)$ are visited and updated. In contrast to fast marching, the solution of the scheme at a point $x$ can potentially depend on neighbors $y \in N_{h}(x)$  with $u(y)>u(x)$. This is because it is possible that $-\Delta_{vv}u(x)>0$ \emph{and} $u(x+v)>u(x)$ or $u(x-v)>u(x)$ (but not both). Therefore, one fast marching iteration is not sufficient. We find, however, that an iterative method based on fast marching converges in far fewer iterations than the implicit iterative method. Each fast marching iteration uses the previous iterate as the initial guess for the next iteration.

% For either method, we say the iterations have converged when the solution satisfies the scheme up to an $O(h)$ error. Since the scheme is monotone, homogeneous (for mean curvature motion), and satisfies a maximum principle, it is possible to show that the resulting solution is within $O(h)$ of the exact solution of the scheme.  

\section{Applications}\label{sec:appl}

In this section, we will construct the monotone convergent schemes for the viscosity solutions of the levelset convex geometric PDEs. In particular, we construct monotone schemes for 
the Tukey depth eikonal equation in~\eqref{eq:tukey_pde},
\begin{equation*}
    |\nabla u(x)| = \int_{(y-x)\cdot \nabla u(x) = 0} \rho(y)\, dS(y) \; \text{ in } \Omega,
\end{equation*}
the mean curvature motion PDE
\begin{equation}\label{eq:MC-eikonal}
\begin{aligned}
    |\nabla u| \kappa &= f \; \text{ in } \Omega\\
    u &= 0 \; \text{ on } \partial\Omega,
\end{aligned}
\end{equation}
and the curvature flow equation
\begin{equation}\label{eq:affine-flow-eikonal}
\begin{aligned}
    |\nabla u| \kappa_+^\alpha &= f \; \text{ in } \Omega\\
    u &= 0 \; \text{ on } \partial\Omega
\end{aligned}
\end{equation}
for $\alpha \in (0,1]$ depending the dimension of the domain $\Omega$.

\subsection{Tukey Depth}\label{subsec:appl-tukey}

From~\eqref{eq:new-H-operator}, define
\begin{equation}\label{eq:H-tukey}
    H(p, u, x) = \frac{p}{|p|} \cdot \nabla u(x) - \int_{(y-x)\cdot p = 0} \rho(y)\, dS(y)
\end{equation}
where $\rho$ is a nonnegative density. In order to establish a monotone convergent scheme for the Tukey depth eikonal equation ~\eqref{eq:tukey_pde}, the task involves demonstrating that $H$ satisfies the Lipschitz condition stated in Proposition~\ref{prop:con} and defining a function $F_h$ that fulfills the requirements outlined in assumptions~\ref{item:F-mon}, \ref{item:F-cts}, and \ref{item:F-H}. 
Then we can easily construct monotone and consistent schemes $S_h$ in~\eqref{eq:general-form-s}. 

To show the nonlocal integral term within the PDE satisfies the Lipschitz condition stated in Proposition~\ref{prop:con}, we assume the data density $\rho$ satisfies the same regularity condition detailed in~\cite{molina2022tukey}, which established the existence of a unique viscosity solution of the Tukey depth eikonal equation.

%\red I thought we needed $\rho$ is Lipschitz, not just uniformly continuous.\nc
\begin{lemma}\label{lem:lip-nabla-phi}
    Suppose a nonnegative density $\rho$ is Lipschitz in an open and bounded support $S \subset \bar\Omega$. Given $\phi\in C^\infty(\mathbb{R}^d)$,  the function $H$ in~\eqref{eq:H-tukey} satisfies
    \begin{equation*}
        |H(p,\phi,x) - H(q,\phi,y)| \leq C\left( \left|\frac{p}{|p|} - \frac{q}{|q|} \right| + |x-y| \right)
    \end{equation*}
    for all $x,y \in \bar\Omega$ and $p,q \in \mathbb{R}^d$, and $C$ is a constant depending on $\rho$ and $\Omega$.
\end{lemma}
\begin{proof}
    Define
    \begin{align*}
        H_1(p,u,x) &= \frac{p}{|p|} \cdot \nabla u(x) \\
        H_2(p,u,x) &= \int_{(y-x)\cdot p = 0} \rho(y)\, dS(y).
    \end{align*}

    \noindent First, we will show $H_1$ is Lipschitz. Choose $x\in\Omega$ and $p,q \in \mathbb{R}^d$. Then
    \begin{align*}
        |H_1(p,\phi,x) - H_1(q,\phi,x)|
        % = \left| \frac{p}{|p|} \cdot \nabla \phi(x) -\frac{q}{|q|} \cdot \nabla \phi(x) \right|
        \leq \left| \frac{p}{|p|}  - \frac{q}{|q|}  \right| \max_{z\in\Omega}|\nabla \phi(z)|.
    \end{align*}
    Choose $x,y\in\Omega$ and $p \in \mathbb{R}^d$. Then
    \begin{align*}
        |H_1(p,\phi,x) - H_1(p,\phi,y)|
        % = \left| \frac{p}{|p|} \cdot \nabla \phi(x) -\frac{q}{|q|} \cdot \nabla \phi(x) \right|
        \leq \left| \nabla \phi(x) - \nabla \phi(y) \right| \leq | x-y | \max_{z\in\Omega}|\nabla^2 \phi(z)|.
    \end{align*}
    Thus, $H_1$ satisfies~\eqref{eq:condition-lipschitz-f}.

    To show $H_2$ is Lipschitz, first fix $p \in \mathbb{R}^d$. For all $x,y \in \Omega$,
    \begin{align*}
        |H_2(p,u,x) - H_2(p,u,y)| 
        &= \left| \int_{(z-y)\cdot p=0} \rho(z+(x-y))\,dS(z) - \int_{(z-y)\cdot p=0} \rho(z)\,dS(z) \right|\\
        &\leq \int_{(z-y)\cdot p=0} \left| \rho(z+(x-y)) - \rho(z) \right| \,dS(z)\\
        &\leq C|x-y| \int_{(z-y)\cdot p=0} \mathds{1}_{\Omega}(z) \,dS(z)\\
        &\leq C |x-y|
    \end{align*}
    where the second ineqaulity uses $\rho$ being Lipschitz and $\mathds{1}_\Omega$ is an indicator function on $\Omega$ and $C$ is a constant depending on $\rho$ and $\Omega$.

    Next, we fix $x \in \Omega$. Choose any $p,q \in \mathbb{R}^d$ and define $p' = \frac{p^\perp}{|p^\perp|}$ and $q' = \frac{q^\perp}{|q^\perp|}$. Then, using the change of variables,
    \begin{align*}
        |H_2(p,u,x) - H_2(q,u,x)| 
         &= \left| \int^\infty_{-\infty} \rho(x + p' t)\,dt - \int^\infty_{-\infty} \rho(x + q't)\,dt \right|\\
         &\leq \int^\infty_{-\infty} \left| \rho(x + p' t)- \rho(x + q't) \right| \,dt\\
         &\leq C |p' - q'| \int^\infty_{-\infty} t \mathds{1}_\Omega(t) \,dt\\
         &\leq C|p' - q'|
         = C \left|\frac{p}{|p|} - \frac{q}{|q|} \right|.
    \end{align*}
    where the second inequality use $\rho$ being Lipschitz. Again, $C$ is a constant depending on $\rho$ and $\Omega$. This proves the lemma.
%    \red These last lines use that $\rho$ is Lipschitz on $\R^d$. It's still interesting to ask whether this works if $\rho$ is Lipschitz on $\Omega$, but then can have a discontinuity along $\partial\Omega$. That would require strong convexity of $\Omega$. \nc
    
\end{proof}

Note that the condition stated in Lemma~\ref{lem:lip-nabla-phi} requires $\rho$ to be Lipschitz continuous within an open and bounded support in the domain. However, it is worth noting that this condition can be relaxed to some extent. In Section~\ref{subsec:tukey-depth}, we illustrate, through numerical examples, that the proposed numerical scheme is capable of approximating the solution when the density $\rho$ is not Lipschitz in $\bar\Omega$.

\subsection{Curvature motion}\label{sec:curvature-motion}

We present our monotone schemes in the simple setting of curvature motion of a convex curve in the plane. This is described by the eikonal equation
\begin{equation}\label{eq:MC}
\left\{ 
\begin{aligned}
|\nabla u|\kappa &= 1&&\text{in } \Omega\\
u &= 0&&\text{on } \partial \Omega,
\end{aligned}
\right.
\end{equation}
where $\Omega\subset \R^2$ is a convex and bounded set, $\partial \Omega$ is the initial curve, and $\kappa(x)$ is the curvature of the level set of $u$ passing through $x$, which is given by $\kappa(x)=-\text{div}(\nabla u/|\nabla u|)$. In this setting, the level sets $\{u=t\}$ evolve with normal velocity $\vb{v}=\kappa$. Since the initial curve $\partial \Omega$ is convex, all the super-level sets $\{u \geq t\}$ of $u$ will be convex, hence $u$ is quasiconcave.

The eikonal equation \eqref{eq:MC} has a particularly simple form, since we can formally expand $\kappa$ to find
\begin{equation}\label{eq:MC2}
{-u_{\eta \eta}=|\nabla u|\kappa = 1,}
\end{equation}
where $\eta = \frac{\nabla u^\perp}{|\nabla u|}$ is a unit vector orthogonal to $\nabla u$, and $u_{\eta \eta} = \eta \cdot  \nabla^2 u\,\eta$. Hence, the problem boils down to constructing a monotone scheme for the pure second derivative $u_{\eta \eta}$. If the direction $\eta$ in \eqref{eq:MC2} was fixed and did not depend on $\nabla u$, then the problem would be simple. The difficulty is that $\eta$ depends on $\nabla u$.

%Since the super level sets of the solution $u$ of \eqref{eq:MC} are convex, it is reasonable to ask that $P_h^-$ is nonempty for the numerical solution $u_h$. This is basically a quasiconcavity constraint. 
We can directly apply our subdifferential $P^-_h[u](x)$ in this setting. For $p \in V_{h}(x)$ we define $p^\perp :=(-p_2,p_1)$. The vector $p^\perp$ plays the role of $\eta$ from \eqref{eq:MC2}.\footnote{We assume our stencil $N_{h}(x)$ is chosen symmetrically, so that $p \in V_{h}(x)$ if and only if $p^\perp \in V_{h}(x)$.} Our scheme for \eqref{eq:MC} is 
\begin{equation}\label{eq:Sm}
{\max_{p \in P^-_h[u](x)} -\Delta^h_{p^\perp p^\perp}u_h(x) = 1 \ \ \text{ for  } x \in \mathcal{X}_{n},}
\end{equation}
where $\Delta^h_{qq}$ is defined as
\[\Delta^h_{qq}u(x):= \frac{u(x+q) - 2u(x) + u(x-q)}{|q|^2}\]
The main idea is that we replaced the \emph{selection} of the direction $\eta=\nabla u^\perp$ with the maximum over the subdifferential $P^-_h[u](x)$.
%A general principle is that monotone, consistent and stable schemes converge to the viscosity solution of (P), provided the latter are unique~\cite{barles1991convergence}.
It is easy to see that for $p \in V_{h}(x)$ the negative of the second order finite difference, i.e., $-\Delta^h_{qq}$ is a monotone scheme, and so it follows from Proposition \ref{prop:Pm} that \eqref{eq:Sm} is \emph{monotone}. 

The schemes we consider here are wide stencil schemes, inspired by schemes for degenerate elliptic equations such as the Monge-Amp\`ere equation~\cite{oberman2008wide}. %For $x \in \X_h$, let $N_{h}(x)$ denote the collection of neighboring grid points to $x$, and let $V_{h}$ denote the vectors $y-x$ for $y\in N_{h}(x)$ corresponding to unique directions from $x$ to neighbors. In other words, the coordinates of $h^{-1}p$ for any $p \in V_{h}$ are relatively prime. 
In particular, we take the stencil $V_{h}$ to be independent of $x$. See Figure \ref{fig:demo} for a depiction of the direction set $V_{h}$ for the standard 9 and 25 point stencils.
\begin{figure}
\centering
\hfill
\subfloat[9 points]{\includegraphics[width=0.2\textwidth]{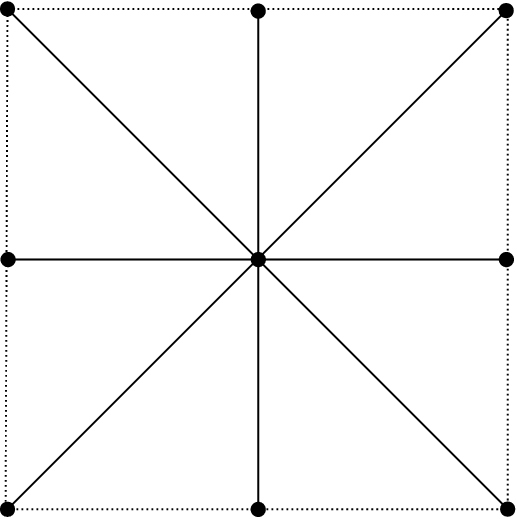}}
\hfill
\subfloat[25 points]{\includegraphics[width=0.4\textwidth]{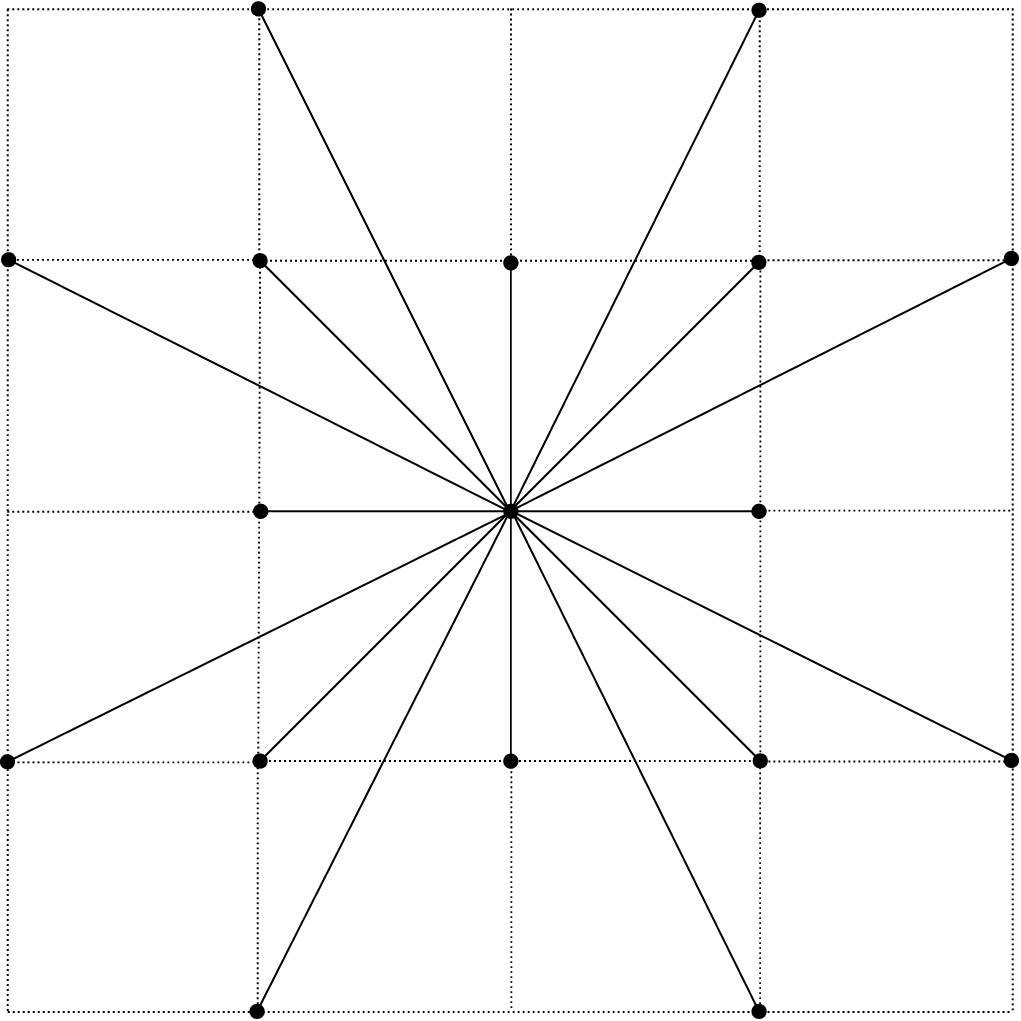}}
\hfill
\caption{Depiction of the available directions $V_{h}$ in the 9 and 25 point stencils.}
\label{fig:demo}
\end{figure}

\begin{remark}
We can easily extend the scheme to motion by a power $\alpha \in (0,1]$ of mean curvature:
\begin{equation}\label{eq:p-alpha}
\left\{ \begin{aligned}
|\nabla u|\kappa_+^\alpha &= f&&\text{in } \Omega\\
u &= 0&&\text{on } \partial \Omega.
\end{aligned}\right.
\end{equation}
The reformulated equation corresponding to \eqref{eq:MC2} becomes
\[|\nabla u|^{1-\alpha} (-u_{\eta\eta})^\alpha_+ = f \ \ \text{in } \Omega\]
and the corresponding scheme is
\begin{equation}\label{eq:scheme-curvature-motion}
\max_{p \in P_h^-[u](x)} |\nabla_p u(x)|^{1-\alpha}(-\Delta_{p^\perp p^\perp}u(x))^\alpha_+ = f(x) \ \  \text{ for }  \ \ x\in \X_n,
\end{equation}
where $\nabla_p u$ is any monotone discretization  gradient, in this case
\[\nabla_p u(x):= \frac{u(x) - u(x-p)}{|p|}.\]
% { \color{cyan}
% where $|\nabla_p u|$ is any monotone discretization of the norm of the gradient, such as
% \[|\nabla_p u(x)|^2:=\sum_{i=1}^2 \max\{\nabla_i^-u(x),-\nabla_i^+u(x),0\}^2.\]
% }
\end{remark}
\begin{remark}
We can furthermore extend the scheme to certain functions of mean curvature, namely
\[ \left\{ \begin{aligned}
|\nabla u|g(\kappa_+) &= f&&\text{in } \Omega\\
u &= 0&&\text{on } \partial \Omega,
\end{aligned}\right.\]
where $g:[0,\infty)\to [0,\infty)$ is increasing and satisfies
\begin{equation}\label{eq:gcond}
g'(s) \leq s^{-1}g(s) \ \ \text{ for all } s>0.
\end{equation}
The corresponding scheme is
\[|\nabla_p u(x)|\max_{p \in P_h^-[u](x)} g\left(\frac{(-\Delta_{p^\perp p^\perp}u(x))_+}{|\nabla_p u(x)|}\right) = f(x) \ \  \text{ for }  \ \ x\in \X_n.\]
The condition \eqref{eq:gcond} ensures that $s \mapsto sg(t/s)$ is increasing for all $t\geq 0$, so that the scheme is monotone. This requirement is satisfied by $g(s)=s^\alpha$ for $0 < \alpha \leq 1$, but also by other monotone functions, such as
\[g(s) = \left(\log\left(\frac{1}{s} + e\right)\right)^{-1}.\]
\end{remark}

We now establish consistency of the above schemes.
% \red The function $f$ below also seems not to be defined.\nc
\begin{lemma}\label{lem:lip-curvature}
    Define a function
    \begin{equation*}
        H(p,\phi,x) = \frac{p \cdot \nabla^2 \phi(x) \, p}{|p|^2}.
    \end{equation*}
    Then, given $\phi\in C^\infty(\mathbb{R}^d)$, the function $H$ satisfies
    the Lipschitz condition stated in Proposition~\ref{prop:con}.
\end{lemma}
\begin{proof}
    Choose $x\in\bar\Omega$ and $p,q \in \mathbb{R}^d$. Then
    \begin{align*}
        |H(p,\phi,x) - H(q,\phi,x)|
        &\leq \left| \frac{p \cdot \nabla^2 \phi(x) \, p}{|p|^2} - \frac{p \cdot \nabla^2 \phi(x) \, q}{|p| |q|} \right|
        + \left| \frac{p \cdot \nabla^2 \phi(x) \, q}{|p| |q|} - \frac{q \cdot \nabla^2 \phi(x) \, q}{|q|^2} \right|\\
        &\leq 2 \left|\frac{p}{|p|} - \frac{q}{|q|} \right| \max_{z\in\bar\Omega} |\nabla^2 \phi(z)|.
    \end{align*}
    Choose $x,y\in\bar\Omega$ and $p \in \mathbb{R}^d$. Then
    \begin{align*}
        |H(p,\phi,x) - H(p,\phi,y)|
        &\leq \left| \frac{p \cdot \left(\nabla^2 \phi(x) - \nabla^2 \phi(y) \right) \, p}{|p|^2} \right|
        \leq |x-y| \max_{z\in\bar\Omega} |\nabla^3\phi(z)|.
    \end{align*}
    This concludes the lemma.
\end{proof}

\subsection{Extensions to higher dimensions}
\label{sec:hd}

We briefly discuss here how the schemes naturally extend to higher dimensions. Consider $d=3$.  We can formulate the scheme to solve motion by mean curvature, Gauss curvature, or more general functions of the principal curvatures. For mean curvature, we wish to solve
\[\left\{ \begin{aligned}
|\nabla u|\kappa_M &= f&&\text{in } \Omega\\
u &= 0&&\text{on } \partial \Omega,
\end{aligned}\right.\]
where $\kappa_M(x)$ is the mean curvature of the level surface of $u$ passing through $x$, given by
\[\kappa_M(x) = -\text{div}\left(\frac{\nabla u}{|\nabla u|}\right) = \frac{u_{\xi\xi} - \Delta u}{|\nabla u|},\]
where $\xi = \nabla u/|\nabla u|$. If $\eta_1,\eta_2$ is any orthonormal basis for $\xi^\perp$, we can write
\[\Delta u = u_{\xi\xi} + u_{\eta_1\eta_1} + u_{\eta_2\eta_2},\]
and therefore we can write $\kappa_M$ as 
\begin{equation}\label{eq:H}
\kappa_M(x) = -\frac{u_{\eta_1\eta_1} + u_{\eta_2\eta_2}}{|\nabla u|}.
\end{equation}
This allows us to rewrite the equation as
\[-(u_{\eta_1\eta_1} + u_{\eta_2\eta_2}) = f \ \ \text{in } \Omega,\]
and the corresponding scheme would be
\[\max_{p \in P_h^-[u](x)} (-\Delta_{v_1(p) v_1(p)}u(x) - \Delta_{v_2(p)v_2(p)}u(x)) = f(x) \ \ \text{ for } \ \ x\in \X_h,\]
where $v_1(p),v_2(p)\in V_{h}$ are an orthonormal basis for $p^\perp$. The Laplacian is rotationally invariant, so the choice of $v_1(p)$, and $v_2(p)$ is not important. 

The affine flow in higher dimensions corresponds to motion of a surface with velocity proportional to $\kappa_G^\frac{1}{d+1}$ where $\kappa_G$ denotes Gauss curvature. Since $d=3$, we wish to solve
\[\left\{ \begin{aligned}
|\nabla u|\kappa_G^\frac{1}{4} &= f&&\text{in } \Omega\\
u &= 0&&\text{on } \partial \Omega.
\end{aligned}\right.\]
We can write Gauss curvature in the level set formulation as
\[\kappa_G = \frac{\nabla u\cdot \cof(-\nabla^2u)\nabla u}{|\nabla u|^4}.\]
If $O$ is any orthogonal transformation such that $O\nabla u(x)=|\nabla u(x)|e_3$, then we have
\[\kappa_G = \frac{O\nabla u \cdot \cof(-O\nabla^2uO^T)O\nabla u\rangle}{|\nabla u|^4} = \frac{\det([-O\nabla^2uO^T]_{33})}{|\nabla u|^2},\]
where $[A]_{33}$ denotes the (3,3)-minor of the matrix $A$.
This is similar to the Monge-Amp\`ere equation restricted to the space orthogonal to $\nabla u$. We can use Hadamard's determinant identity, as was used for Monge-Amp\`ere in~\cite{froese2011convergent}, to write
\[\kappa_G = \min_{\{v_1,v_2\}} \frac{(-u_{v_1v_2})_+(-u_{v_2v_2})_+}{|\nabla u|^2},\]
where the minimum is over all orthonormal bases $\{v_1,v_2\}$ of $\nabla u^\perp$. The corresponding monotone discretization scheme is
\[\max_{p \in P_h^-[u](x)} |\nabla_p u|^\frac{1}{2}\min_{\{v_1,v_2\} \in p^\perp} (-\Delta_{v_1 v_1}u)^\frac{1}{4}_+(-\Delta_{v_2v_2}u)^\frac{1}{4}_+ = f(x) \ \ \text{ for } \ \ x\in \X_h,\]
where $p^\perp$ denotes the collection of orthonormal bases of the space orthogonal to $p$ consisting of vectors $v_1,v_2 \in V_{h}$.

\section{Numerical implementation and experiments}\label{sec:exp}

In this section, we present numerical results using the proposed wide stencil finite difference scheme (Algorithm~\ref{alg:main}) to solve Hamilton-Jacobi equations in various settings. Throughout this section we will assume that the domain $\Omega = [0,1]^d$ is the unit square in $\mathbb{R}^d$. The numerical simulations in this section were coded in C++ and Python and were run on a 2019 MacBook Pro with a $2.6$ GHz 6-core processor and $16$ GB RAM. The first set of experiments (Section~\ref{subsec:eikonal}) computes the solutions of a simple eikonal equation on unstructured point clouds in $\mathbb{R}^2$ and $\mathbb{R}^3$ with various boundary conditions. The second set of experiments (Section~\ref{subsec:affine-flows}) computes the solution of the affine flow on regular rectangular grids in $\mathbb{R}^2$ with various boundary conditions. The third set of experiments (Section~\ref{subsec:tukey-depth}) computes the solutions of the Tukey depth eikonal equation on unstructured point clouds. Lastly, in the third set of experiments (Section~\ref{subsec:mnist}), we use the proposed algorithm to compute the Tukey depth measure on more complex dataset such as MNIST \cite{lecun1998gradient} and Fasion-MNIST dataset \cite{xiao2017fashion}.

\subsection{Eikonal equation}\label{subsec:eikonal}
In this set of experiments, we use Algorithm~\ref{alg:main} to solve the simple eikonal equation on unstructured point clouds in $\mathbb{R}^d$.
\begin{equation}\label{eq:eikonal-exp}
\left\{
\begin{aligned}
    |\nabla u| &= f \quad \text{in } \Omega\\
    u &= 0 \quad \text{on } \partial\Omega,
\end{aligned}
\right.
\end{equation}
where the function $f$ is an indicator function on a set $E \subset \Omega$ such that $f=1$ on $E$ and $0$ otherwise. 
We consider three different shapes for $E$: (1) the box, (2) a rotated ellipse, and (3) two disjoint balls (Figure~\ref{fig:exp1-f}).

\begin{figure}[!ht]
    \centering
    \begin{subfigure}[t]{0.25\textwidth}
        \centering
        \includegraphics[width=\textwidth]{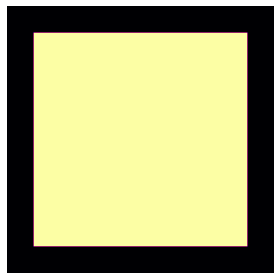}
        \caption{Square}
    \end{subfigure}
    \hfill
    \begin{subfigure}[t]{0.25\textwidth}
        \centering
        \includegraphics[width=\textwidth]{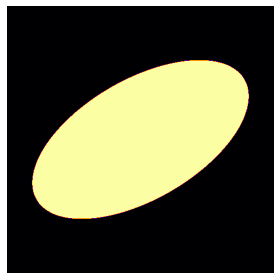}
        \caption{Ellipse}
    \end{subfigure}
    \hfill
    \begin{subfigure}[t]{0.25\textwidth}
        \centering
        \includegraphics[width=\textwidth]{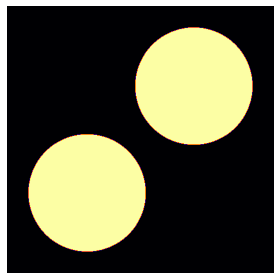}
        \caption{Two balls}
    \end{subfigure}
    \hfill
    \caption{The indicator function $f$ with three different shapes for the set $E$ in~\eqref{eq:eikonal-exp}. Black pixels and bright pixels indicate $0$ and $1$, respectively.}
    \label{fig:exp1-f}
\end{figure}

Let $n$ be the number of points in the unstructured point cloud $\mathcal{X}_{n} \subset \Omega$. We define the set of neighbors $N_h(x)$ for each $x\in\mathcal{X}_{n}$ by constructing $k$-Euclidean distance nearest neighbor ($k$NN) graphs from $\mathcal{X}_{n}$  with $k=20$ (where $k$ represents the number of neighbors). The numerical scheme to solve the PDE is
\[
    S_h(u,u(x),x) = \begin{cases}
       \displaystyle \max_{p \in P^-_h[u](x)} \nabla_p u(x) - f(x) & \text{if } P^-_h[u](x) \neq \emptyset\\
        -\infty & \text{otherwise}
    \end{cases}
\]
which is proven to be monotone and consistent in the preceding sections.
Given an initial guess $u^{(0)} \equiv 0$, 
 use Algorithm~\ref{alg:main} to iterate the algorithm to compute the solution of~\eqref{eq:eikonal-exp} on $\mathcal{X}_{n}$ until the convergence. The experiment was repeated for two different dimensions ($d=2,3$) and different number of points ($n=1000,2000,\cdots,16000$). The computation time and the total number of iterations to compute the solutions are displayed in Table~\ref{tab:time-1}. Figure~\ref{fig:exp1} shows the computed solutions on $\mathbb{R}^2$ with $8000$ data points. Note that the algorithm converged fastest on the two balls domain and slowest on the square domain, showing that the rate of convergence depends on the convexity of the domain. This is due to the fact that the scheme $S_h$ tests whether the subdifferential set is empty or not which is equivalent to testing whether the function is locally strictly quasiconcave at $x$. In a square domain, the point $x$ near the flat surface of the square requires a stricter condition on the angular resolution $d\theta(x)$ so that the subdifferential set is nonempty. The emptiness of subdifferential sets near flat surface could slow down the convergence of the algorithm.  We note that there are many faster numerical methods for solving the eikonal equation, such as fast marching \cite{sethian1999fast,sethian1996fast} and fast sweeping \cite{zhao2005fast}. The point of these experiments is just to illustrate our methods and their computational complexity on simple equations.
 % We may improve the scheme on such non-strictly quasiconcave areas by considering
 % \[
 %    S_h(u,u(x),x) = \begin{cases}
 %        \max_{p \in P^-_h[u](x)} F_h(p,u,u(x),x) & \text{if } P^-_h[u](x) \neq \emptyset\\
 %        G_h(u,u(x),x) & \text{otherwise}
 %    \end{cases}
 % \]
 % where $G_h$ is a monotone function that works on $x$

\begin{table}[!ht]
\centering
\begin{tabular}{|c|c|c|c|c|c|c|c|}
 \hline
& &\multicolumn{2}{c|}{Square} & \multicolumn{2}{c|}{Ellipse} & \multicolumn{2}{c|}{Two balls}\\
\hline
$d$ & $n$ & Iterations & Time & Iterations& Time & Iterations& Time\\
\hline
\multirow{3}{*}{2} & 1000    & 28  & 0.25s & 23  & 0.25s & 19  & 0.17s\\
                   & 2000    & 38  & 0.66s & 26  & 0.47s & 24  & 0.40s\\
                   & 4000    & 49  & 1.61s & 37  & 1.31s & 33  & 1.18s\\
                   & 8000    & 70  & 4.92s & 47  & 3.30s & 39  & 2.39s\\
 \hline
\multirow{3}{*}{3} & 4000    & 46  & 1.57s & 29  & 1.00s & 23  & 0.81s\\
                   & 8000    & 62  & 3.93s & 29  & 1.90s & 30  & 2.04s\\
                   & 16000   & 72  & 9.21s & 44  & 5.88s & 20  & 2.70s\\
\hline
\end{tabular}
\caption{The number of iterations and computation time (Section~\ref{subsec:eikonal}).}
\label{tab:time-1}
\end{table}

\begin{figure}[!ht]
    \centering
    \begin{subfigure}[t]{0.32\textwidth}
        \centering
        \includegraphics[width=\textwidth]{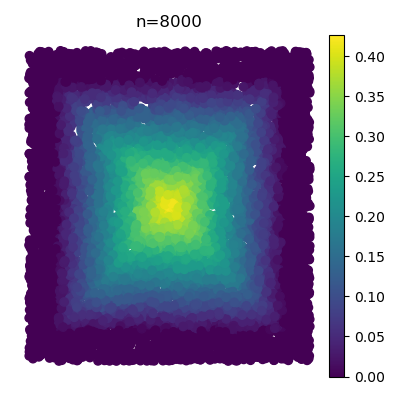}
        % \caption{}
    \end{subfigure}
    \hfill
    \begin{subfigure}[t]{0.32\textwidth}
        \centering
        \includegraphics[width=\textwidth]{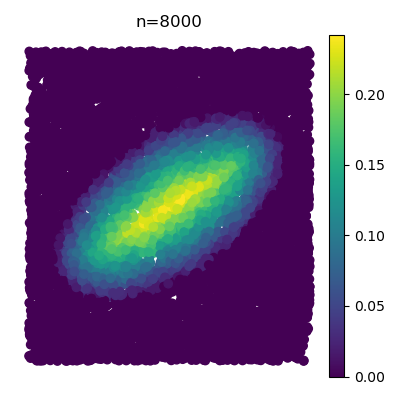}
        % \caption{}
    \end{subfigure}
    \hfill
    \begin{subfigure}[t]{0.32\textwidth}
        \centering
        \includegraphics[width=\textwidth]{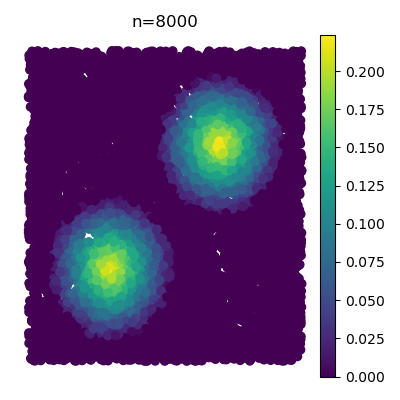}
        % \caption{}
    \end{subfigure}
    \hfill
    \caption{Computed solutions of the eikonal equation~\eqref{eq:eikonal-exp} on an unstructured point cloud in $\mathbb{R}^2$ with 8000 data points. The first image shows the result from a square domain, the second image shows from an ellipse domain, and the last image shows from two balls domain.}
    \label{fig:exp1}
\end{figure}

\subsection{Curvature motion PDEs}\label{subsec:affine-flows}

In this set of experiments, we use Algorithm~\ref{alg:main} to solve curvature motion PDEs on 2D and 3D Cartesian grids. First, we consider the affine flows in the 2D domain $\Omega=[0,1]^2$
\begin{equation}\label{eq:exp-affine-flow}
\left\{
    \begin{aligned}
        |\nabla u| \kappa^{1/3}_+ &= f \quad \text{in } \Omega\\
        u &= 0 \quad \text{on } \partial \Omega,
    \end{aligned}
\right.
\end{equation}
which corresponds to~\eqref{eq:p-alpha} with $\alpha = 1/3$. The function $f$ is chosen as the indicator function of the square, ellipse, and two balls domains, as in the preceding experiment.
As it was noted in Section~\ref{sec:curvature-motion}, the wide stencil scheme for the affine flow requires symmetry of the point cloud. Thus, we compute the solutions on a Cartesian grid with a $7\times7=49$ point stencil. We considered 3 different shapes as in the preceding experiment: (1) the box, (2) a rotated ellipse,  and (3) two disjoint balls. 

We employed the convergent numerical scheme $S_h$ in~\eqref{eq:scheme-curvature-motion} to compute the viscosity solutions of~\eqref{eq:exp-affine-flow} on grids of dimensions $32\times32$, $64\times64$, and $128\times128$.  The contour plots of the solutions for the box, the ellipse, and two balls are shown in Figure~\ref{fig:exp2}. It should be noted that the solution of affine flows is unique only when $f>0$ and is not unique when $f\geq0$. We provide examples of nonunique solutions in Figures~\ref{fig:exp2-twoballs-0} and~\ref{fig:exp2-twoballs-1}, where Figure~\ref{fig:exp2-twoballs-0} shows the computed solution with the initial guess function $u^{(0)}\equiv 0$ and Figure~\ref{fig:exp2-twoballs-1} shows the computed solution with $u^{(0)}\equiv 1$. The quantitative results of the experiments are presented in Table~\ref{tab:affine-flow-result-1}.

Next, we consider the mean curvature PDE in 3D domain $\Omega=[0,1]^3$ given by 
\begin{equation}\label{eq:MC-exp}
    \left\{ 
    \begin{aligned}
    |\nabla u|\kappa &= f&&\text{in } \Omega\\
    u &= 0&&\text{on } \partial \Omega.
    \end{aligned}
    \right.
\end{equation}
Again, the function $f$ serves as an indicator function. In this experiment, we discretize the domain using a grid of size $50\times 50 \times 50$. The solution of the PDE was computed employing Algorithm~\ref{alg:main} with stencils of size $7\times 7\times 7$. Figure~\ref{fig:exp-curvature-pde-3d} presents two computed solutions with two different $f$. The left plot illustrates the numerical solution of the PDE for the case where $f=1$ everywhere in $\Omega$, while the right plot depicts the solution where $f=1$ in two separate spherical regions centered at $(0.3,0.3,0.3)$ and $(0.7,0.7,0.7)$, each with a radius of $0.3$.

\begin{figure}[!ht]
    \centering
    \begin{subfigure}[t]{0.24\textwidth}
        \centering
        \includegraphics[width=\textwidth]{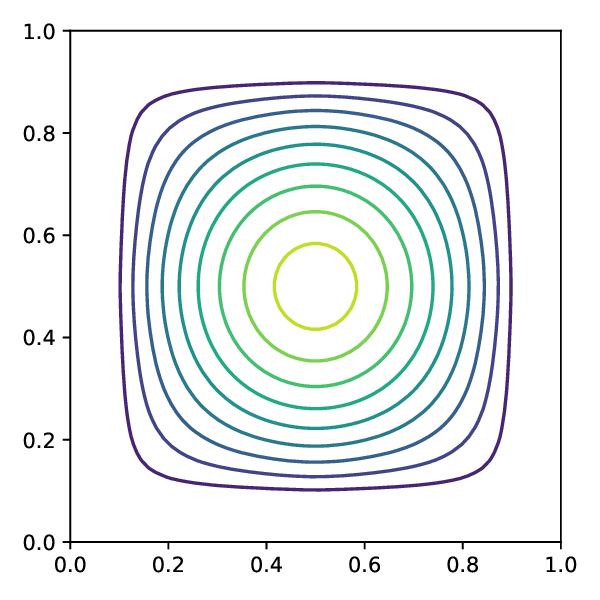}
        \caption{Square}
    \end{subfigure}
    \hfill
    \begin{subfigure}[t]{0.24\textwidth}
        \centering
        \includegraphics[width=\textwidth]{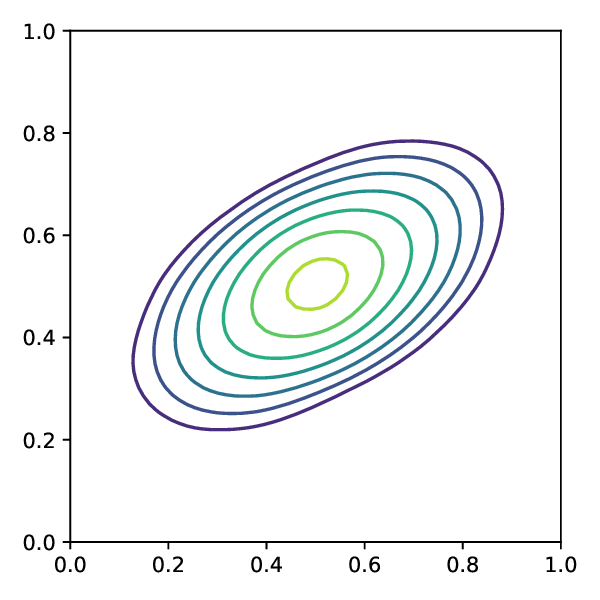}
        \caption{Ellipse}
    \end{subfigure}
    \hfill
    \begin{subfigure}[t]{0.24\textwidth}
        \centering
        \includegraphics[width=\textwidth]{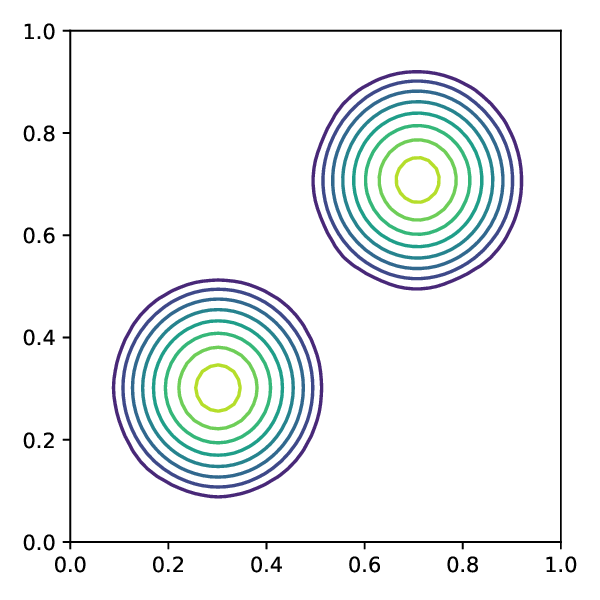}
        \caption{Two balls ($u^{(0)}\equiv 0$)}
        \label{fig:exp2-twoballs-0}
    \end{subfigure}
    \hfill
    \begin{subfigure}[t]{0.24\textwidth}
        \centering
        \includegraphics[width=\textwidth]{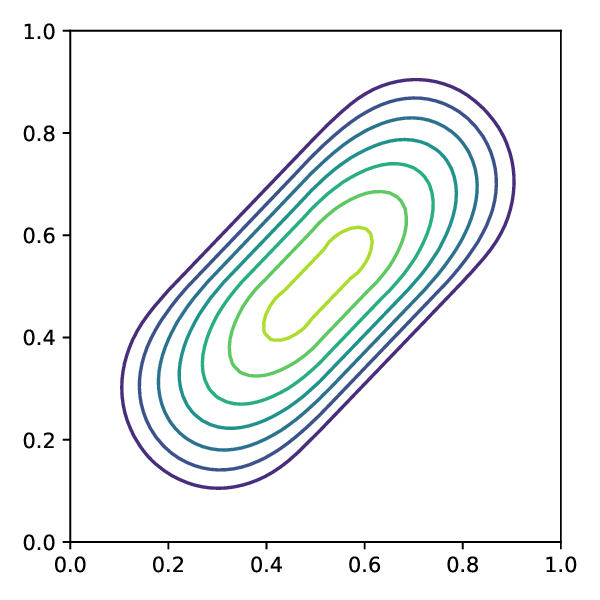}
        \caption{Two balls ($u^{(0)}\equiv 1$)}
        \label{fig:exp2-twoballs-1}
    \end{subfigure}
    \hfill
    \caption{Computed solutions of the affine flows~\eqref{eq:exp-affine-flow} with $\alpha = 1/3$ on $128\times 128$ grid with $7\times7$ stencils. Figures~\ref{fig:exp2-twoballs-0} and ~\ref{fig:exp2-twoballs-1} show two different solutions given different initial guess function $u^{(0)}$.}
    \label{fig:exp2}
\end{figure}

% \begin{table}[th!]
% \centering
% \begin{tabular}{|c|c|c|}
% \hline
% Domain & $\#$ of Iterations & CPU Time\\
% \hline
% Square         & 176  & 2.93s \\
% Ellipse        & 127  & 2.44s \\
% Two balls ($u^{(0)}\equiv 0$)      & 62   & 1.13s \\
% Two balls ($u^{(0)}\equiv 1$)      & 494  & 6.93s \\
% \hline
% \end{tabular}
% \caption{Computation time and the total number of iterations for affine flows on $64\times 64$ grid with $7\times7$ stencils.}
% \label{tab:affine-flow-result-1}
% \end{table}

% \begin{table}[th!]
% \centering
% \begin{tabular}{|c|c|c|c|c|c|}
% \hline
% \multirow{2}{*}{Domain} & \multirow{2}{*}{Error tolerance} & \multicolumn{3}{c|}{Grid size}\\ \cline{3-5}
%  & & $32\times 32$ & $64\times 64$ & $128\times 128$\\
% \hline
% Square                             & $5\times 10^{-3}$ & 0.19s  & 1.41s & 15.45s \\
% Ellipse                            & $3\times 10^{-3}$ & 0.14s  & 1.10s & 7.26s \\
% Two balls ($u^{(0)}\equiv 0$)      & $3\times 10^{-3}$ & 0.10s  & 0.54s & 3.87s \\
% Two balls ($u^{(0)}\equiv 1$)      & $3\times 10^{-3}$ & 0.66s  & 1.33s & 10.62s \\
% \hline
% \end{tabular}
% \caption{Computation time for affine flows on various grids with $7\times7$ stencils.}
% \label{tab:affine-flow-result-1}
% \end{table}

\begin{table}[th!]
\centering
\begin{tabular}{|c|c|c|c|c|c|}
\hline
\multirow{2}{*}{Domain} & \multirow{2}{*}{Error tolerance} & \multicolumn{3}{c|}{Grid size}\\ \cline{3-5}
 & & $32\times 32$ & $64\times 64$ & $128\times 128$\\
\hline
Square                             & $5\times 10^{-3}$ & 0.19s  & 1.41s & 15.45s \\
Ellipse                            & $3\times 10^{-3}$ & 0.14s  & 1.10s & 7.26s \\
Two balls ($u^{(0)}\equiv 0$)      & $3\times 10^{-3}$ & 0.10s  & 0.54s & 3.87s \\
Two balls ($u^{(0)}\equiv 1$)      & $3\times 10^{-3}$ & 0.66s  & 1.33s & 10.62s \\
\hline
\end{tabular}
\caption{Computation time for affine flows on various grids with $7\times7$ stencils.}
\label{tab:affine-flow-result-1}
\end{table}

\begin{figure}[!ht]
    \centering
    \begin{subfigure}[t]{0.45\textwidth}
        \centering
        \includegraphics[width=\textwidth]{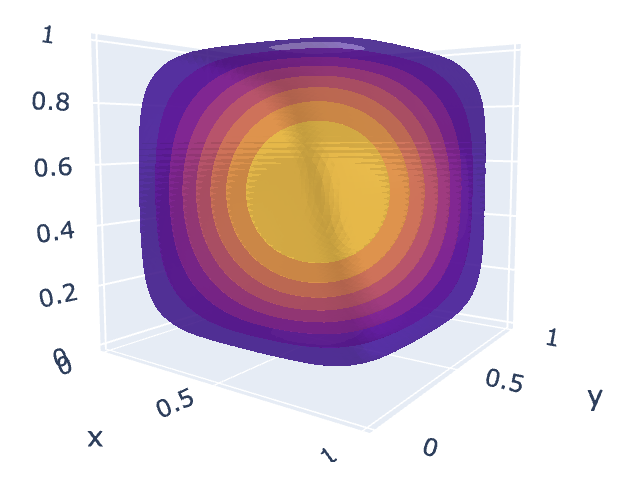}
        \caption{Cube}
    \end{subfigure}
    % \hfill
    \begin{subfigure}[t]{0.45\textwidth}
        \centering
        \includegraphics[width=\textwidth]{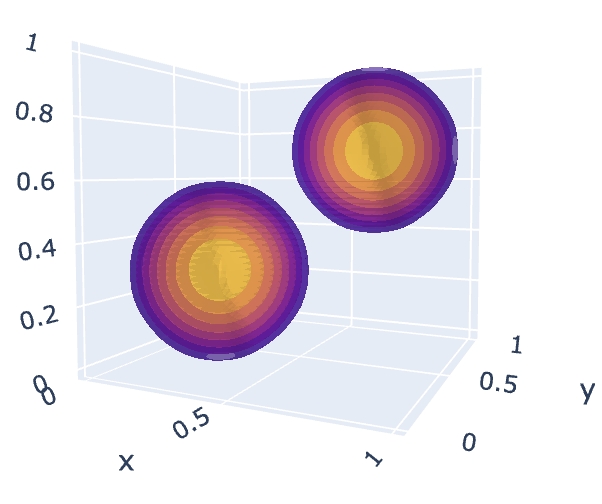}
        \caption{Two balls}
    \end{subfigure}
    \hfill
    \caption{Computed solutions of the mean curvature PDE given by equation~\eqref{eq:MC-exp} in a 3D domain $\Omega = [0,1]^3$. The left plot illustrates the numerical solution of the PDE for the case where $f=1$ everywhere in $\Omega$, while the right plot depicts the solution where $f=1$ in two separate spherical regions.}
    \label{fig:exp-curvature-pde-3d}
\end{figure}

\subsection{Tukey depth}\label{subsec:tukey-depth}

In this section, we use Algorithm~\ref{alg:main} to compute the viscosity solution of the Tukey depth eikonal equation~\eqref{eq:tukey_pde}
\[
    |\nabla u(x)| - \int_{(y-x)\cdot \nabla u(x) = 0} \rho(y)\, dS(y) = 0, \quad x \in \Omega.
\]
We present two experiments for computing Tukey depth measures. In the first experiment, we consider a Cartesian grid on a domain $\Omega=[0,1]^2$ and $\rho$ is a defined as
\[
    \rho(x) = \begin{cases}
         1 & \text{if } x \in E\\
         0 & \text{otherwise}.
    \end{cases}
\] 
We consider three different shapes for $E \subset \Omega$: a square, a circle, and a donut  (Figure~\ref{fig:tukey-4-dist}). 
\begin{figure}[ht!]
    \centering
    \begin{subfigure}[b]{0.25\textwidth}
        \centering
        \includegraphics[width=\textwidth]{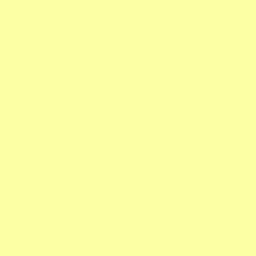}
        \caption{Square}
    \end{subfigure}
    % \hfill
    \begin{subfigure}[b]{0.25\textwidth}
        \centering
        \includegraphics[width=\textwidth]{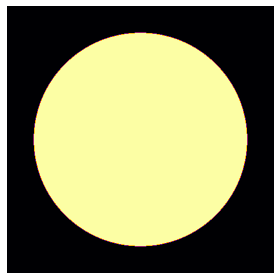}
        \caption{Circle}
    \end{subfigure}
    % \hfill
    \begin{subfigure}[b]{0.25\textwidth}
        \centering
        \includegraphics[width=\textwidth]{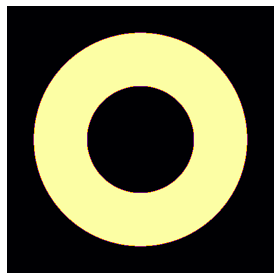}
        \caption{Donut}
    \end{subfigure}
    % \hfill
    \caption{Three different shapes for $\rho$ considered in Section~\ref{subsec:tukey-depth}.}
    \label{fig:tukey-4-dist}
\end{figure}
In this experiment, instead of the usual wide stencil used in the preceding experiment, a different approach for the wide stencil scheme was implemented.

Given a point $x_0 \in \mathcal{X}_{n}$ and a displacement vector $p \in V_h(x_0)$, the nonlocal integral term from the PDE is approximated by
\[
    \int_{(y-x_0)\cdot p = 0} \rho(y)\, dS(y) = {|p|}\sum_{x_j \in I(x_0,p)} \rho(x_j) + O(h)
\]
where $I(x_0,p)$ contains points in $\mathcal{X}_{n}$ along the line with a slope of $p$ passing through $x_0$ (Figure~\ref{fig:illustration-indices-set}).
\begin{figure}[th!]
     \centering
\tikzset{every picture/.style={line width=0.5pt}} %set default line width to 0.75pt        

\begin{tikzpicture}[x=0.75pt,y=0.75pt,yscale=-0.8,xscale=0.8]
%uncomment if require: \path (0,270); %set diagram left start at 0, and has height of 270

%Shape: Grid [id:dp9334453870016493] 
\draw  [draw opacity=0][dash pattern={on 0.84pt off 2.51pt}] (10,12) -- (350,12) -- (350,250) -- (10,250) -- cycle ; \draw  [dash pattern={on 0.84pt off 2.51pt}] (10,12) -- (10,250)(45,12) -- (45,250)(80,12) -- (80,250)(115,12) -- (115,250)(150,12) -- (150,250)(185,12) -- (185,250)(220,12) -- (220,250)(255,12) -- (255,250)(290,12) -- (290,250)(325,12) -- (325,250) ; \draw  [dash pattern={on 0.84pt off 2.51pt}] (10,12) -- (350,12)(10,47) -- (350,47)(10,82) -- (350,82)(10,117) -- (350,117)(10,152) -- (350,152)(10,187) -- (350,187)(10,222) -- (350,222) ; \draw  [dash pattern={on 0.84pt off 2.51pt}]  ;
%Straight Lines [id:da016378774580324218] 
\draw    (2,25) -- (360,205) ;
%Shape: Circle [id:dp7422366960361954] 
\draw  [fill={rgb, 255:red, 0; green, 0; blue, 0 }  ,fill opacity=1 ] (40,47) .. controls (40,44.24) and (42.24,42) .. (45,42) .. controls (47.76,42) and (50,44.24) .. (50,47) .. controls (50,49.76) and (47.76,52) .. (45,52) .. controls (42.24,52) and (40,49.76) .. (40,47) -- cycle ;
%Shape: Circle [id:dp37968498127729433] 
\draw  [fill={rgb, 255:red, 0; green, 0; blue, 0 }  ,fill opacity=1 ] (110,82) .. controls (110,79.24) and (112.24,77) .. (115,77) .. controls (117.76,77) and (120,79.24) .. (120,82) .. controls (120,84.76) and (117.76,87) .. (115,87) .. controls (112.24,87) and (110,84.76) .. (110,82) -- cycle ;
%Shape: Circle [id:dp6475255171624642] 
\draw  [fill={rgb, 255:red, 0; green, 0; blue, 0 }  ,fill opacity=1 ] (181,118) .. controls (181,115.79) and (182.79,114) .. (185,114) .. controls (187.21,114) and (189,115.79) .. (189,118) .. controls (189,120.21) and (187.21,122) .. (185,122) .. controls (182.79,122) and (181,120.21) .. (181,118) -- cycle ;
%Shape: Circle [id:dp002179254023645383] 
\draw  [fill={rgb, 255:red, 0; green, 0; blue, 0 }  ,fill opacity=1 ] (250,152) .. controls (250,149.24) and (252.24,147) .. (255,147) .. controls (257.76,147) and (260,149.24) .. (260,152) .. controls (260,154.76) and (257.76,157) .. (255,157) .. controls (252.24,157) and (250,154.76) .. (250,152) -- cycle ;
%Shape: Circle [id:dp5053801997669256] 
\draw  [fill={rgb, 255:red, 0; green, 0; blue, 0 }  ,fill opacity=1 ] (320,187) .. controls (320,184.24) and (322.24,182) .. (325,182) .. controls (327.76,182) and (330,184.24) .. (330,187) .. controls (330,189.76) and (327.76,192) .. (325,192) .. controls (322.24,192) and (320,189.76) .. (320,187) -- cycle ;
%Straight Lines [id:da7630890470325973] 
\draw [line width=1.5]    (185,117) -- (252.32,150.66) ;
\draw [shift={(255,152)}, rotate = 206.57] [color={rgb, 255:red, 0; green, 0; blue, 0 }  ][line width=1.5]    (14.21,-4.28) .. controls (9.04,-1.82) and (4.3,-0.39) .. (0,0) .. controls (4.3,0.39) and (9.04,1.82) .. (14.21,4.28)   ;

% Text Node
\draw (181,93.4) node [anchor=north west][inner sep=0.75pt]    {$x_{0}$};
% Text Node
\draw (222,110.4) node [anchor=north west][inner sep=0.75pt]    {$p \in V_h(x_0)$};
\end{tikzpicture}
     \caption{Illustration of the set $I(x_0,p) \subset \mathcal{X}_{n}$. Given a point $x_i \in \mathcal{X}_{n}$ and a displacement vector $p\in V_h(x_0)$ the set $I(x_0,p)$ includes all the points (represented as dots in the figure) that the line with a slope $p$ passes through including $x_0$. }
     \label{fig:illustration-indices-set}
 \end{figure}
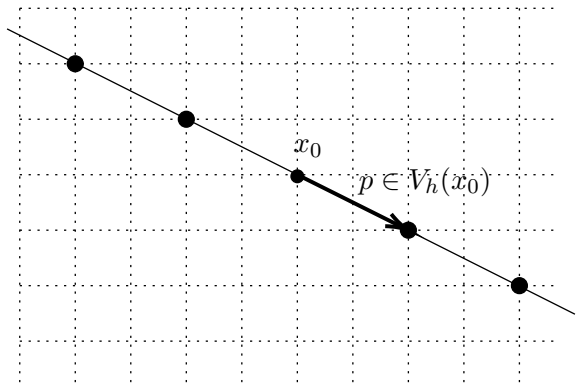 

 In this experiment, a different stencil scheme is utilized compared to the one described in Figure~\ref{fig:demo}. When considering a grid point $x_0$, instead of selecting points solely from the grid, points are chosen from the linear interpolation of the available points within a $3\times 3$ grid surrounding the center point $x_0$ (Figure~\ref{fig:tukey-square-stencils}). The advantage of adopting this stencil scheme lies in its capability to reduce the directional resolution $d\theta$, while preserving the length of the displacement vectors in $V_h(x_0)$. It's important to note that these interpolated points are approximated using a first-order approximation. Consequently, this stencil scheme is only applicable to first-order PDEs.

 \begin{figure}[ht!]
    \centering
    \begin{subfigure}[t]{0.32\textwidth}
        \centering
\tikzset{every picture/.style={line width=0.4pt}} %set default line width to 0.75pt        

\begin{tikzpicture}[x=0.75pt,y=0.75pt,yscale=-0.4,xscale=0.4]
%uncomment if require: \path (0,439); %set diagram left start at 0, and has height of 439

%Shape: Grid [id:dp8792405903493652] 
\draw  [draw opacity=0][dash pattern={on 0.84pt off 2.51pt}] (20,30) -- (421.27,30) -- (421.27,431.27) -- (20,431.27) -- cycle ; \draw  [dash pattern={on 0.84pt off 2.51pt}] (220.64,30) -- (220.64,431.27) ; \draw  [dash pattern={on 0.84pt off 2.51pt}] (20,230.64) -- (421.27,230.64) ; \draw  [dash pattern={on 0.84pt off 2.51pt}] (20,30) -- (421.27,30) -- (421.27,431.27) -- (20,431.27) -- cycle ;
%Straight Lines [id:da5304590517614232] 
\draw    (220.64,230.64) -- (418.41,230.64) ;
\draw [shift={(418.41,230.64)}, rotate = 0] [color={rgb, 255:red, 0; green, 0; blue, 0 }  ][fill={rgb, 255:red, 0; green, 0; blue, 0 }  ][line width=0.75]      (0, 0) circle [x radius= 3.35, y radius= 3.35]   ;
%Straight Lines [id:da12492954184943217] 
\draw    (20,230.64) -- (220.64,230.64) ;
\draw [shift={(220.64,230.64)}, rotate = 0] [color={rgb, 255:red, 0; green, 0; blue, 0 }  ][fill={rgb, 255:red, 0; green, 0; blue, 0 }  ][line width=0.75]      (0, 0) circle [x radius= 3.35, y radius= 3.35]   ;
\draw [shift={(20,230.64)}, rotate = 0] [color={rgb, 255:red, 0; green, 0; blue, 0 }  ][fill={rgb, 255:red, 0; green, 0; blue, 0 }  ][line width=0.75]      (0, 0) circle [x radius= 3.35, y radius= 3.35]   ;
%Straight Lines [id:da2929897468524899] 
\draw    (220.64,230.64) -- (220.64,428.41) ;
\draw [shift={(220.64,428.41)}, rotate = 90] [color={rgb, 255:red, 0; green, 0; blue, 0 }  ][fill={rgb, 255:red, 0; green, 0; blue, 0 }  ][line width=0.75]      (0, 0) circle [x radius= 3.35, y radius= 3.35]   ;
%Straight Lines [id:da3246322044493134] 
\draw    (220.64,230.64) -- (220.64,30) ;
\draw [shift={(220.64,30)}, rotate = 270] [color={rgb, 255:red, 0; green, 0; blue, 0 }  ][fill={rgb, 255:red, 0; green, 0; blue, 0 }  ][line width=0.75]      (0, 0) circle [x radius= 3.35, y radius= 3.35]   ;
%Straight Lines [id:da48770783311985055] 
\draw    (220.64,230.64) -- (421.27,30) ;
\draw [shift={(421.27,30)}, rotate = 315] [color={rgb, 255:red, 0; green, 0; blue, 0 }  ][fill={rgb, 255:red, 0; green, 0; blue, 0 }  ][line width=0.75]      (0, 0) circle [x radius= 3.35, y radius= 3.35]   ;
%Straight Lines [id:da17551191492893692] 
\draw    (220.64,230.64) -- (421.27,431.27) ;
\draw [shift={(421.27,431.27)}, rotate = 45] [color={rgb, 255:red, 0; green, 0; blue, 0 }  ][fill={rgb, 255:red, 0; green, 0; blue, 0 }  ][line width=0.75]      (0, 0) circle [x radius= 3.35, y radius= 3.35]   ;
%Straight Lines [id:da09234344301655939] 
\draw    (220.64,230.64) -- (20,431.27) ;
\draw [shift={(20,431.27)}, rotate = 135] [color={rgb, 255:red, 0; green, 0; blue, 0 }  ][fill={rgb, 255:red, 0; green, 0; blue, 0 }  ][line width=0.75]      (0, 0) circle [x radius= 3.35, y radius= 3.35]   ;
%Straight Lines [id:da7141911906455681] 
\draw    (220.64,230.64) -- (20,30) ;
\draw [shift={(20,30)}, rotate = 225] [color={rgb, 255:red, 0; green, 0; blue, 0 }  ][fill={rgb, 255:red, 0; green, 0; blue, 0 }  ][line width=0.75]      (0, 0) circle [x radius= 3.35, y radius= 3.35]   ;
\draw [shift={(220.64,230.64)}, rotate = 225] [color={rgb, 255:red, 0; green, 0; blue, 0 }  ][fill={rgb, 255:red, 0; green, 0; blue, 0 }  ][line width=0.75]      (0, 0) circle [x radius= 3.35, y radius= 3.35]   ;
%Straight Lines [id:da45998995782010277] 
\draw    (220.64,230.64) -- (420,130) ;
\draw [shift={(420,130)}, rotate = 333.22] [color={rgb, 255:red, 0; green, 0; blue, 0 }  ][fill={rgb, 255:red, 0; green, 0; blue, 0 }  ][line width=0.75]      (0, 0) circle [x radius= 3.35, y radius= 3.35]   ;
%Straight Lines [id:da09598392195095085] 
\draw    (220.64,230.64) -- (320,30) ;
\draw [shift={(320,30)}, rotate = 296.35] [color={rgb, 255:red, 0; green, 0; blue, 0 }  ][fill={rgb, 255:red, 0; green, 0; blue, 0 }  ][line width=0.75]      (0, 0) circle [x radius= 3.35, y radius= 3.35]   ;
%Straight Lines [id:da07741685650815144] 
\draw    (220,230) -- (120,30) ;
\draw [shift={(120,30)}, rotate = 243.43] [color={rgb, 255:red, 0; green, 0; blue, 0 }  ][fill={rgb, 255:red, 0; green, 0; blue, 0 }  ][line width=0.75]      (0, 0) circle [x radius= 3.35, y radius= 3.35]   ;
%Straight Lines [id:da5324213178194492] 
\draw    (220.64,230.64) -- (20,130) ;
\draw [shift={(20,130)}, rotate = 206.64] [color={rgb, 255:red, 0; green, 0; blue, 0 }  ][fill={rgb, 255:red, 0; green, 0; blue, 0 }  ][line width=0.75]      (0, 0) circle [x radius= 3.35, y radius= 3.35]   ;
%Straight Lines [id:da3819854577317867] 
\draw    (220.64,230.64) -- (20,330) ;
\draw [shift={(20,330)}, rotate = 153.65] [color={rgb, 255:red, 0; green, 0; blue, 0 }  ][fill={rgb, 255:red, 0; green, 0; blue, 0 }  ][line width=0.75]      (0, 0) circle [x radius= 3.35, y radius= 3.35]   ;
%Straight Lines [id:da49861813910439934] 
\draw    (220.64,230.64) -- (120,430) ;
\draw [shift={(120,430)}, rotate = 116.78] [color={rgb, 255:red, 0; green, 0; blue, 0 }  ][fill={rgb, 255:red, 0; green, 0; blue, 0 }  ][line width=0.75]      (0, 0) circle [x radius= 3.35, y radius= 3.35]   ;
%Straight Lines [id:da4221176840036751] 
\draw    (220.64,230.64) -- (320,430) ;
\draw [shift={(320,430)}, rotate = 63.51] [color={rgb, 255:red, 0; green, 0; blue, 0 }  ][fill={rgb, 255:red, 0; green, 0; blue, 0 }  ][line width=0.75]      (0, 0) circle [x radius= 3.35, y radius= 3.35]   ;
%Straight Lines [id:da026350255767252362] 
\draw    (220.64,230.64) -- (420,330) ;
\draw [shift={(420,330)}, rotate = 26.49] [color={rgb, 255:red, 0; green, 0; blue, 0 }  ][fill={rgb, 255:red, 0; green, 0; blue, 0 }  ][line width=0.75]      (0, 0) circle [x radius= 3.35, y radius= 3.35]   ;

\end{tikzpicture}
        \caption{$k=16$}
    \end{subfigure}
    \hspace{1cm}
    \begin{subfigure}[t]{0.32\textwidth}
        \centering

\tikzset{every picture/.style={line width=0.4pt}} %set default line width to 0.75pt        

\begin{tikzpicture}[x=0.75pt,y=0.75pt,yscale=-0.4,xscale=0.4]
%uncomment if require: \path (0,439); %set diagram left start at 0, and has height of 439

%Shape: Grid [id:dp8792405903493652] 
\draw  [draw opacity=0][dash pattern={on 0.84pt off 2.51pt}] (20,30) -- (421.27,30) -- (421.27,431.27) -- (20,431.27) -- cycle ; \draw  [dash pattern={on 0.84pt off 2.51pt}] (220.64,30) -- (220.64,431.27) ; \draw  [dash pattern={on 0.84pt off 2.51pt}] (20,230.64) -- (421.27,230.64) ; \draw  [dash pattern={on 0.84pt off 2.51pt}] (20,30) -- (421.27,30) -- (421.27,431.27) -- (20,431.27) -- cycle ;
%Straight Lines [id:da5304590517614232] 
\draw    (220.64,230.64) -- (421.27,230.64) ;
\draw [shift={(421.27,230.64)}, rotate = 0] [color={rgb, 255:red, 0; green, 0; blue, 0 }  ][fill={rgb, 255:red, 0; green, 0; blue, 0 }  ][line width=0.75]      (0, 0) circle [x radius= 3.35, y radius= 3.35]   ;
%Straight Lines [id:da12492954184943217] 
\draw    (20,230.64) -- (220.64,230.64) ;
\draw [shift={(220.64,230.64)}, rotate = 0] [color={rgb, 255:red, 0; green, 0; blue, 0 }  ][fill={rgb, 255:red, 0; green, 0; blue, 0 }  ][line width=0.75]      (0, 0) circle [x radius= 3.35, y radius= 3.35]   ;
\draw [shift={(20,230.64)}, rotate = 0] [color={rgb, 255:red, 0; green, 0; blue, 0 }  ][fill={rgb, 255:red, 0; green, 0; blue, 0 }  ][line width=0.75]      (0, 0) circle [x radius= 3.35, y radius= 3.35]   ;
%Straight Lines [id:da2929897468524899] 
\draw    (220.64,230.64) -- (220.64,428.41) ;
\draw [shift={(220.64,428.41)}, rotate = 90] [color={rgb, 255:red, 0; green, 0; blue, 0 }  ][fill={rgb, 255:red, 0; green, 0; blue, 0 }  ][line width=0.75]      (0, 0) circle [x radius= 3.35, y radius= 3.35]   ;
%Straight Lines [id:da3246322044493134] 
\draw    (220.64,230.64) -- (220.64,30) ;
\draw [shift={(220.64,30)}, rotate = 270] [color={rgb, 255:red, 0; green, 0; blue, 0 }  ][fill={rgb, 255:red, 0; green, 0; blue, 0 }  ][line width=0.75]      (0, 0) circle [x radius= 3.35, y radius= 3.35]   ;
%Straight Lines [id:da48770783311985055] 
\draw    (220.64,230.64) -- (421.27,30) ;
\draw [shift={(421.27,30)}, rotate = 315] [color={rgb, 255:red, 0; green, 0; blue, 0 }  ][fill={rgb, 255:red, 0; green, 0; blue, 0 }  ][line width=0.75]      (0, 0) circle [x radius= 3.35, y radius= 3.35]   ;
%Straight Lines [id:da17551191492893692] 
\draw    (220.64,230.64) -- (421.27,431.27) ;
\draw [shift={(421.27,431.27)}, rotate = 45] [color={rgb, 255:red, 0; green, 0; blue, 0 }  ][fill={rgb, 255:red, 0; green, 0; blue, 0 }  ][line width=0.75]      (0, 0) circle [x radius= 3.35, y radius= 3.35]   ;
%Straight Lines [id:da09234344301655939] 
\draw    (220.64,230.64) -- (20,431.27) ;
\draw [shift={(20,431.27)}, rotate = 135] [color={rgb, 255:red, 0; green, 0; blue, 0 }  ][fill={rgb, 255:red, 0; green, 0; blue, 0 }  ][line width=0.75]      (0, 0) circle [x radius= 3.35, y radius= 3.35]   ;
%Straight Lines [id:da7141911906455681] 
\draw    (220.64,230.64) -- (20,30) ;
\draw [shift={(20,30)}, rotate = 225] [color={rgb, 255:red, 0; green, 0; blue, 0 }  ][fill={rgb, 255:red, 0; green, 0; blue, 0 }  ][line width=0.75]      (0, 0) circle [x radius= 3.35, y radius= 3.35]   ;
\draw [shift={(220.64,230.64)}, rotate = 225] [color={rgb, 255:red, 0; green, 0; blue, 0 }  ][fill={rgb, 255:red, 0; green, 0; blue, 0 }  ][line width=0.75]      (0, 0) circle [x radius= 3.35, y radius= 3.35]   ;
%Straight Lines [id:da45998995782010277] 
\draw    (220.64,230.64) -- (420,130) ;
\draw [shift={(420,130)}, rotate = 333.22] [color={rgb, 255:red, 0; green, 0; blue, 0 }  ][fill={rgb, 255:red, 0; green, 0; blue, 0 }  ][line width=0.75]      (0, 0) circle [x radius= 3.35, y radius= 3.35]   ;
%Straight Lines [id:da09598392195095085] 
\draw    (220.64,230.64) -- (320,30) ;
\draw [shift={(320,30)}, rotate = 296.35] [color={rgb, 255:red, 0; green, 0; blue, 0 }  ][fill={rgb, 255:red, 0; green, 0; blue, 0 }  ][line width=0.75]      (0, 0) circle [x radius= 3.35, y radius= 3.35]   ;
%Straight Lines [id:da07741685650815144] 
\draw    (220,230) -- (120,30) ;
\draw [shift={(120,30)}, rotate = 243.43] [color={rgb, 255:red, 0; green, 0; blue, 0 }  ][fill={rgb, 255:red, 0; green, 0; blue, 0 }  ][line width=0.75]      (0, 0) circle [x radius= 3.35, y radius= 3.35]   ;
%Straight Lines [id:da5324213178194492] 
\draw    (220.64,230.64) -- (20,130) ;
\draw [shift={(20,130)}, rotate = 206.64] [color={rgb, 255:red, 0; green, 0; blue, 0 }  ][fill={rgb, 255:red, 0; green, 0; blue, 0 }  ][line width=0.75]      (0, 0) circle [x radius= 3.35, y radius= 3.35]   ;
%Straight Lines [id:da3819854577317867] 
\draw    (220.64,230.64) -- (20,330) ;
\draw [shift={(20,330)}, rotate = 153.65] [color={rgb, 255:red, 0; green, 0; blue, 0 }  ][fill={rgb, 255:red, 0; green, 0; blue, 0 }  ][line width=0.75]      (0, 0) circle [x radius= 3.35, y radius= 3.35]   ;
%Straight Lines [id:da49861813910439934] 
\draw    (220.64,230.64) -- (120,430) ;
\draw [shift={(120,430)}, rotate = 116.78] [color={rgb, 255:red, 0; green, 0; blue, 0 }  ][fill={rgb, 255:red, 0; green, 0; blue, 0 }  ][line width=0.75]      (0, 0) circle [x radius= 3.35, y radius= 3.35]   ;
%Straight Lines [id:da4221176840036751] 
\draw    (220.64,230.64) -- (320,430) ;
\draw [shift={(320,430)}, rotate = 63.51] [color={rgb, 255:red, 0; green, 0; blue, 0 }  ][fill={rgb, 255:red, 0; green, 0; blue, 0 }  ][line width=0.75]      (0, 0) circle [x radius= 3.35, y radius= 3.35]   ;
%Straight Lines [id:da026350255767252362] 
\draw    (220.64,230.64) -- (420,330) ;
\draw [shift={(420,330)}, rotate = 26.49] [color={rgb, 255:red, 0; green, 0; blue, 0 }  ][fill={rgb, 255:red, 0; green, 0; blue, 0 }  ][line width=0.75]      (0, 0) circle [x radius= 3.35, y radius= 3.35]   ;
%Straight Lines [id:da4405112739498519] 
\draw    (220.64,230.64) -- (420,180) ;
\draw [shift={(420,180)}, rotate = 345.75] [color={rgb, 255:red, 0; green, 0; blue, 0 }  ][fill={rgb, 255:red, 0; green, 0; blue, 0 }  ][line width=0.75]      (0, 0) circle [x radius= 3.35, y radius= 3.35]   ;
%Straight Lines [id:da7374821388761136] 
\draw    (220.64,230.64) -- (420,80) ;
\draw [shift={(420,80)}, rotate = 322.93] [color={rgb, 255:red, 0; green, 0; blue, 0 }  ][fill={rgb, 255:red, 0; green, 0; blue, 0 }  ][line width=0.75]      (0, 0) circle [x radius= 3.35, y radius= 3.35]   ;
%Straight Lines [id:da05843737777242952] 
\draw    (220.64,230.64) -- (370,30) ;
\draw [shift={(370,30)}, rotate = 306.67] [color={rgb, 255:red, 0; green, 0; blue, 0 }  ][fill={rgb, 255:red, 0; green, 0; blue, 0 }  ][line width=0.75]      (0, 0) circle [x radius= 3.35, y radius= 3.35]   ;
%Straight Lines [id:da6261424421537009] 
\draw    (220,230) -- (270,30) ;
\draw [shift={(270,30)}, rotate = 284.04] [color={rgb, 255:red, 0; green, 0; blue, 0 }  ][fill={rgb, 255:red, 0; green, 0; blue, 0 }  ][line width=0.75]      (0, 0) circle [x radius= 3.35, y radius= 3.35]   ;
%Straight Lines [id:da578402822777891] 
\draw    (220.64,230.64) -- (170,30) ;
\draw [shift={(170,30)}, rotate = 255.84] [color={rgb, 255:red, 0; green, 0; blue, 0 }  ][fill={rgb, 255:red, 0; green, 0; blue, 0 }  ][line width=0.75]      (0, 0) circle [x radius= 3.35, y radius= 3.35]   ;
%Straight Lines [id:da8152779078723323] 
\draw    (220,230) -- (70,30) ;
\draw [shift={(70,30)}, rotate = 233.13] [color={rgb, 255:red, 0; green, 0; blue, 0 }  ][fill={rgb, 255:red, 0; green, 0; blue, 0 }  ][line width=0.75]      (0, 0) circle [x radius= 3.35, y radius= 3.35]   ;
%Straight Lines [id:da21334685890063632] 
\draw    (220.64,230.64) -- (20,80) ;
\draw [shift={(20,80)}, rotate = 216.9] [color={rgb, 255:red, 0; green, 0; blue, 0 }  ][fill={rgb, 255:red, 0; green, 0; blue, 0 }  ][line width=0.75]      (0, 0) circle [x radius= 3.35, y radius= 3.35]   ;
%Straight Lines [id:da6096006353378831] 
\draw    (220.64,230.64) -- (20,180) ;
\draw [shift={(20,180)}, rotate = 194.16] [color={rgb, 255:red, 0; green, 0; blue, 0 }  ][fill={rgb, 255:red, 0; green, 0; blue, 0 }  ][line width=0.75]      (0, 0) circle [x radius= 3.35, y radius= 3.35]   ;
%Straight Lines [id:da6100787873622522] 
\draw    (220.64,230.64) -- (20,280) ;
\draw [shift={(20,280)}, rotate = 166.18] [color={rgb, 255:red, 0; green, 0; blue, 0 }  ][fill={rgb, 255:red, 0; green, 0; blue, 0 }  ][line width=0.75]      (0, 0) circle [x radius= 3.35, y radius= 3.35]   ;
%Straight Lines [id:da6372430331404082] 
\draw    (220,230) -- (20,380) ;
\draw [shift={(20,380)}, rotate = 143.13] [color={rgb, 255:red, 0; green, 0; blue, 0 }  ][fill={rgb, 255:red, 0; green, 0; blue, 0 }  ][line width=0.75]      (0, 0) circle [x radius= 3.35, y radius= 3.35]   ;
%Straight Lines [id:da5891307593542882] 
\draw    (220.64,230.64) -- (70,430) ;
\draw [shift={(70,430)}, rotate = 127.07] [color={rgb, 255:red, 0; green, 0; blue, 0 }  ][fill={rgb, 255:red, 0; green, 0; blue, 0 }  ][line width=0.75]      (0, 0) circle [x radius= 3.35, y radius= 3.35]   ;
%Straight Lines [id:da2823630850314375] 
\draw    (220.64,230.64) -- (170,430) ;
\draw [shift={(170,430)}, rotate = 104.25] [color={rgb, 255:red, 0; green, 0; blue, 0 }  ][fill={rgb, 255:red, 0; green, 0; blue, 0 }  ][line width=0.75]      (0, 0) circle [x radius= 3.35, y radius= 3.35]   ;
%Straight Lines [id:da5533434101224228] 
\draw    (220.64,230.64) -- (270,430) ;
\draw [shift={(270,430)}, rotate = 76.09] [color={rgb, 255:red, 0; green, 0; blue, 0 }  ][fill={rgb, 255:red, 0; green, 0; blue, 0 }  ][line width=0.75]      (0, 0) circle [x radius= 3.35, y radius= 3.35]   ;
%Straight Lines [id:da11112213878303911] 
\draw    (220.64,230.64) -- (370,430) ;
\draw [shift={(370,430)}, rotate = 53.16] [color={rgb, 255:red, 0; green, 0; blue, 0 }  ][fill={rgb, 255:red, 0; green, 0; blue, 0 }  ][line width=0.75]      (0, 0) circle [x radius= 3.35, y radius= 3.35]   ;
%Straight Lines [id:da6040411839969062] 
\draw    (220.64,230.64) -- (420,380) ;
\draw [shift={(420,380)}, rotate = 36.84] [color={rgb, 255:red, 0; green, 0; blue, 0 }  ][fill={rgb, 255:red, 0; green, 0; blue, 0 }  ][line width=0.75]      (0, 0) circle [x radius= 3.35, y radius= 3.35]   ;
%Straight Lines [id:da17754434219982806] 
\draw    (220.64,230.64) -- (420,280) ;
\draw [shift={(420,280)}, rotate = 13.91] [color={rgb, 255:red, 0; green, 0; blue, 0 }  ][fill={rgb, 255:red, 0; green, 0; blue, 0 }  ][line width=0.75]      (0, 0) circle [x radius= 3.35, y radius= 3.35]   ;

\end{tikzpicture}
        \caption{$k=32$}
    \end{subfigure}
    \hfill
    \caption{The stencil scheme used for the first order Hamilton-Jacobi equations. Figures show the number of stencil points (a) $k=16$ and (b) $k=32$. }
    \label{fig:tukey-grids}
\end{figure}
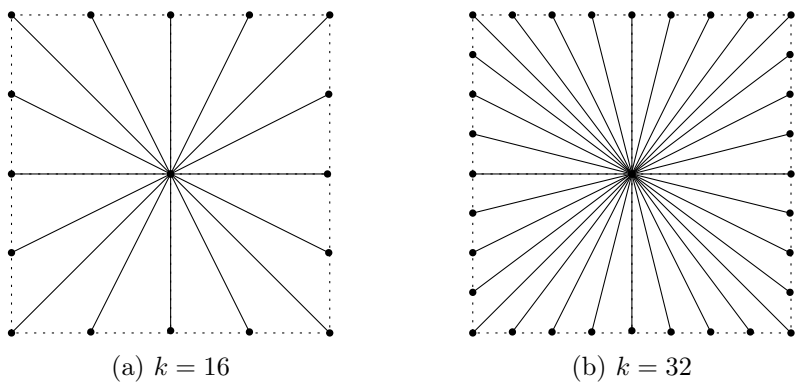

Using Algorithm~\ref{alg:main}, we computed the viscosity solutions of the PDE for each $\rho$, as well as for different pairs of grid mesh and the number of stencil points: ($32\times 32$, $k=16$), ($64\times 64$, $k=32$), and ($96 \times 96$, $k=64$). The analytical solutions for the Tukey depth measure were known, and we display the errors between the computed solutions and analytical solutions in Table~\ref{tab:tukey-result}, along with the elapsed time for computations. Furthermore, Figure~\ref{fig:tukey-128-results} displays the computed solutions on $128\times128$ grids. Note that the highest values of the solution indicate the medians of the density $\rho$. When $\rho$ is a donut, which is not quasiconcave, the computed viscosity solution is quasiconcave, as expected from the analytical solution of the Tukey depth measure.

\begin{table}[th!]
\centering
\begin{tabular}{|c|c|c|c|c|c|c|}
\hline
\multirow{2}{*}{$\rho$}  & \multicolumn{2}{c|}{$32\times 32$, $k = 16$} & \multicolumn{2}{c|}{$64\times 64$, $k=32$} & \multicolumn{2}{c|}{$96\times 96$, $k=48$}\\ \cline{2-3} \cline{4-5} \cline{6-7} 
  & Time & Error & Time & Error & Time & Error\\
\hline
Square                      & 0.26s & $7.66\times 10^{-3}$ & 4.72s & $6.94 \times 10^{-3}$ & 34.98s & $6.29 \times 10^{-3}$ \\
Circle                      & 0.26s & $6.45\times 10^{-2}$ & 4.51s & $2.13 \times 10^{-3}$ & 33.88s & $1.48\times 10^{-3}$\\
Donut                       & 0.33s & $5.15\times 10^{-3}$ & 4.93s & $1.36 \times 10^{-3}$ & 37.69s & $7.84\times 10^{-4}$ \\
\hline
\end{tabular}
\caption{Computation time and errors for Tukey depth eikonal equation on various sizes grids and stencils.}
\label{tab:tukey-result}
\end{table}

\begin{figure}[ht!]
    \centering
    \begin{subfigure}[t]{0.32\textwidth}
        \centering
        \includegraphics[width=\textwidth]{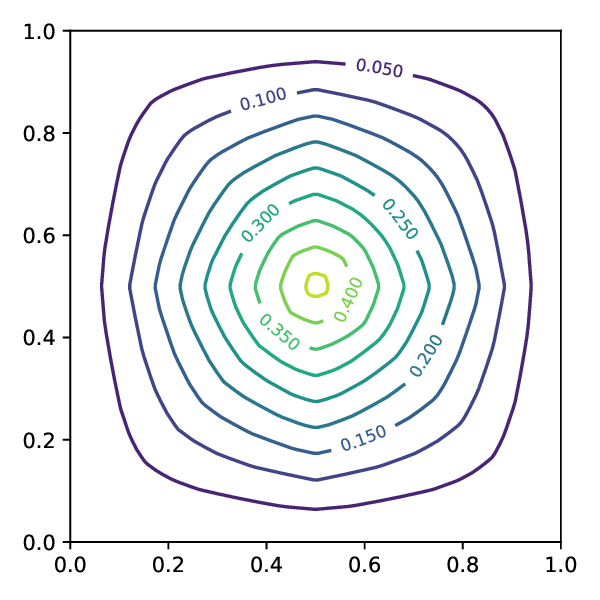}
        \caption{\footnotesize Computed solution on a square}
        %\label{fig:figure}
    \end{subfigure}
    \hfill
    \begin{subfigure}[t]{0.32\textwidth}
        \centering
        \includegraphics[width=\textwidth]{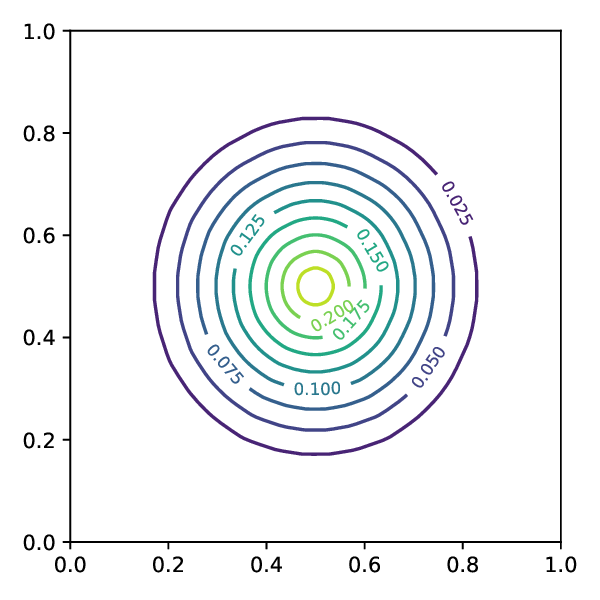}
        \caption{\footnotesize Computed solution on a circle}
        %\label{fig:figure2}
    \end{subfigure}
    \hfill
    \begin{subfigure}[t]{0.32\textwidth}
        \centering
        \includegraphics[width=\textwidth]{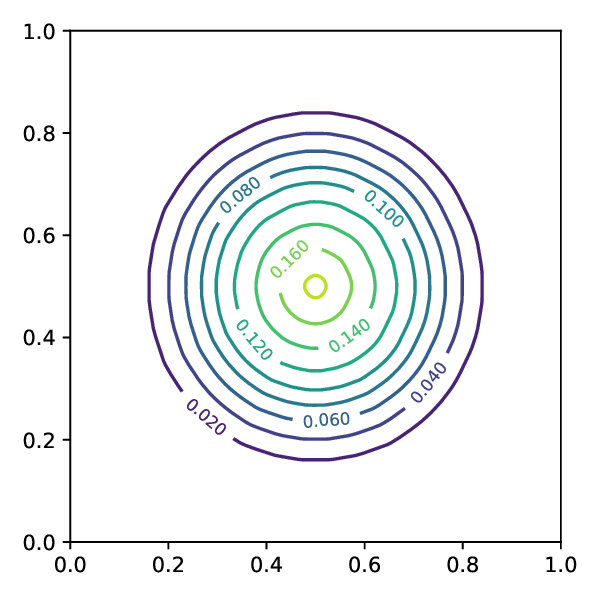}
        \caption{\footnotesize Computed solution on a donut}
        %\label{fig:figure2}
    \end{subfigure}
    \hfill
    \begin{subfigure}[t]{0.32\textwidth}
        \centering
        \includegraphics[width=\textwidth]{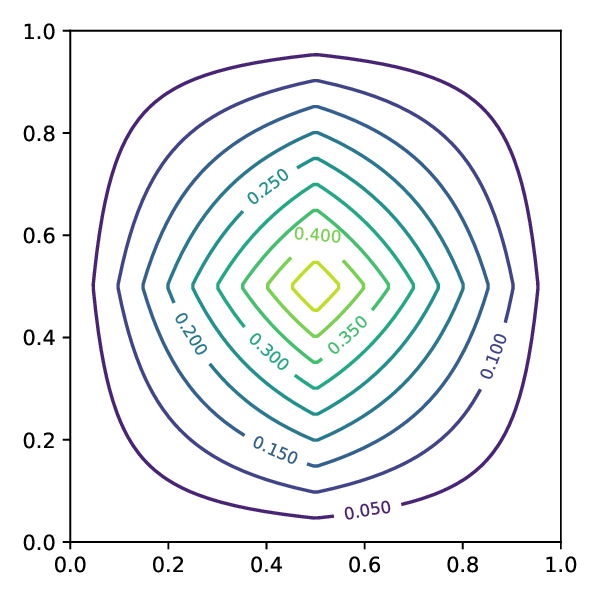}
        \caption{\footnotesize Analytical solution on a square}
        %\label{fig:figure}
    \end{subfigure}
    \hfill
    \begin{subfigure}[t]{0.32\textwidth}
        \centering
        \includegraphics[width=\textwidth]{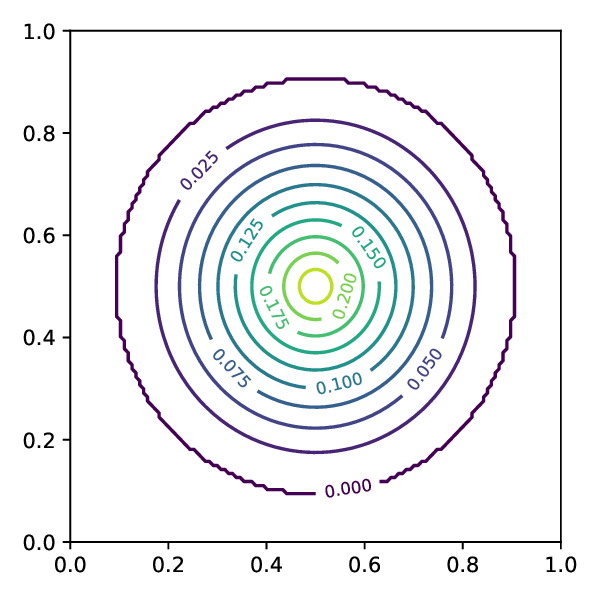}
        \caption{\footnotesize Analytical solution on a circle}
        %\label{fig:figure2}
    \end{subfigure}
    \hfill
    \begin{subfigure}[t]{0.32\textwidth}
        \centering
        \includegraphics[width=\textwidth]{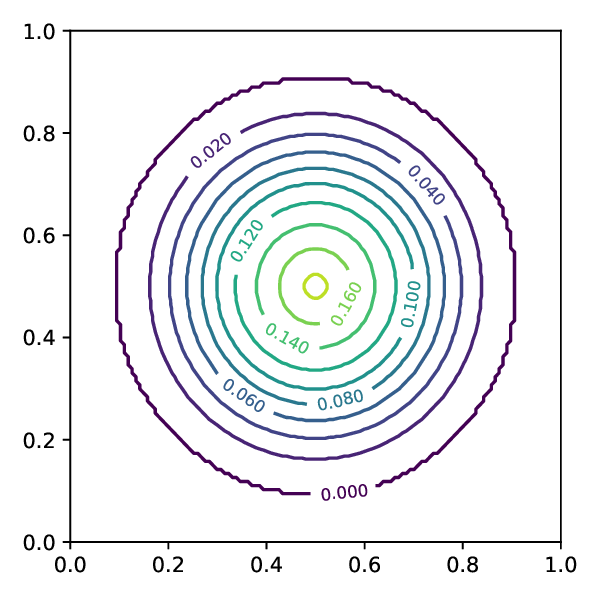}
        \caption{\footnotesize Analytical solution on a donut}
        %\label{fig:figure2}
    \end{subfigure}
    \hfill
    \caption{Computed results and analytical solutions of the Tukey depth eikonal equation.}
    \label{fig:tukey-128-results}
\end{figure}

Note that in Figure~\ref{fig:tukey-128-results}, the computed solution for the square density differs noticeably from the analytical solution. The level sets of the analytical solution are squares near the center, while the level sets of the computed solution resemble octagons. This computation can be improved by increasing the number of stencils, or in other words, by reducing the value of $d\theta$. Figure~\ref{fig:tukey-square-stencils} illustrates the computed solutions on a $512\times 512$ domain using the number of stencil points $k=8, 48, 240$. As evident, the level sets of the computed solutions tend to become more square-like as the number of stencils increases.

\begin{figure}[ht!]
    \centering
    \begin{subfigure}[t]{0.24\textwidth}
        \centering
        \includegraphics[width=\textwidth]{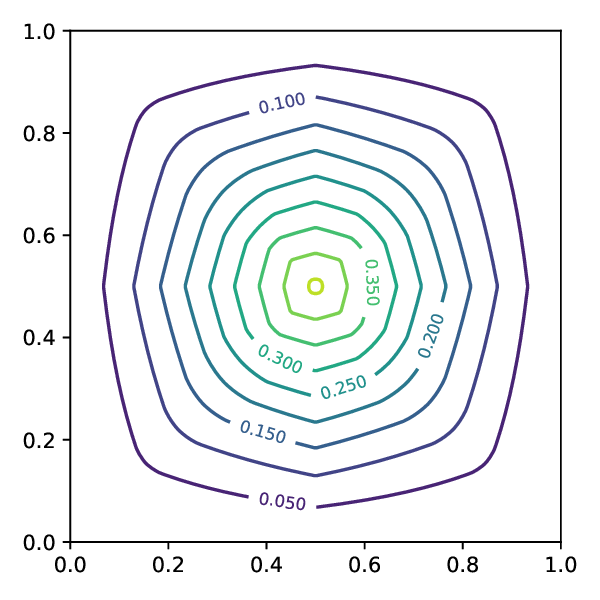}
        \caption{\footnotesize $16$ stencil points}
        %\label{fig:figure}
    \end{subfigure}
    \hfill
    \begin{subfigure}[t]{0.24\textwidth}
        \centering
        \includegraphics[width=\textwidth]{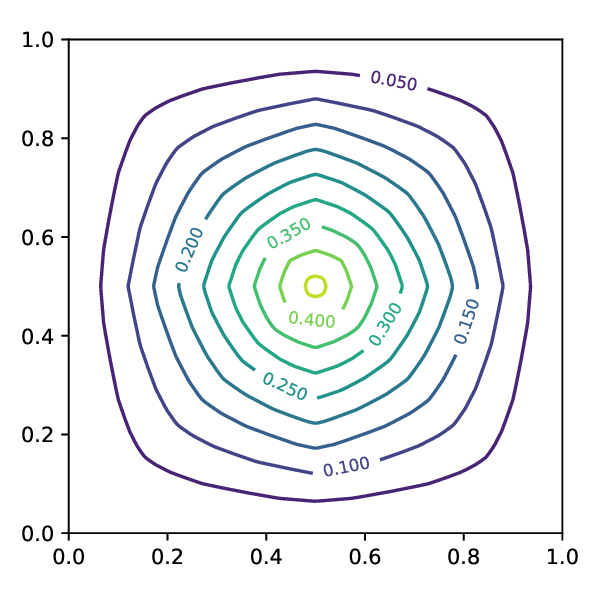}
        \caption{\footnotesize $48$ stencil points}
        %\label{fig:figure2}
    \end{subfigure}
    \hfill
    \begin{subfigure}[t]{0.24\textwidth}
        \centering
        \includegraphics[width=\textwidth]{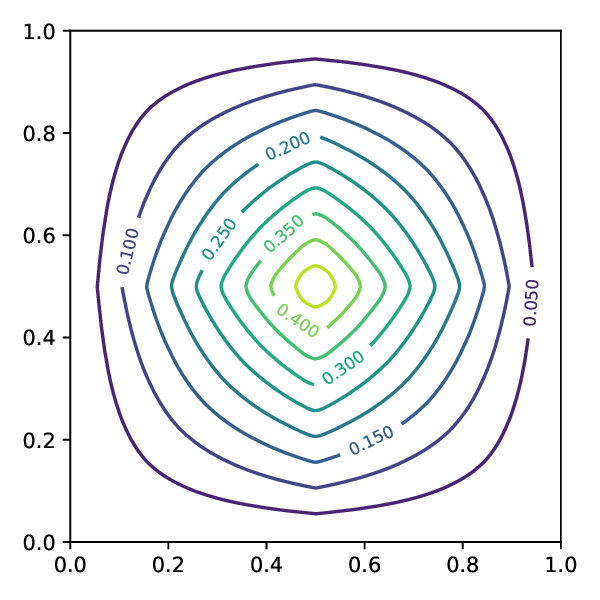}
        \caption{\footnotesize $240$ stencil points}
    \end{subfigure}
    \hfill
    \begin{subfigure}[t]{0.24\textwidth}
        \centering
        \includegraphics[width=\textwidth]{images/tukey-square-sol.eps}
        \caption{\footnotesize Analytical solution}
        %\label{fig:figure}
    \end{subfigure}
    \hfill
    \caption{Computed results on $512\times 512$ grids and analytical solutions of the Tukey depth eikonal equation where $\rho$ represents a uniform distribution on $[0,1]^2$. The computed solutions become closer to the analytical solution as the number of stencil points increases, i.e., as $d\theta$ approaches 0.}
    \label{fig:tukey-square-stencils}
\end{figure}

To demonstrate the robustness of the statistical depth provided by Tukey depth, we compare the solutions of the eikonal equation:
\begin{equation*}
\left\{
    \begin{aligned}
        |\nabla u(x)| &= \rho(x), && x \in \Omega\\
        u(x) &= 0, && x \in \partial \Omega
    \end{aligned}
\right.
\end{equation*}
 and of the Tukey depth eikonal equation:
 \begin{equation*}
\left\{
    \begin{aligned}
        |\nabla u(x)| = \int_{(y-x) \cdot \nabla u(x) = 0} \rho(y) \, dS(y), \quad x \in \Omega.
    \end{aligned}
\right.
\end{equation*}
Here, $\rho=1$ on some subset $E \subset \Omega$, and $\rho=0$ otherwise. The shape of $E$ is visually represented in Figure~\ref{fig:eikonal-comparison-a}, where it can be observed that $E$ assumes the form of a circle with a minor perturbation within its interior, i.e., $\rho=0$ on a small area in the interior. Figure~\ref{fig:eikonal-comparison} illustrates the computed solutions of these two equations on $512\times 512$ grids. It is important to note that because $\rho$ is not strictly positive, the solution to the eikonal equation is not unique, and the computed solution may depend on the chosen initialization of $u^{0}$. Figure~\ref{fig:eikonal-comparison} presents the computed solution with the initialization $u^0\equiv 0$. It is evident from the figures that the solution to the eikonal equation is significantly influenced by the small perturbation. On the other hand, the solution of the Tukey depth eikonal equation is unique, even when $\rho$ vanishes (see \cite{molina2022tukey}), and hence the solution remains relatively unperturbed by it.  Consequently, the results affirm the robustness of the Tukey depth eikonal equation in the presence of perturbations. %\red I'm not sure I understand, is the perturbation in the indicator function $f$, or in the boundary conditions? Let's discuss when we meet to make sure it's a fair comparison. It looks very appealing.\nc

\begin{figure}[ht!]
    \centering
    \begin{subfigure}[t]{0.31\textwidth}
        \centering
        \includegraphics[width=\textwidth]{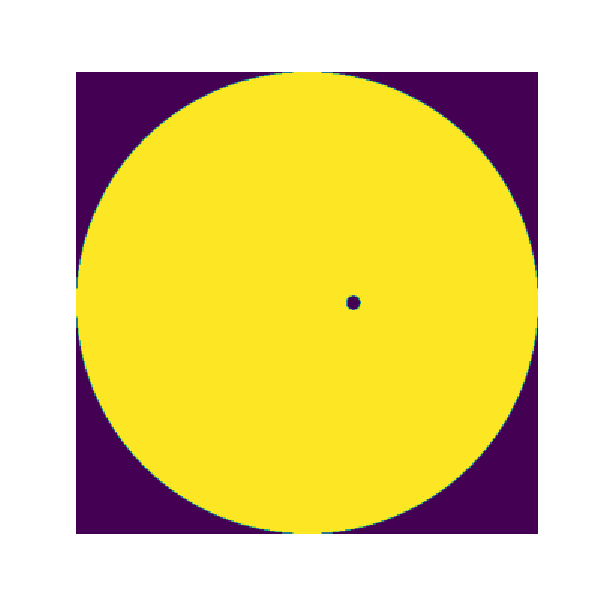}
        \caption{ $\rho$}
        \label{fig:eikonal-comparison-a}
    \end{subfigure}
    \hfill
    \begin{subfigure}[t]{0.31\textwidth}
        \centering
        \includegraphics[width=\textwidth]{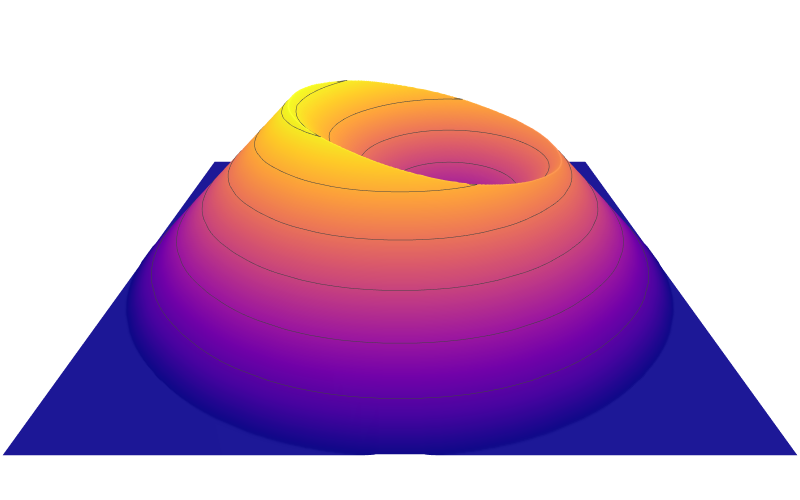}
        \caption{ Eikonal equation}
        %\label{fig:figure2}
    \end{subfigure}
    \hfill
    \begin{subfigure}[t]{0.31\textwidth}
        \centering
        \includegraphics[width=\textwidth]{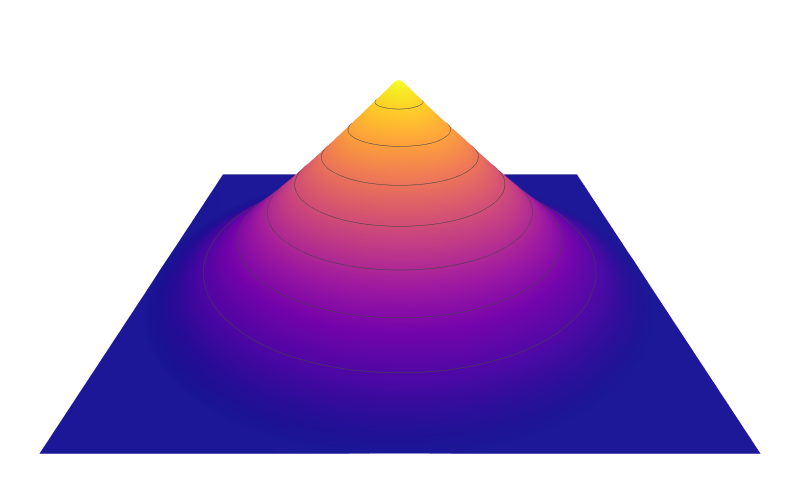}
        \caption{ Tukey depth}
        %\label{fig:figure2}
    \end{subfigure}
    \hfill
    \caption{(a) illustrates the density $\rho$ with a small perturbation in the interior, where bright pixels represent a value of $1$ and dark pixels represent a value of $0$.
    (b) and (c) depict the computed solutions of the eikonal equation and the Tukey depth eikonal equation, respectively, using the density $\rho$ on $512\times 512$ grids.}
    \label{fig:eikonal-comparison}
\end{figure}

% In order to demonstrate the algorithm's proficiency in accurately computing solutions for the Tukey eikonal equation based on nonconvex data densities, two supplementary experiments were carried out using nonconvex $\rho$ on $128 \times 128$ grid with the size of stencils $k=16$. The results of the experiments are illustrated in Figure~\ref{fig:tukey-extra-2d}. The figure presents two computed solutions where $\rho=1$ on two spherical regions and $\rho=0$ otherwise, and another case where $\rho=1$ on four spherical regions and $\rho=0$ otherwise. Both results show the center corresponds to the largest value of the solution, which aligns with the true solution of the Tukey depth.

% \begin{figure}[ht!]
%     \centering
%     \begin{subfigure}[t]{0.49\textwidth}
%         \centering
%         \includegraphics[width=\textwidth]{images/tukey-2balls.eps}
%         \caption{\footnotesize Computed solution on two balls}
%         %\label{fig:figure}
%     \end{subfigure}
%     \hfill
%     \begin{subfigure}[t]{0.49\textwidth}
%         \centering
%         \includegraphics[width=\textwidth]{images/tukey-4balls.eps}
%         \caption{\footnotesize Computed solution on four balls}
%         %\label{fig:figure2}
%     \end{subfigure}
%     \hfill
%     \caption{Computed solutions of Tukey depth eikonal equation  with nonconvex data densities $\rho$. The supports of $\rho$ are (a) two balls and (b) four balls.}
%     \label{fig:tukey-extra-2d}
% \end{figure}

Next, we solve the Tukey depth eikonal equation on unstructured point clouds $\mathcal{X}_{n}$ that are independent and identically distributed sampled from the uniform distribution $\rho \in \mathcal{P}(\Omega)$ on a square in $\mathbb{R}^2$ and on a ball in $\mathbb{R}^2$ and $\mathbb{R}^3$. We construct a $k=30$ Euclidean distance nearest neighbors graph from $\mathcal{X}_{n}$ to define $N_h(x)$ for each $x\in\mathcal{X}_{n}$. The imposed boundary condition is a Dirichlet boundary condition such that
\[
    u(x) = 0, \quad x \in \partial_\epsilon \mathcal{X}_{n}
\]
where $\partial_\epsilon \mathcal{X}_{n} := \{ x \in \mathcal{X}_{n}: d(x, \partial \Omega) < \epsilon \}$ and $d(x,y)=|x-y|$.

Note that there are various density estimation techniques that can be used to approximate the nonlocal integral function $(x, p) \mapsto \int_{(y - x) \cdot p = 0} \rho(y) , dS(y)$. However, in this experiment, we analytically compute the function for demonstration purposes. The quantitative results, showing the error between computed solutions and analytical solutions in $\mathbb{R}^2$ and $\mathbb{R}^3$, are displayed in Table~\ref{tab:tukey-2-result} and visualized in Figure~\ref{fig:tukey-2-result} and Figure~\ref{fig:tukey-3d-result}. The error is computed through the $L^1$ norm between the computed solutions $u_{c}:\mathcal{X}_n\rightarrow \mathbb{R}$ and analytical solutions~$u_{a}:\mathcal{X}_n\rightarrow \mathbb{R}$:
\begin{align*}
\text{Error} = \|u_{c} - u_{a}\|_{L^1(\mathcal{X}_n)} = \frac{1}{n} \sum_{x \in \mathcal{X}_n} |u_{c}(x) - u_{a}(x)|.
\end{align*}

\begin{table}[th!]
\centering
\begin{tabular}{|c|c|c|c|c|c|c|}
\hline
\multirow{2}{*}{$\rho$}  & \multicolumn{2}{c|}{$n=1000$} & \multicolumn{2}{c|}{$n=3000$} & \multicolumn{2}{c|}{$n=10000$}\\ \cline{2-3} \cline{4-5} \cline{6-7} 
  & Time & Error & Time & Error & Time & Error\\
\hline
Square                     & 0.28s & $5.74\times 10^{-4}$ & 1.15s & $2.24 \times 10^{-4}$ & 5.44s & $1.59 \times 10^{-4}$ \\
Circle (2D)                 & 0.39s & $2.24\times 10^{-3}$ & 1.20s & $9.15 \times 10^{-4}$ & 5.10s & $8.38\times 10^{-4}$\\
Circle (3D)                 & 0.41s & $5.76\times 10^{-4}$ & 1.07s & $3.31 \times 10^{-4}$ & 4.60s & $2.76\times 10^{-4}$\\
\hline
\end{tabular}
\caption{Computation time and errors for Tukey depth eikonal equation on 2D point clouds.}
\label{tab:tukey-2-result}
\end{table}

\begin{figure}[ht!]
    \centering
    \begin{subfigure}[b]{0.48\textwidth}
        \centering
        \includegraphics[width=\textwidth]{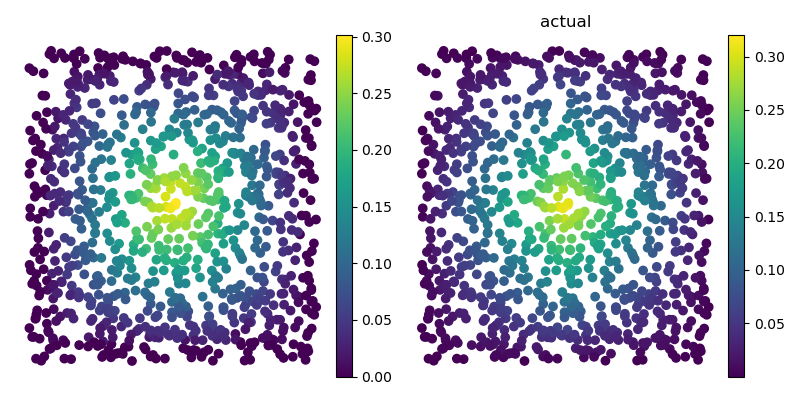}
        \caption{$n = 1000$ on a square density}
    \end{subfigure}
    \hfill
    \begin{subfigure}[b]{0.48\textwidth}
        \centering
        \includegraphics[width=\textwidth]{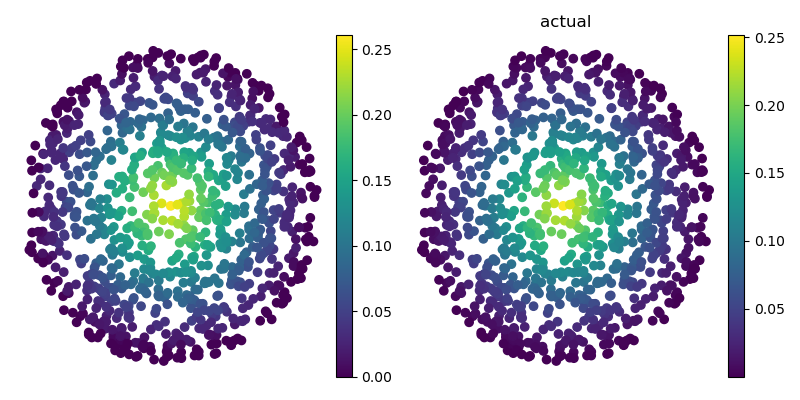}
        \caption{$n = 1000$ on a circle density}
    \end{subfigure}
    \begin{subfigure}[b]{0.48\textwidth}
        \centering
        \includegraphics[width=\textwidth]{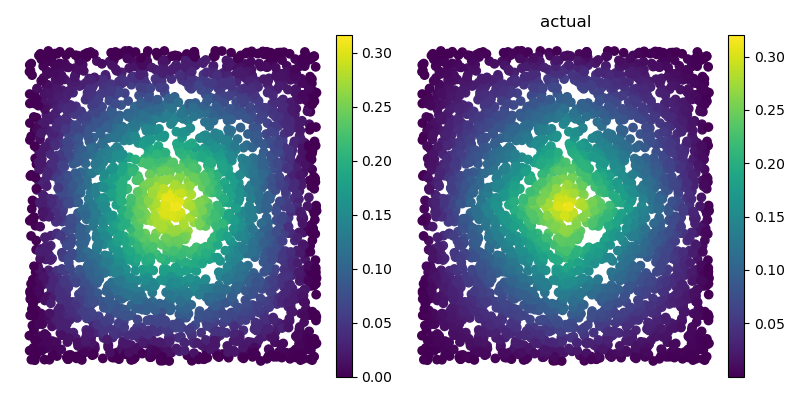}
        \caption{$n = 3000$ on a square density}
    \end{subfigure}
    \hfill
    \begin{subfigure}[b]{0.48\textwidth}
        \centering
        \includegraphics[width=\textwidth]{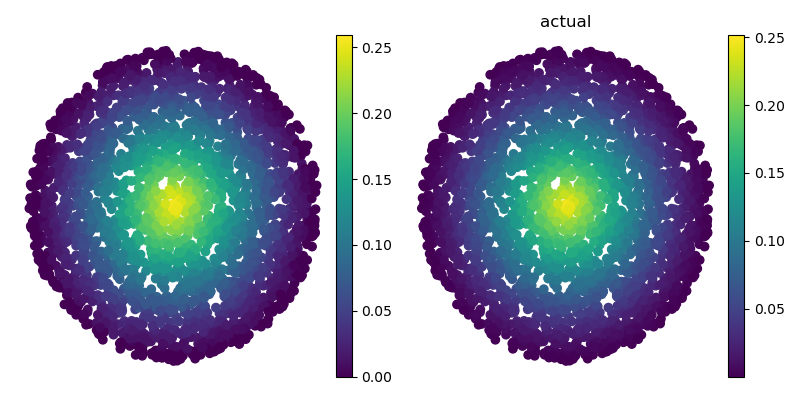}
        \caption{$n = 3000$ on a circle density}
    \end{subfigure}
    \begin{subfigure}[b]{0.48\textwidth}
        \centering
        \includegraphics[width=\textwidth]{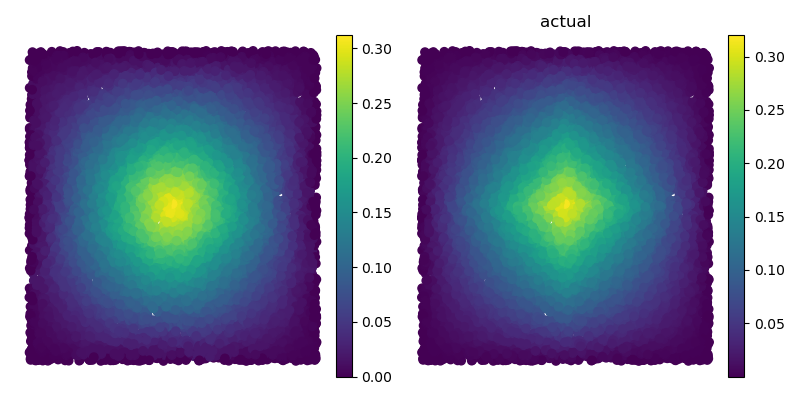}
        \caption{$n = 10000$ on a square density}
    \end{subfigure}
    \hfill
    \begin{subfigure}[b]{0.48\textwidth}
        \centering
        \includegraphics[width=\textwidth]{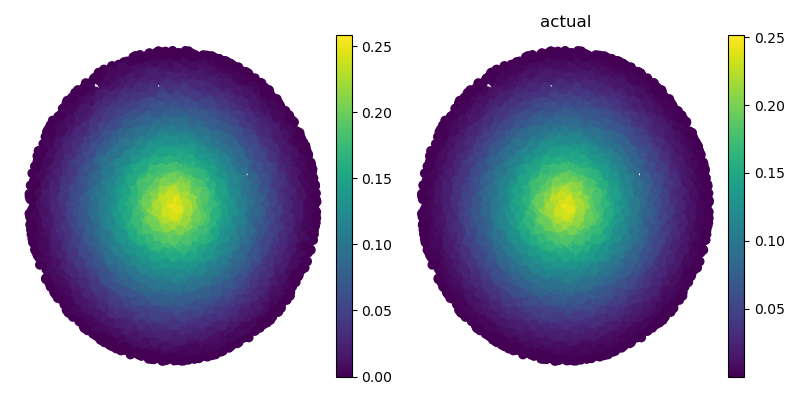}
        \caption{$n = 10000$ on a circle density}
    \end{subfigure}
    \hfill
    \caption{Computed solutions and analytical solutions of Tukey depth eikonal equation on point clouds in $\mathbb{R}^2$. Each subplot (a)-(f) displays the computed solution on the left and the analytical solution on the right, for varying numbers of points and densities.}
    \label{fig:tukey-2-result}
\end{figure}

\begin{figure}[ht!]
    \centering
    \begin{subfigure}[b]{0.32\textwidth}
        \centering
        \includegraphics[width=\textwidth]{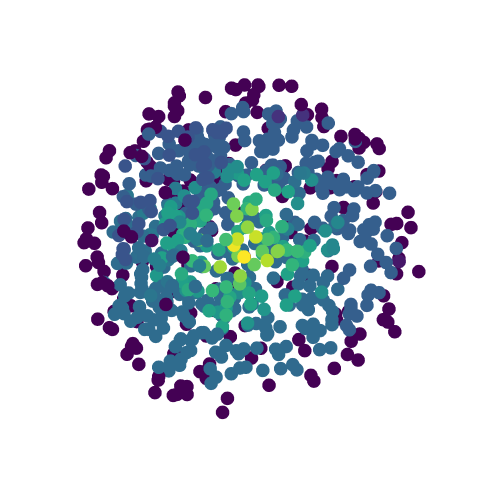}
        \caption{$n = 1000$ on a sphere}
    \end{subfigure}
    \hfill
    \begin{subfigure}[b]{0.32\textwidth}
        \centering
        \includegraphics[width=\textwidth]{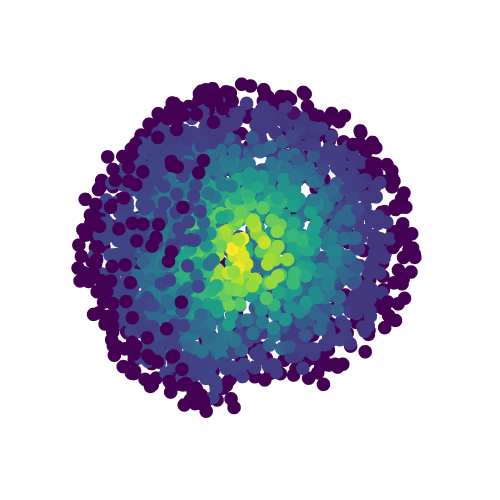}
        \caption{$n = 3000$ on a sphere}
    \end{subfigure}
    \begin{subfigure}[b]{0.32\textwidth}
        \centering
        \includegraphics[width=\textwidth]{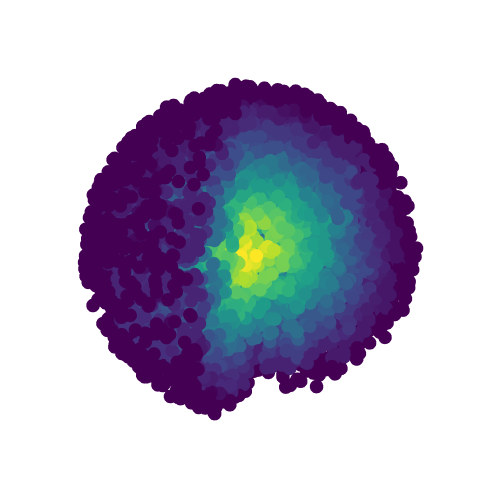}
        \caption{$n = 10000$ on a sphere}
    \end{subfigure}
    \hfill
    \caption{Computed solutions of Tukey depth eikonal equation on point clouds sampled from the uniform distribution on a sphere in $\mathbb{R}^3$. The figures display the segmented 3D sphere for clear visualization.}
    \label{fig:tukey-3d-result}
\end{figure}

% In the last experiment, we use the same settings as in the second experiment of this section and compute solutions on a sphere density in $\mathbb{R}^d$. We consider various dimensions $d = 3, 10, 100$ to demonstrate the algorithm works with high-dimensional settings. In high-dimensional setting, one requires a large amount of data points to have a reasonable spatial resolution $h$. Otherwise, the spatial resolution is too large, and therefore the scheme cannot accurately approximate the true solution of the eikonal equations. Thus, in this experiment, we present graphs that show the relationship between $|x|^2$ and the values of the computed solutions at $x$ for each $x \in \mathcal{X}_{n}$. The graphs show that the computed solutions approximate the correct ordering of the data points, i.e. values of points near boundary are small and those near the center of the sphere are large. Figure shows the computed solution on $\mathbb{R}^3$.

\subsection{Applications to high-dimensional datasets}\label{subsec:mnist}

In this set of experiments, we solve the Tukey depth eikonal equation on high-dimensional datasets. We consider the MNIST~\cite{lecun1998gradient} and Fashion-MNIST~\cite{xiao2017fashion} datasets. The MNIST dataset consists of $28\times 28$ grayscale images of handwritten digits from $0$ to $9$, while Fashion-MNIST consists of $28\times 28$ grayscale images of ten classes of clothing such that shoes, t-shirts, and so on. 

Let $\mathcal{X}_{n} \subset \mathbb{R}^{784}$ be a point cloud containing $4000$ images of a single digit ($0,\cdots,9$) from MNIST dataset. Thus, $\mathcal{X}_{n}$ is an empirical distribution of a data density of a given digit from the MNIST dataset. We then construct $k=30$ Euclidean distance nearest neighbors graph from $\mathcal{X}_{n}$, which defines the set of neighbors $N_h(x)$ for each $x\in\mathcal{X}_{n}$. Since this is a high-dimensional problem, computing an integral on the hyperplane of $\mathbb{R}^{784}$ is a challenging task.  In this experiment, we approximate the nonlocal integral term by
\[
    \int_{(y-x)\cdot p = 0} \rho(y) \, dS(y) \approx \int_{(y-x) \cdot p = 0} \rho(y) \mathcal{N}_\sigma(|x-y|)\, dS(y)
\]
where $\mathcal{N}_\sigma$ is a normal distribution with a variance $\sigma$ and a mean $0$. We compute this integral term using Monte-Carlo simulation 
\[
    \int_{(y-x) \cdot p = 0} \rho(y) \mathcal{N}_\sigma(|x-y|)\, dS(y) \approx \frac{1}{N} \sum^N_{i=1} \rho(y_i)
\]
where $y_i$ are samplings from a normal distribution on a hyperplane $\{y: (y-x) \cdot p=0\}$. In the expression, $\rho(y_i)$ is computed by a kernel density estimation such that
\[
    \rho(y_i) \approx \frac{1}{M} \sum^M_{j=1} \mathcal{N}_{r}(x_j - y_i).
\]
The same Dirichlet boundary condition of a point cloud is used as in the preceding experiment. 

In the high-dimensional setting, the spatial resolution $h$ is very large, since the distance between points grows exponentially with dimension (put another way, to keep $h$ fixed as $d\to \infty$ would require an exponentially growing number of points, as we encounter the \emph{curse of dimensionality}). Thus, we do not expect to obtain a highly accurate approximation of the true solution. Furthermore, we do not have access to the exact solution anyway, so we cannot check the accuracy. Instead, in the present experiments we visualize the images with the highest and lowest computed depth values (i.e., the deepest and shallowest points) in order to demonstrate the algorithm's ability to approximate a reasonable notion of data depth. 

We repeat the experiment for each digit from $0$ to $9$ in MNIST dataset and for each class of clothings in Fashion-MNIST dataset. The results are displayed in Figure~\ref{fig:mnist-result} (MNIST) and Figure~\ref{fig:fmnist-result} (Fashion-MNIST). Each figure in Figure~\ref{fig:mnist-result} shows $16$ highest points from (a) the computed solutions of Tukey depth eikonal equation, (b) the distance function $\dist(x, \partial \Omega)$ from the eikonal equation,
(c) $16$ boundary points in $\partial_\epsilon \mathcal{X}_{n}$, and (d) $16$ random points from $\mathcal{X}_{n}$. The boundary points were computed using the method in \cite{calder2022boundary}. The highest points from the computed solutions of Tukey depth eikonal equation correspond to median points of the datasets. When comparing (a) with other results, we can see the median points from Tukey depth show the most consistent shapes of the digits. Similarly, Figure~\ref{fig:fmnist-result} (a) shows the most consistent results of all. Thus, even though our numerical method may not accuractely approximate the true solution of the Tukey depth PDE in a high dimenional setting, the method is computationally efficient and produces reasonable results for data depth.

%I think this is too ``pie in the sky'' to pose as a future problem
%\begin{remark}
%    It should be noted that assuming Dirichlet boundary conditions of the PDE imposes an assumption that the data density domain for each digit satisfies some degree of convexity, which may not be applicable in most data science problems. We believe that the scheme can be improved to solve the eikonal equation without imposing a Dirichlet boundary condition. This would require us to develop a monotone scheme that operates outside the domain of quasiconcavity, for instance:
%    \[
%        S_h(u,u(x),x) = \begin{cases}
%            \max_{p \in P^-_h[u](x)} F_h(p,u,u(x),x) & \text{if } P^-_h[u](x) \neq \emptyset\\
%            G_h(u,u(x),x) & \text{otherwise}
%        \end{cases}
%    \]
%    where $G_h$ is a monotone function. This monotone scheme could be utilized without identifying boundary points from the point clouds and could be applied to compute not just quasiconcave solutions but general forms of viscosity solutions as well. Developing such a general monotone scheme is what we intend to investigate in future research.
%
%\end{remark}

\begin{figure}[ht!]
    \centering
    \begin{subfigure}[b]{0.49\textwidth}
        \centering
        \includegraphics[width=\textwidth]{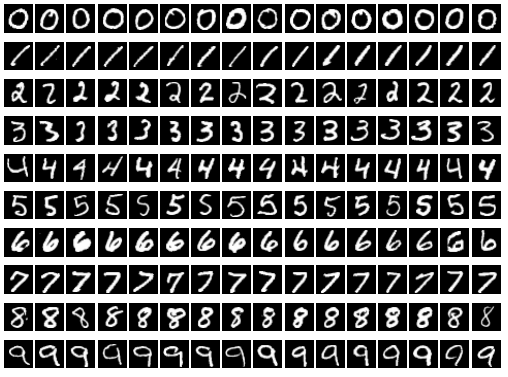}
        \caption{Tukey depth median}
        \label{fig:mnist-1}
    \end{subfigure}
    \hfill
    \begin{subfigure}[b]{0.49\textwidth}
        \centering
        \includegraphics[width=\textwidth]{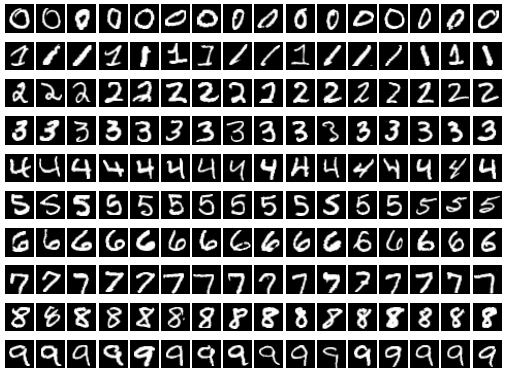}
        \caption{Eikonal median}
        \label{fig:mnist-4}
    \end{subfigure}
    \hfill

    \medskip

    \begin{subfigure}[b]{0.49\textwidth}
        \centering
        \includegraphics[width=\textwidth]{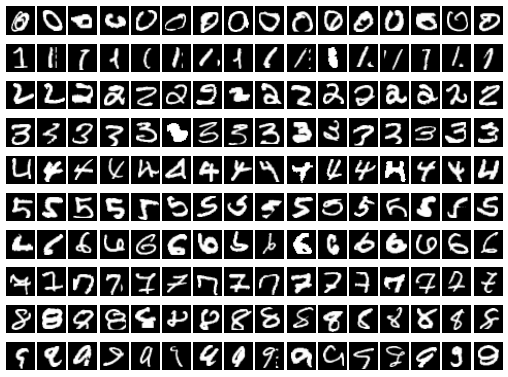}
        \caption{Boundary images}
        \label{fig:mnist-3}
    \end{subfigure}
    \hfill
    \begin{subfigure}[b]{0.49\textwidth}
        \centering
        \includegraphics[width=\textwidth]{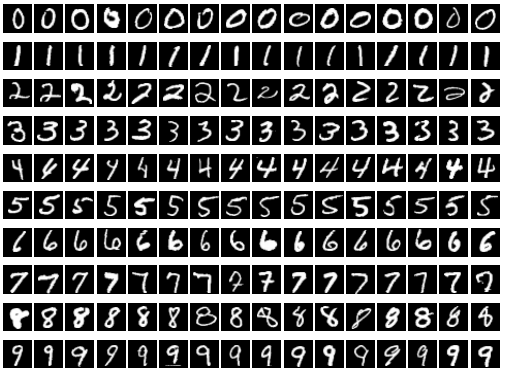}
        \caption{Random images}
        \label{fig:mnist-2}
    \end{subfigure}
    \hfill
    \caption{Median images from MNIST dataset.}
    \label{fig:mnist-result}
\end{figure}

\begin{figure}[ht!]
    \centering
    \begin{subfigure}[b]{0.49\textwidth}
        \centering
        \includegraphics[width=\textwidth]{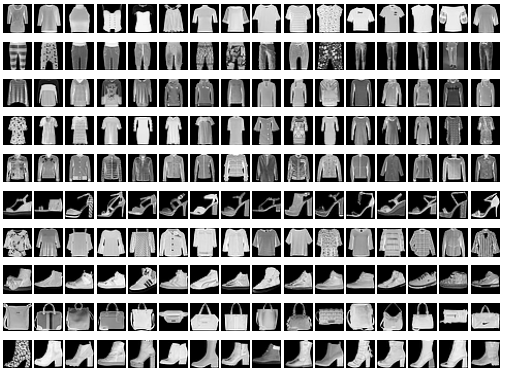}
        \caption{Tukey depth median}
        \label{fig:fmnist-1}
    \end{subfigure}
    \hfill
    \begin{subfigure}[b]{0.49\textwidth}
        \centering
        \includegraphics[width=\textwidth]{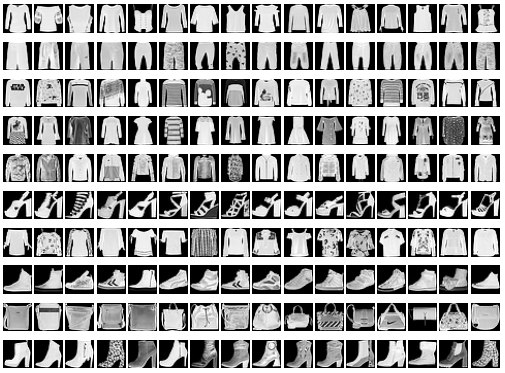}
        \caption{Eikonal median}
        \label{fig:fmnist-4}
    \end{subfigure}
    \hfill

    \medskip

    \begin{subfigure}[b]{0.49\textwidth}
        \centering
        \includegraphics[width=\textwidth]{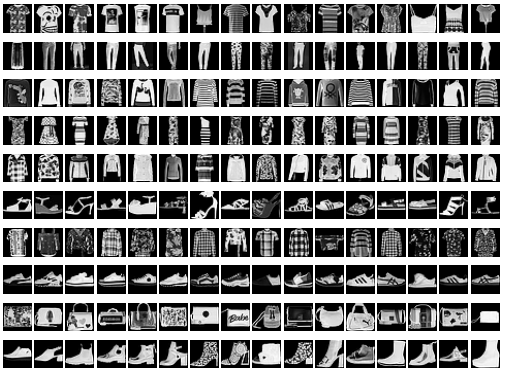}
        \caption{Boundary images}
        \label{fig:fmnist-3}
    \end{subfigure}
    \hfill
    \begin{subfigure}[b]{0.49\textwidth}
        \centering
        \includegraphics[width=\textwidth]{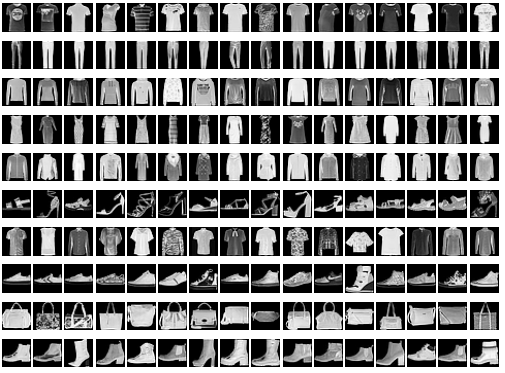}
        \caption{Random images}
        \label{fig:fmnist-2}
    \end{subfigure}
    \hfill
    \caption{Median images from MNIST dataset.}
    \label{fig:fmnist-result}
\end{figure}

\section{Conclusion}

In this paper, we developed a new monotone finite difference scheme for solving Hamilton-Jacobi equations with quasiconcave solutions. The method is based on a novel set-valued monotone discretization for the direction of the gradient. We proved that the method converges to the viscosity solution of the underlying Hamilton-Jacobi equation, and presented a series of numerical results on various types of curvature motion in $d=2$ and $d=3$ dimensions, as well as applications to computing the Tukey depth on high dimensional image datasets including MNIST and FashionMNIST. Future work will focus on expanding the methods to equations that do not enjoy the quasiconcavity property.

\section*{Acknowledgments}
The authors thank the Institute for Mathematics and its Applications (IMA). JC acknowledges funding from NSF grant DMS:1944925, the Alfred P.~Sloan foundation, a McKnight Presidential Fellowship, and the Albert and Dorothy Marden Professorship. WL acknowledges funding from the National Institute of Standards and Technology (NIST) under award number 70NANB22H021.

\bibliography{ref}

\begin{thebibliography}{10}

\bibitem{allen1979microscopic}
S.~M. Allen and J.~W. Cahn.
\newblock A microscopic theory for antiphase boundary motion and its
  application to antiphase domain coarsening.
\newblock {\em Acta Metallurgica}, 27(6):1085--1095, 1979.

\bibitem{alvarez1993axioms}
L.~Alvarez, F.~Guichard, P.-L. Lions, and J.-M. Morel.
\newblock Axioms and fundamental equations of image processing.
\newblock {\em Archive for Rational Mechanics and Analysis}, 123(3):199--257,
  1993.

\bibitem{bardi1997optimal}
M.~Bardi, I.~C. Dolcetta, et~al.
\newblock {\em Optimal control and viscosity solutions of
  Hamilton-Jacobi-Bellman equations}, volume~12.
\newblock Springer, 1997.

\bibitem{barles1995simple}
G.~Barles and C.~Georgelin.
\newblock A simple proof of convergence for an approximation scheme for
  computing motions by mean curvature.
\newblock {\em SIAM Journal on Numerical Analysis}, 32(2):484--500, 1995.

\bibitem{barles1991convergence}
G.~Barles and P.~E. Souganidis.
\newblock Convergence of approximation schemes for fully nonlinear second order
  equations.
\newblock {\em Asymptotic analysis}, 4(3):271--283, 1991.

\bibitem{barnett1976ordering}
V.~Barnett.
\newblock The ordering of multivariate data.
\newblock {\em Journal of the Royal Statistical Society. Series A (General)},
  pages 318--355, 1976.

\bibitem{barron2013quasiconvex}
E.~Barron, R.~Goebel, and R.~Jensen.
\newblock Quasiconvex functions and nonlinear pdes.
\newblock {\em Transactions of the American Mathematical Society},
  365(8):4229--4255, 2013.

\bibitem{benamou2010two}
J.-D. Benamou, B.~D. Froese, and A.~M. Oberman.
\newblock Two numerical methods for the elliptic {M}onge-{A}mpere equation.
\newblock {\em ESAIM: Mathematical Modelling and Numerical Analysis},
  44(4):737--758, 2010.

\bibitem{bou2021hamilton}
A.~Bou-Rabee and P.~S. Morfe.
\newblock Hamilton-jacobi scaling limits of pareto peeling in 2d.
\newblock {\em arXiv preprint arXiv:2110.06016}, 2021.

\bibitem{boyd2004convex}
S.~Boyd, S.~P. Boyd, and L.~Vandenberghe.
\newblock {\em Convex optimization}.
\newblock Cambridge university press, 2004.

\bibitem{calder2018lecture}
J.~Calder.
\newblock Lecture notes on viscosity solutions.
\newblock {\em Lecture notes}, 2018.

\bibitem{calder2014}
J.~Calder, S.~Esedo\=glu, and A.~O. Hero~III.
\newblock A {H}amilton-{J}acobi equation for the continuum limit of
  non-dominated sorting.
\newblock {\em SIAM Journal on Mathematical Analysis}, 46(1):603--638, 2014.

\bibitem{calder2015PDE}
J.~Calder, S.~Esedo\=glu, and A.~O. Hero~III.
\newblock A {PDE}-based approach to non-dominated sorting.
\newblock {\em SIAM Journal on Numerical Analysis}, 53(1):82--104, 2015.

\bibitem{JMLR:v23:22-0293}
J.~Calder and M.~Ettehad.
\newblock Hamilton-jacobi equations on graphs with applications to
  semi-supervised learning and data depth.
\newblock {\em Journal of Machine Learning Research}, 23(318):1--62, 2022.

\bibitem{calder2022boundary}
J.~Calder, S.~Park, and D.~Slep{\v{c}}ev.
\newblock Boundary estimation from point clouds: Algorithms, guarantees and
  applications.
\newblock {\em Journal of Scientific Computing}, 92(2):56, 2022.

\bibitem{calder2020convex}
J.~Calder and C.~K. Smart.
\newblock {The limit shape of convex hull peeling}.
\newblock {\em Duke Mathematical Journal}, 169(11):2079 -- 2124, 2020.

\bibitem{carrizosa1996characterization}
E.~Carrizosa.
\newblock A characterization of halfspace depth.
\newblock {\em Journal of multivariate analysis}, 58(1):21--26, 1996.

\bibitem{chan2001active}
T.~F. Chan and L.~A. Vese.
\newblock Active contours without edges.
\newblock {\em IEEE Transactions on image processing}, 10(2):266--277, 2001.

\bibitem{chang1996level}
Y.-C. Chang, T.~Hou, B.~Merriman, and S.~Osher.
\newblock A level set formulation of {E}ulerian interface capturing methods for
  incompressible fluid flows.
\newblock {\em Journal of computational Physics}, 124(2):449--464, 1996.

\bibitem{chepoi2010pareto}
V.~Chepoi, K.~Nouioua, E.~Thiel, and Y.~Vaxes.
\newblock Pareto envelopes in simple polygons.
\newblock {\em International Journal of Computational Geometry \&
  Applications}, 20(06):707--721, 2010.

\bibitem{chernozhukov2017monge}
V.~Chernozhukov, A.~Galichon, M.~Hallin, and M.~Henry.
\newblock {Monge–Kantorovich depth, quantiles, ranks and signs}.
\newblock {\em The Annals of Statistics}, 45(1):223 -- 256, 2017.

\bibitem{cook2022rates}
B.~Cook and J.~Calder.
\newblock Rates of convergence for the continuum limit of nondominated sorting.
\newblock {\em SIAM Journal on Mathematical Analysis}, 54(1):872--911, 2022.

\bibitem{crandall1984some}
M.~G. Crandall, L.~C. Evans, and P.-L. Lions.
\newblock Some properties of viscosity solutions of {H}amilton-{J}acobi
  equations.
\newblock {\em Transactions of the American Mathematical Society},
  282(2):487--502, 1984.

\bibitem{crandall1992user}
M.~G. Crandall, H.~Ishii, and P.-L. Lions.
\newblock User’s guide to viscosity solutions of second order partial
  differential equations.
\newblock {\em Bulletin of the American Mathematical Society}, 27(1):1--67,
  1992.

\bibitem{crandall1983viscosity}
M.~G. Crandall and P.-L. Lions.
\newblock Viscosity solutions of hamilton-jacobi equations.
\newblock {\em Transactions of the American mathematical society},
  277(1):1--42, 1983.

\bibitem{elsey2011large}
M.~Elsey, S.~Esedoglu, and P.~Smereka.
\newblock Large-scale simulation of normal grain growth via diffusion-generated
  motion.
\newblock {\em Proceedings of the Royal Society A: Mathematical, Physical and
  Engineering Sciences}, 467(2126):381--401, 2011.

\bibitem{elsey2018threshold}
M.~Elsey and S.~Esedoḡlu.
\newblock Threshold dynamics for anisotropic surface energies.
\newblock {\em Mathematics of Computation}, 87(312):1721--1756, 2018.

\bibitem{esedog2010diffusion}
S.~Esedog, S.~Ruuth, R.~Tsai, et~al.
\newblock Diffusion generated motion using signed distance functions.
\newblock {\em Journal of Computational Physics}, 229(4):1017--1042, 2010.

\bibitem{evans1993convergence}
L.~C. Evans.
\newblock Convergence of an algorithm for mean curvature motion.
\newblock {\em Indiana University Mathematics Journal}, 42(2):533--557, 1993.

\bibitem{evans1991motion}
L.~C. Evans, J.~Spruck, et~al.
\newblock Motion of level sets by mean curvature {I}.
\newblock {\em Journal of Differential Geometry}, 33(3):635--681, 1991.

\bibitem{froese2011convergent}
B.~D. Froese and A.~M. Oberman.
\newblock Convergent finite difference solvers for viscosity solutions of the
  elliptic {M}onge-{A}mpere equation in dimensions two and higher.
\newblock {\em SIAM Journal on Numerical Analysis}, 49(4):1692--1714, 2011.

\bibitem{froese2013convergent}
B.~D. Froese and A.~M. Oberman.
\newblock Convergent filtered schemes for the monge--ampe{\'e}re partial
  differential equation.
\newblock {\em SIAM Journal on Numerical Analysis}, 51(1):423--444, 2013.

\bibitem{lecun1998gradient}
Y.~LeCun, L.~Bottou, Y.~Bengio, and P.~Haffner.
\newblock Gradient-based learning applied to document recognition.
\newblock {\em Proceedings of the IEEE}, 86(11):2278--2324, 1998.

\bibitem{merriman1992diffusion}
B.~Merriman, J.~K. Bence, and S.~Osher.
\newblock {\em Diffusion generated motion by mean curvature}.
\newblock Department of Mathematics, University of California, Los Angeles,
  1992.

\bibitem{molina2022eikonal}
M.~Molina-Fructuoso and R.~Murray.
\newblock Eikonal depth: an optimal control approach to statistical depths.
\newblock {\em arXiv preprint arXiv:2201.05274}, 2022.

\bibitem{molina2022tukey}
M.~Molina-Fructuoso and R.~Murray.
\newblock Tukey depths and hamilton--jacobi differential equations.
\newblock {\em SIAM Journal on Mathematics of Data Science}, 4(2):604--633,
  2022.

\bibitem{mullins1956two}
W.~W. Mullins.
\newblock Two-dimensional motion of idealized grain boundaries.
\newblock {\em Journal of Applied Physics}, 27(8):900--904, 1956.

\bibitem{mumford1989optimal}
D.~Mumford and J.~Shah.
\newblock Optimal approximations by piecewise smooth functions and associated
  variational problems.
\newblock {\em Communications on Pure and Applied Mathematics}, 42(5):577--685,
  1989.

\bibitem{oberman2004convergent}
A.~M. Oberman.
\newblock A convergent monotone difference scheme for motion of level sets by
  mean curvature.
\newblock {\em Numerische Mathematik}, 99(2):365--379, 2004.

\bibitem{oberman2006convergent}
A.~M. Oberman.
\newblock Convergent difference schemes for degenerate elliptic and parabolic
  equations: Hamilton--jacobi equations and free boundary problems.
\newblock {\em SIAM Journal on Numerical Analysis}, 44(2):879--895, 2006.

\bibitem{oberman2008wide}
A.~M. Oberman.
\newblock Wide stencil finite difference schemes for the elliptic
  {M}onge-{A}mpere equation and functions of the eigenvalues of the {H}essian.
\newblock {\em Discrete Contin. Dyn. Syst. Ser. B}, 10(1):221--238, 2008.

\bibitem{oberman2018numerical}
A.~M. Oberman and T.~Salvador.
\newblock Numerical methods for motion of level sets by affine curvature.
\newblock {\em IMA Journal of Numerical Analysis}, 38(4):1735--1767, 2018.

\bibitem{osher1988fronts}
S.~Osher and J.~A. Sethian.
\newblock Fronts propagating with curvature-dependent speed: algorithms based
  on {H}amilton-{J}acobi formulations.
\newblock {\em Journal of Computational Physics}, 79(1):12--49, 1988.

\bibitem{sethian1996fast}
J.~A. Sethian.
\newblock A fast marching level set method for monotonically advancing fronts.
\newblock {\em Proceedings of the National Academy of Sciences},
  93(4):1591--1595, 1996.

\bibitem{sethian1999fast}
J.~A. Sethian.
\newblock Fast marching methods.
\newblock {\em SIAM review}, 41(2):199--235, 1999.

\bibitem{small1997multidimensional}
C.~G. Small.
\newblock Multidimensional medians arising from geodesics on graphs.
\newblock {\em The Annals of Statistics}, pages 478--494, 1997.

\bibitem{soner2003stochastic}
H.~M. Soner and N.~Touzi.
\newblock A stochastic representation for mean curvature type geometric flows.
\newblock {\em Annals of probability}, pages 1145--1165, 2003.

\bibitem{sussman2000coupled}
M.~Sussman and E.~G. Puckett.
\newblock A coupled level set and volume-of-fluid method for computing 3{D} and
  axisymmetric incompressible two-phase flows.
\newblock {\em Journal of Computational Physics}, 162(2):301--337, 2000.

\bibitem{tukey1975mathematics}
J.~W. Tukey.
\newblock Mathematics and the picturing of data.
\newblock In {\em Proceedings of the International Congress of Mathematicians,
  Vancouver, 1975}, volume~2, pages 523--531, 1975.

\bibitem{xiao2017fashion}
H.~Xiao, K.~Rasul, and R.~Vollgraf.
\newblock Fashion-mnist: a novel image dataset for benchmarking machine
  learning algorithms.
\newblock {\em arXiv preprint arXiv:1708.07747}, 2017.

\bibitem{zhao2005fast}
H.~Zhao.
\newblock A fast sweeping method for eikonal equations.
\newblock {\em Mathematics of computation}, 74(250):603--627, 2005.

\end{thebibliography}
\bibliographystyle{abbrv}
\end{document}